\documentclass[%
  12pt, % font size
  noamsfonts, % not use ams fonts as already used newtx fonts
]{amsart}
  \usepackage{geometry} % Load early to set page dimensions
  \usepackage[T1]{fontenc} % T1 font
  \usepackage[varg]{newtx} % Times New Roman font with math support; % varg = distinctive g,v,w,y
  \usepackage[
      scr=boondox, 
      cal=euler, 
      bb=ncmbbr
  ]{mathalpha}
  \usepackage{mathtools} % Loads amsmath automatically
  \usepackage{mathcommand} % define math macros
  \usepackage{xcolor} % Load early to avoid clashes

  \usepackage[extdef]{delimset} % delimset handles sizing delimiters and paired delimiters with ease
  \usepackage{interval} %Format mathematical intervals, ensuring proper spacing
  \usepackage{fixdif} % differential operators
  \usepackage{xfrac} % slanted fractions
  \usepackage{leftindex} % better left superscript 
  \usepackage{accents} % Multiple mathematical accents
  \usepackage{aligned-overset} % Custom overset/underset with alignment

  \usepackage{tikz} 
  \usepackage{quiver} % commutative diagrams with tikz-cd, https://q.uiver.app/
  \usepackage{simpler-wick}
  \usepackage{qworld} %  drawing string diagrams in monoidal categories and quantum theory

  \usepackage[shortlabels,inline]{enumitem} % Enumerate and itemize lists
  \usepackage{aliascnt} % alias of counters
  \usepackage{keytheorems} % An expl3-implementation of a key-value interface to amsthm
  \usepackage{caption} % Caption

  \usepackage[pagebackref]{hyperref}
  \usepackage{zref-clever} % MUST load after hyperref
  \numberwithin{equation}{section}
  \mathtoolsset{showonlyrefs}
  \let\realItem\item % save a copy of the original item
  \makeatletter
  \NewDocumentCommand\myItem{ o }{%
    \IfNoValueTF{#1}%
        {\realItem}% add an item
        {\realItem[#1{\MakeLinkTarget[item]{}}]\def\@currentlabel{#1}}% add an item and update label
  }
  \makeatother

  \setlist[description,1]{% 1st level of description
    labelindent = 1.5\parindent,
    style=multiline
  }

  \zcsetup{%
      cap,
      abbrev,
      nameinlink=false,
  }
  \hypersetup{%
      bookmarksnumbered=true,
      colorlinks=true,
      linkcolor=blue,
      citecolor=cyan,
      pdfstartview=FitBH,
  }
  \renewcommand*{\backrefalt}[4]{%
      \ifcase #1 %
        No citations.%
      \or
        $\uparrow$ #2%
      \else
        $\uparrow$ #2%
      \fi
  }

  \makeatletter
  \newcommand{\mathcm}[1]{%
    \mbox{\usefont{OMS}{cmsy}{m}{n}#1}% use the OMS (Old Math Symbols) encoding
  }
  \makeatother

  \providecommand{\noopsort}[1]{} % fix for bib order
  \newkeytheorem{mainthm}[%
      name=Theorem,%
      counter-format=\Alph{mainthm},%
  ]% "letter-numbered" theorems

  \newkeytheorem{%
      theorem,%
      proposition,%
      lemma,%
      corollary,%
      conjecture,%
  }[sharenumber=equation]

  \newkeytheorem{nothm}[name={Theorem},numbered=false]

  \newkeytheorem{nocoro}[name={Corollary},numbered=false]

  \newkeytheorem{%
      definition,%
      remark,%
      notation,%
      example,%
      warning,%
      construction,%
      convention,%
      recollection,%
  }[sharenumber=equation,style=remark]

  \newkeytheorem{egrmk}[style=remark,name={Example/Remark},sharenumber=equation]

  \newkeytheoremstyle{para}{%
      bodyfont=\normalfont,%
      headfont=\itshape,%
      notebraces={}{},%
      notefont={},%
      noteseparator={},
  }
  \newkeytheorem{para}[name={},numbered=false,style=para]

  \NewDocumentCommand \ListInThm { m O{\textup} O{(\roman{enumi})} O{.} }
    {
      \AtBeginEnvironment{#1}
        {
          \zcsetup{ countertype={enumi=#1} }
          \setlist[enumerate,1]
            {
              label={#2{#3}},
              ref={\csname the#1\endcsname#4{#3}}
            }%
        }
    }

  \ListInThm{mainthm}
  \ListInThm{theorem}
  \ListInThm{proposition}
  \ListInThm{lemma}
  \ListInThm{corollary}

  \NewDocumentCommand{\aststitle}{m}{
    \begin{center}
      $\ast\ast\ast$ 
      {#1}
      $\ast\ast\ast$ 
    \end{center}
  }
  \usetikzlibrary{%cd,%
      arrows.meta,% <-- Modern replacement for 'arrows'
      backgrounds,%
      calc,%
      decorations,%
      decorations.markings,% <-- Highly recommended for string/physics diagrams
      decorations.pathmorphing,%
      shapes,%
      tikzmark,%
  }
  \tikzcdset{%
      scale cd/.style = %
          {%
            every label/.append style = {scale = #1},%
            cells = {nodes = {scale = #1}}%
          }%
  }
  \NewDocumentCommand \concept {m} {\textbf{#1}}
  \NewDocumentCommand \nota {m} {{\color{orange}#1}}
  
  \NewDocumentCommand \nocolor {m} {{\color{black}#1}}
  \NewDocumentMathCommand\txand%
      { O{\quad} }{#1\text{and}#1}
  \NewDocumentMathCommand\txforall%
      { O{\quad} }{#1\text{for all}#1}
  \NewDocumentMathCommand\txforsome%
      { O{\quad} }{#1\text{for some}#1}
  \NewDocumentCommand\placeholder{}{\:\cdot\:} % modefy as needed
  \newmathcommand{\dbrack}{\delim{\lbrack\mkern-4mu\lbrack}{\rbrack\mkern-4mu\rbrack}}
  \newmathcommand{\dbrace}{\delim{\lbrace\mkern-4mu\lbrace}{\rbrace\mkern-4mu\rbrace}}
  \newmathcommand{\dparen}{\delim{\lparen\mkern-4mu\lparen}{\rparen\mkern-4mu\rparen}}
  \newmathcommand{\dangle}{\delim{\langle\mkern-4mu\langle}{\rangle\mkern-4mu\rangle}}

  \newmathcommand{\ceil}{\delim\lceil\rceil}
  \newmathcommand{\floor}{\delim\lfloor\rfloor}

  \newmathcommand{\Set}{\brk[c]} 
  \newmathcommand{\GSet}{\brk[a]}
  \newmathcommand{\PSet}{\brk[r]}
  \newmathcommand{\BSet}{\brk[s]}
  \newmathcommand{\gens}{\brk[s]*} 
  \newmathcommand{\rels}{\brk[r]*}

  \newmathcommand{\SetCond}{\delimpair\lbrace{[i]\vert}\rbrace}
  \newmathcommand{\aglsetcond}{\delimpair\langle{[i]\vert}\rangle}
  \newmathcommand{\prnsetcond}{\delimpair\lparen{[i]\vert}\rparen}
  \newmathcommand{\brksetcond}{\delimpair\lbrack{[i]\vert}\rbrack}
  \newmathcommand{\gensPair}{\brksetcond*}
  \newmathcommand{\relsPair}{\prnsetcond*}

  \newmathcommand{\textSet}[1]{\Set*{\text{#1}}}
  \providecommand\given{\vert} % Fallback if used elsewhere

  \newmathcommand{\pairing}{\delimpair\langle{*,}\rangle} % <..,..>
  \newmathcommand{\bracket}{\delimtriple\langle\vert\vert\rangle} % <..|..|..>

  \NewDocumentCommand \fun { m e{^_} d() }
      {%
        \operatorname{#1}%
        \IfValueT{#2}{\sp{#2}}%
        \IfValueT{#3}{\sb{#3}}%
        \IfNoValueTF{#4}{}{\brk*{#4}}%
      }
  \LoopCommands{%
      NZQRCKkWM%
  }[#1]{\declaremathcommand#2{\mathbb#1}}
  
  \LoopCommands{%
      {Aut}{an}{Coker}{Cotor}{coker}{Der}{End}{Ext}{Gal}{Gr}{gl}{gr}{Ho}{Hom}{Id}{id}{Ind}{Ker}{Mor}{Ob}{Pic}{Proj}{pr}{Res}{Sq}{Spc}{Spec}{Spm}{supp}{Tor}{tr}{wt}%
  }[#1]{\DeclareMathOperator{#2}{#1}}

  \renewmathcommand\le\leqslant
  \renewmathcommand\ge\geqslant
  \declaremathcommand\unit{\mathbb{1}}
  \declaremathcommand\one{{\scriptstyle\mathfrak{I}}}
  \declaremathcommand\iu{\mathtt{i}}
  \declaremathcommand\e{\mathsf{e}}
  \newmathcommand\bPhi{\mathbf{\Phi}}
  \NewDocumentMathCommand\vect{m}{\boldsymbol{#1}}

  \declaremathcommand\Rf{\operatorname{\mathbf{R}}}
  \declaremathcommand{\H}{\operatorname{H}}
  \declaremathcommand\SG{\mathfrak{S}}
  
  \NewDocumentMathCommand\cat{m}{\operatorname{\mathsf{#1}}}
  \NewDocumentMathCommand\grp{m}{\operatorname{\mathsf{#1}}}
  \NewDocumentCommand\Mod{so}{%
      \cat{\IfNoValueTF{#2}{}{{}^{\mathtt{#2}}}\IfBooleanTF{#1}{mod}{Mod}}%
  }
  
  \NewDocumentMathCommand\opp{m}{\IfBlankTF{#1}{(\:\cdot\:)}{#1}^{\mathrm{op}}}
  \NewDocumentMathCommand\dual{m}{\IfBlankTF{#1}{(\:\cdot\:)}{#1}^{\ast}}
  \NewDocumentMathCommand\rldual{m}{\IfBlankTF{#1}{(\:\cdot\:)}{#1}^{\dagger}}
  \NewDocumentMathCommand\invo{m}{\prescript{\theta}{}{#1}}

  \DeclareFontFamily{U}{dmjhira}{}
  \DeclareFontShape{U}{dmjhira}{m}{n}{ <-> dmjhira }{}
  \DeclareRobustCommand{\yo}{\text{\usefont{U}{dmjhira}{m}{n}\symbol{"48}}}

  \declaremathcommand\O{\mathscr{O}}
  \declaremathcommand\shHom{\operatorname{\mathscr{Hom}}}
  \declaremathcommand\et{\text{\'et}}
  \declaremathcommand\Moduli{\mathcm{M}}
  \declaremathcommand\Gm{\mathbb{G}_{\mathtt{m}}}
  \declaremathcommand\Ga{\mathbb{G}_{\mathtt{a}}}
  \declaremathcommand\pt{\mathsf{p}}
  \declaremathcommand\tpt{\tilde{\mathsf{p}}}
  \declaremathcommand\qt{\mathsf{q}}

  \newmathcommand\spn{\operatorname{span}}
  \newmathcommand\ord{\operatorname{ord}}
  \newmathcommand\Image{\operatorname{Im}}
  \newmathcommand\Lie\comm
  \newmathcommand\ctensor{\mathbin{\mathop{\widehat{\otimes}}}}
  \newmathcommand\rtensor{\mathbin{\mathop{\vec{\otimes}}}}
  
  \makeatletter
    \newmathcommand{\ostar}{\mathbin{\mathpalette\make@circled\star}}
    \newcommand{\make@circled}[2]{%
      \ooalign{$\m@th#1\smallbigcirc{#1}$\cr\hidewidth$\m@th#1#2$\hidewidth\cr}%
    }
    \newcommand{\smallbigcirc}[1]{\vcenter{\hbox{\scalebox{0.7}{$\m@th#1\bigcirc$}}}}
  \makeatother

  \NewDocumentMathCommand\odv{m}{\frac{\d}{\d{#1}}}
  \NewDocumentMathCommand\pdv{m}{\frac{\partial}{\partial{#1}}}

  \declaremathcommand\vac{\mathbf{1}}
  \declaremathcommand\cch{\mathscr{c}}
  \declaremathcommand\cfv{\upomega}
  \declaremathcommand\qbar{\mathscr{q}}
  
  \declaremathcommand\hollowcolon{{}^{\circ}_{\circ}}
  \NewDocumentMathCommand\normord{m}{\mathopen{\hollowcolon}\mathinner{#1}\mathclose{\hollowcolon}}

  \NewDocumentMathCommand\vo{mm}{#1_{(#2)}}
  \NewDocumentMathCommand\lo{mm}{#1_{[#2]}}
  \NewDocumentMathCommand\VL{m}{L_{#1}}
  
  \NewDocumentMathCommand\NL{O{\bullet}}{\mathsf{N}_{\mathsf{L}}^{#1}}
  \NewDocumentMathCommand\NR{O{\bullet}}{\mathsf{N}_{\mathsf{R}}^{#1}}
  \NewDocumentMathCommand\NLR{O{\bullet}}{\mathsf{N}^{#1}}
  \NewDocumentMathCommand\cNL{O{\bullet}}{\prescript{\mathrm{c}}{}{\mathsf{N}}_{\mathsf{L}}^{#1}}
  \NewDocumentMathCommand\cNR{O{\bullet}}{\prescript{\mathrm{c}}{}{\mathsf{N}}_{\mathsf{R}}^{#1}}
  \NewDocumentMathCommand\cNLR{O{\bullet}}{\prescript{\mathrm{c}}{}{\mathsf{N}}^{#1}}

  \newcommand{\SIE}{\hyperref[def:SIE]{\textsf{SIE}}}
  \newcommand{\SIC}{\hyperref[def:SIC]{\textsf{SIC}}}

\begin{document}%%-----------------------------------------
%%%% Document Metadata 
% region
  \title{On strong identities of almost-canonically seminormed rings}
  
  {
    \author[X.~Gao]{Xu Gao}
    \address{Xu Gao \newline  \indent  School of Mathematical Sciences, Tongji University,  Shanghai, 200092, China}
    \email{gausyu@tongji.edu.cn}
  }

  {
    \author[J.~Liu]{Jianqi Liu}
    \address{Jianqi Liu \newline  
        \indent  Department of Mathematics, University of Pennsylvania,  Phil, PA 08904, USA}
    \email{jliu230@sas.upenn.edu}
  }
% endregion

%%-----------------------------------------

\begin{abstract}
  We investigate the strong identity condition (SIC) for almost-canonically seminormed rings, a class of topological graded rings that includes enveloping algebras of vertex operator algebras. This condition was introduced in the algebro-geometric theory of conformal blocks, where it governs the smoothing of nodal curves.

  To understand the representation-theoretic meaning of SIC, we develop the representation theory of almost-canonically seminormed rings, including Zhu-type algebras, induced modules, rationality conditions, tensor product compatibility, and an end formula for the mode transition algebra.
  Our main result characterizes the strong identity condition in terms of orthogonal expansions, projectivity of canonical modules, and Morita-type equivalences induced by Zhu-type algebras.

  As an application, we show that for vertex operator algebras of CFT type, the smoothing property is equivalent to the Zhu algebra inducing a Morita-type equivalence with the category of admissible modules. Consequently, the strong identity condition identifies the precise representation-theoretic obstruction to extending algebraic smoothing beyond the semisimple setting.
  We further illustrate the theory through explicit examples, including the Weyl algebra and several irrational vertex operator algebras where the strong identity condition fails.
\end{abstract}

\maketitle

%%-----------------------------------------

\tableofcontents

\section*{Introduction}
% --- Purpose of this paper and main results ---

  Conformal blocks are fundamental objects in the study of two-dimensional conformal field theories and the representation theory of vertex operator algebras (VOAs).
  They form logarithmic $\mathcal{D}$-modules on moduli spaces of curves and encode the tensor category structure of VOA modules.
  A central pillar of their theory is the ability to reconstruct conformal blocks on general curves from those on simpler ones. 
  In the analytic setting (eg. \cite{Huang97,HLZ-I,GZ-III}), this is achieved robustly by \emph{sewing} along gluing tubes. Algebraically, the counterpart is the \emph{smoothing} of nodal curves.
  To formalize it, the \emph{mode transition algebra} $\mathfrak{A}$ was introduced in \cite{DGK23}. However, this purely algebraic approach hinges entirely on the \emph{strong identity condition} (\textbf{SIC}): the existence of well-behaved diagonal identity elements in $\mathfrak{A}$. Because it requires this condition, algebraic smoothing fails for important irrational VOAs, such as the triplet algebra $\mathcal{W}(p)$, where analytic sewing succeeds.
  This raises a fundamental question: \emph{what exactly does SIC measure?}

\subsection*{Results and applications} 
  A main purpose of this paper is to determine the representation-theoretic nature of this algebraic barrier. To isolate the intrinsic algebraic phenomena, we work within the framework of \emph{almost-canonically seminormed rings} $U$, abstracting the enveloping algebra $\mathscr{U}(\mathbb{V})$ of a VOA $V$. Our main result provides a topological, homological, and categorical characterization of SIC, revealing its role in governing \emph{Morita-type equivalences}\footnote{Here and throughout the paper, ``Morita-type'' refers to equivalence phenomena realized by induction functors, in analogy with classical Morita theory.}.

  \begin{mainthm}\label{thm:IntroSIC}
    Let $U$ be an almost-canonically seminormed ring.
    Then, the \emph{strong identity condition} is equivalent to the following:
    \begin{enumerate}
    \item \textbf{Topological:} $U$ admits a \emph{strong identity expansion} (an orthogonal series expansion of $1_{U}$ satisfying specific conditions).
    \item \textbf{Homological:} The canonical modules $\mathfrak{L}^n$ and $\mathfrak{R}^n$ are projective in the category of exhaustive $U$-modules.
    \item \textbf{Categorical:} For all $n \ge 0$, the \emph{thick} induction functors establish Morita-type equivalences for the \emph{thick Zhu algebras} $\mathscr{A}_n$.
    \end{enumerate}
    Furthermore, under appropriate assumptions, these conditions are equivalent to 
    \begin{enumerate}
      \item[\textup{(iii')}] \textbf{Categorical:} the induction functors establish a Morita-type equivalence for the level-zero \emph{Zhu algebra} $\mathsf{A}_{0}$.
    \end{enumerate}
  \end{mainthm}
  The precise definitions of these conditions, along with commutative diagrams outlining their logical interdependencies, are detailed at the beginning of \zcref{part2}.

\aststitle{Vertex operator algebras and algebraic smoothing}

  The significance of \zcref{thm:IntroSIC} is that it identifies the strong identity condition as the precise mechanism through which Zhu-type representation theory governs the surrounding module theory.
  Returning to the motivating setting of conformal blocks, this perspective explains why this condition appears naturally in algebraic smoothing. The mode transition algebra method realizes smoothing through induced modules over the Zhu algebra, and thus the validity of smoothing depends fundamentally on whether these induced modules exhaust the relevant representation theory.

  In the VOA setting, this principle admits the following concrete form.
  \begin{mainthm}[cf. \zcref{sec:VOA}]\label{thm:IntroVOA}
    For a VOA $\mathbb{V}$ of CFT type, the algebraic smoothing via the mode transition algebra $\mathfrak{A}$ succeeds if and only if the Zhu algebra $A_0(\mathbb{V})$ establishes a Morita-type equivalence for the categories of admissible $\mathbb{V}$-modules.
  \end{mainthm}

  Consequently, the discrepancy between algebraic smoothing and analytic sewing admits the following representation-theoretic characterization.
  \begin{nocoro}[\zcref{coro:CFT-VOA}]
    For a VOA $\mathbb{V}$ of CFT type, the following conditions are equivalent:
    \begin{enumerate}[label=\textup{(\roman*)}]
      \item $\mathbb{V}$ satisfies smoothing (cf. \cite[Definition 5.0.2]{DGK23}).
      \item The equivalent conditions in \zcref{thm:SIC-CFT} hold. In particular, all admissible $\mathbb{V}$-modules are induced from $\mathsf{A}_{0}$-modules.
      \item Any sheaf of coinvariants of $\mathbb{V}$-modules, assuming coherent, is locally free on any of the moduli spaces $\widehat{\overline{\Moduli}}_{g,n}$.
      \item The mode transition algebra $\mathfrak{A}$ equals to 
      \begin{equation}
        \mathfrak{A}(\mathscr{U}(\mathbb{V})) = \int_{W} W \otimes_{\k} W^{\prime}
      \end{equation}
      where the end is taken over all grading-restricted generalized $\mathbb{V}$-modules.
    \end{enumerate}
  \end{nocoro}

\aststitle{Further structure of almost-canonically seminormed rings}

  The preceding application explains the role of the strong identity condition in the original setting of conformal blocks. However, the framework of almost-canonically seminormed rings is not merely a device for reformulating the VOA smoothing problem. It also provides a natural setting in which several finiteness, rationality, and functoriality phenomena arising in VOA representation theory can be formulated and studied uniformly. 
  The following results illustrate the broader structural scope of the theory.
  
  First, motivated by the contragredient duality formalism in VOA theory, we introduce notions of quasi-rigidity and rationality for almost-canonically seminormed rings. These notions provide a natural framework for relating semisimplicity to SIC:
  \begin{mainthm}[cf. \zcref{sec:rationality}]
    \label{thm:IntroRationality}
    If an almost-canonically seminormed ring $U$ is $\k$-rational, then:
    \begin{enumerate}
      \item $U$ is $\mathsf{A}_{0}$-rational, the Zhu algebras $\mathsf{A}_n$ of all levels are semisimple, and the strong identity condition holds.
      \item Every exhaustive $U$-module is gradable, and the category of exhaustive modules is itself semisimple.
    \end{enumerate}
  \end{mainthm}
  This result for $\k$-rationality is complemented by a further observation regarding $\mathsf{A}_{0}$-rationality: the semisimplicity of the category of positively-graded $(U|\mathsf{A}_{0})$-bimodules is sufficient to ensure that $U$ is $\mathsf{A}_{0}$-rational and the strong identity condition holds. This provides a powerful criterion for verifying SIC in practice.

  Next, we show that the framework of almost-canonically seminormed rings behaves well under tensor products of rings. This compatibility is systematic, extending from the underlying algebraic invariants to the associated representation categories and SIC:
  \begin{mainthm}[cf. \zcref{sec:tensor}]\label{thm:IntroTensor}
    Let $U^{\tt I}$ and $U^{\tt II}$ be almost-canonically seminormed rings. Then, their tensor product $U = U^{\tt I} \otimes U^{\tt II}$ is an almost-canonically seminormed ring. 
    The mode transition algebra and the category of exhaustive modules exhibit natural factorizations:
    \[
      \mathfrak{A}(U) \cong \mathfrak{A}(U^{\tt I}) \otimes \mathfrak{A}(U^{\tt II}), \txand
      \Mod*[Ex](U) \simeq \Mod*[Ex](U^{\tt I}) \boxtimes \Mod*[Ex](U^{\tt II}).
    \]
    Furthermore, under suitable finiteness assumptions, the strong identity condition holds for $U$ if and only if it holds for both factors $U^{\tt I}$ and $U^{\tt II}$.
  \end{mainthm}

  To bridge the gap between analytic sewing via dual module pairs and the algebraic approach via the mode transition algebra, we investigate the categorical behavior of $\mathfrak{A}$. We show that $\mathfrak{A}$ is realized as an \concept{end}, a construction that recovers the ``sum over dual pairs'' perspective within the algebraic framework. 
  This formula identifies SIC as the precise condition under which these two reconstruction mechanisms coincide.
  \begin{mainthm}[{\zcref{thm:EndMode,coro:EndMode}}]\label{thm:IntroEnd}
    Under suitable rigidity or finiteness assumptions, the mode transition algebra $\mathfrak{A}$ of an almost-canonically seminormed ring $U$ is expressed as an end:
    \[
      \mathfrak{A} \cong \int_{M \in \Mod*(\mathsf{A}_{0})} \Phi^{\mathsf{L}}_{0}(M) \otimes \Phi^{\mathsf{R}}_{0}(M^{\vee}),
    \]
    which realizes $\mathfrak{A}$ via the pairing of left and right induced modules.

    If the strong identity condition holds, this formula reduces to 
    \[
      \int_{W \in \Mod*[Ex](U)} W \otimes W^{\prime},
    \]
    namely the end of quasi-rigid exhaustive left $U$-modules tensoring with their graded duals.
  \end{mainthm}

\aststitle{Examples and counterexamples}

  We conclude with several examples illustrating both the scope and the limitations of the theory.

  On the positive side, the Weyl algebra (cf. \zcref{sec:Weyl}) provides an explicit and tractable model in which SIC can be verified directly. 
  Under the Fuchsian grading, the canonical seminorm recovers the natural filtrations by differential order and polynomial degree, allowing explicit descriptions of the associated Zhu algebras, canonical modules, and exhaustive modules. In particular, exhaustive modules are identified with vector bundles equipped with trivial connections.

  In contrast, we show that several irrational vertex operator algebras fail SIC, including the cyclic orbifold $M(1)^+$ and affine vertex operator algebras $V^k(\mathfrak{g})$ and $L_k(\mathfrak{g})$ (cf. \zcref{sec:Examples}) at generic or positive integral levels. These examples illustrate concretely the limitations of algebraic smoothing via mode transition algebras and Zhu-type induction.

\subsection*{History}

  Having outlined these results and applications, we now explain the historical developments leading to the present framework. The strong identity condition arises at the intersection of several threads: enveloping algebras of vertex operator algebras, Zhu-type reconstruction theory, and algebraic approaches to smoothing and sewing of conformal blocks. The present framework may be viewed as the point where these developments converge.

\aststitle{From enveloping algebras to almost-canonically seminormed rings}

  One source of the present framework comes from the theory of \emph{enveloping algebras} of vertex operator algebras. 
  The notion of enveloping algebras originates in the representation theory of Lie algebras.
  The idea is straightforward: although the Lie bracket is non-associative, elements of a Lie algebra act as linear operators on its representations, and those operators compose associatively. 
  The enveloping algebra is then the associative algebra that captures all the possible operators occurring in this way. 
  Consequently, representations of a Lie algebra are equivalently described by its enveloping algebra:
  \[
    \Set*{\text{representations of a Lie algebra}} 
    \simeq
    \Set*{\text{modules of its enveloping algebra}}
  \]

  In the context of vertex operator algebras (VOAs), the situation is vastly more subtle. The notion of \emph{enveloping algebras} was first introduced by Igor Frenkel and Yongchang Zhu in \cite[\S 1.3]{FZ92}. 
  Elements of such an enveloping algebra $\mathscr{U}(\mathbb{V})$ of a VOA $\mathbb{V}$ are possibly infinite linear combinations of compositions of vertex modes:
  \[
    J_{n_1}(a^1) \cdots J_{n_k}(a^k),\qquad
    \text{where }a^i\in \mathbb{V}, n_i\in\mathbb{Z}.
  \]
  Unlike the Lie-theoretic situation, infinite sums are unavoidable in the VOA setting. Consequently, the construction of enveloping algebras necessarily involves a compatible topology. The topology on $\mathscr{U}(\mathbb{V})$ arises in a very specific way. 
  Here, we present a sketch of its construction following \cite{DGK23}. 
  First, the free algebra of modes is graded: where $\deg J_{n}(a) = -n$ and the degree of a composition of modes is just the sum of degrees of individual modes. Then, one can consider the seminorm whose system of neighborhoods of $0$ is given by 
  \[
    \cNL[n]:=\fun{span}\SetCond*{
      J_{n_1}(a^1) \cdots J_{n_k}(a^k) 
    }{ 
      n_{k_0}+\cdots+n_k \ge n\text{ for some }k_0
    }.
  \]
  Such a seminorm is called the \emph{canonical seminorm} in \cite{DGK23}. 
  Then, one can complete this algebra with respect to the canonical seminorm to allow infinite sums. Then, the enveloping algebra $\mathscr{U}(\mathbb{V})$ is obtained as the quotient of this completion by the homogeneous two-sided ideal generated by \emph{vacuum property}, \emph{Jacobi identity}, and \emph{Virasoro relations}. 
  One can check that the enveloping algebra $\mathscr{U}(\mathbb{V})$ is also graded and seminormed and its seminorm behaves similarly to a canonical seminorm; indeed, the canonical seminorm is dense in the former. Such a seminorm is thus called an \emph{almost-canonical seminorm}. 

  Following \cite{FZ92}, several related constructions have appeared in the literature from different motivations, including \cite{FBZ04,NT,Fre07,MNT10,Han20}. 
  All these constructions begin with a graded or split-filtered algebra, complete it with respect to its canonical seminorm, and then take the quotient by a two-sided ideal generated by homogeneous relations.
  This process is axiomatized in \cite[\S A.6]{DGK23} in terms of \emph{almost-canonically seminormed rings}.
  In \cite{DGK23}, building on \cite{MNT10}, some foundations of the general theory are established. Many aspects of this theory, however, remain open for further investigation and advancement.

\aststitle{Zhu's algebras and representation theory}

  While enveloping algebras encode the full representation-theoretic structure of a VOA, they are typically too large and intricate to be used directly in practice. 
  Instead, a much smaller associative algebra $A(\mathbb{V})$, introduced in Zhu's thesis \cite{ZhuThesis} and now called the (zeroth) \emph{Zhu algebra}, is ubiquitously used. 
  In fact, irreducible modules of $A(\mathbb{V})$ are in one-to-one correspondence with irreducible admissible modules of $\mathbb{V}$. Hence, for \emph{rational} VOAs, where all admissible modules are semisimple, the representation theory of $\mathbb{V}$ is completely determined by $A(\mathbb{V})$.

  However, without the assumption of rationality, there may exist admissible $\mathbb{V}$-modules not induced from $A(\mathbb{V})$-modules. To study such modules, and most importantly, to determine when $\mathbb{V}$ is rational, higher-level generalizations of $A(\mathbb{V})$, called the \emph{higher Zhu algebras} $A_{n}(\mathbb{V})$, were introduced by Chongying Dong, Haisheng Li, and Geoffrey Mason in \cite{DLM2}. Relatedly, a family of bimodules $A_{m,n}(\mathbb{V})$ was introduced in \cite{DJmz} to study degree components of modules induced from $A_{n}(\mathbb{V})$. 
  The algebras $A_n(\mathbb{V})$ and the bimodules $A_{m,n}(\mathbb{V})$ are defined as quotients of $\mathbb{V}$ by certain subspaces $O_n(\mathbb{V})$ and $O_{m,n}(\mathbb{V})$ respectively. It was shown in \cite{FZ92,NT} that $A(\mathbb{V})$ can be realized as a subquotient of $\mathscr{U}(\mathbb{V})$; the same was proved for $A_n(\mathbb{V})$ in \cite{He17} and for $A_{m,n}(\mathbb{V})$ in \cite{Han22}. 
  These constructions can be unified in the framework of \emph{almost-canonically seminormed rings and mode transition algebras}; see \zcref{sec:Zhu-mode}. 
  In addition, mode transition algebras provide a finer control on components of admissible $\mathbb{V}$-modules, see e.g. \cite{DGK23,DGK24-Morita}.

  Nevertheless, there remain several questions concerning the algebras $A_n(\mathbb{V})$ and the bimodules $A_{m,n}(\mathbb{V})$. For instance, 
  \begin{enumerate}
    \item For rational VOAs, there is an equivalence between the categories of $A(\mathbb{V})$-modules and of admissible $\mathbb{V}$-modules. It is not yet known what the precise conditions are under which such a Morita-type equivalence holds.
    \item For $n>0$, the higher Zhu algebra $A_n(\mathbb{V})$ never yields a Morita-type equivalence as above (see \zcref{rem:noPhiOmegan}). Rather, a certain subcategory of $A_n(\mathbb{V})$-modules is used in practice. One thus wonders if there is a \emph{thick} version of $A_n(\mathbb{V})$ that does yield a Morita-type equivalence.
  \end{enumerate}
  These questions ultimately lead to the representation-theoretic role of mode transition algebras and the strong identity condition studied in the present work.

\aststitle{Conformal blocks on smoothings via mode transition algebras}

  The preceding developments converge naturally in the algebraic study of conformal blocks and smoothing. The main topic of \cite{DGK23} is the introduction of \emph{mode transition algebras} to study algebraic structures on moduli of stable curves together with the representation theory of the underlying VOAs.
  The mode transition algebra $\mathfrak{A}$ of an almost-canonically seminormed ring\footnote{In the VOA-context considered in \cite{DGK23}, $U$ is the enveloping algebra of the VOA $\mathbb{V}$.} $U$, by its construction, is a bigraded bimodule that has a non-unital associative multiplication. 
  More concretely, it is the double induced bimodule of the \emph{Zhu algebra}\footnote{In the VOA-context, $\mathsf{A}$ is isomorphic to the Zhu algebra $A(\mathbb{V})$.} $\mathsf{A}$:
  \[
    \mathfrak{A} := U/\NL[1]U\otimes_{U_{0}}\mathsf{A}\otimes_{U_{0}}U/\NR[1]U.
  \]

  The key application of the mode transition algebra $\mathfrak{A}$ in \cite{DGK23} is the celebrated characterization of the \emph{smoothing} property in terms of SIC.
  \begin{nothm}[{\cite[Theorem 5.0.3]{DGK23}}]
    For a vertex operator algebra $\mathbb{V}$ of CFT-type (i.e. $\mathbb{V}$ is $\N$-graded with $\mathbb{V}_0=\mathbb{C}\vac$), the following two conditions are equivalent:
    \begin{enumerate}
      \item The mode transition algebra $\mathfrak{A}$ has \emph{strong identity elements}. Namely, elements $\one_{n}\in\mathfrak{A}_{n,-n}$ ($n\in\N$) satisfying the identities
      \begin{equation}
        \one_{n}\star\mathfrak{a}=\mathfrak{a}=\mathfrak{a}\star\one_{m}
        \txforall\mathfrak{a}\in\mathfrak{A}_{n,-m} \,(n,m\in\N).
      \end{equation}
      \item The VOA $\mathbb{V}$ satisfies \emph{smoothing}. 
      That is to say, for any $n$-tuple of admissible $\mathbb{V}$-modules, there exists an element $\one(q)=\sum_{n}\one_{n}q^n\in\mathfrak{A}\dbrack{q}$, such that the map 
      \[
        \alpha\colon W^{\bullet}\dbrack{q} 
        \longrightarrow
        (W^{\bullet}\otimes\mathfrak{A})\dbrack{q},
        \qquad
        u\longmapsto u\otimes \one(q)
      \]
      is compatible with the actions of the chiral Lie algebra $\mathcal{L}_{\mathscr{C}\setminus \pt_{\bullet}}(\mathbb{V})$ for a \emph{formal smoothing family}\footnote{Namely, $\mathscr{C}$ is a family over the formal disk $\Spec(\mathbb{C}\dbrack{q})$ with nodal special fiber and smooth generic fiber, and $\pt_{\bullet}=(\pt_1,\cdots,\pt_n)$ is an $n$-tuple of smooth marked sections of $\mathscr{C}$} $(\mathscr{C},\pt_\bullet)$ 
      and extends the map 
      \[
        \alpha_{0}\colon W^{\bullet}
        \longrightarrow
        W^{\bullet}\otimes\mathsf{A},
        \qquad
        u\longmapsto u\otimes 1_{\mathsf{A}}.
      \]
    \end{enumerate}
  \end{nothm}

  When $\mathbb{V}$ is \emph{$C_1$-cofinite}, it is shown in \cite{DGK22} that $\alpha_0$ induces an isomorphism between spaces of coinvariants
  \[
    \brk[s]{W^{\bullet}}_{(\mathscr{C}_{0}, \pt_{\bullet},  t_{\bullet})}
    \cong
    \brk[s]{W^{\bullet}\otimes\mathfrak{A}}_{(\widetilde{\mathscr{C}_{0}}, \pt_{\bullet} \sqcup \qt_{\pm},  t_{\bullet} \sqcup s_{\pm})},
  \]
  where $\widetilde{\mathscr{C}_{0}}$ is the normalization of the nodal curve $\mathscr{C}_{0}$ with $\qt_{\pm}$ being the two preimages of the node, and $t_{\bullet}, s_{\pm}$ are chosen formal coordinates at the marked points $\pt_{\bullet}$ and at $\qt_{\pm}$.
  If this is the case and assume the sheaves 
  \[
    \brk[s]{W^{\bullet}}_{(\mathscr{C}, \pt_{\bullet},  t_{\bullet})}
    \txand
    \brk[s]{W^{\bullet}\otimes\mathfrak{A}}_{(\widetilde{\mathscr{C}}, \pt_{\bullet} \sqcup \qt_{\pm},  t_{\bullet} \sqcup s_{\pm})},
  \]
  where $\widetilde{\mathscr{C}}$ is the trivial extension of $\widetilde{\mathscr{C}_{0}}$ over the formal disk $\Spec(\mathbb{C}\dbrack{q})$, are coherent. Then, $\alpha$ induces an isomorphism $[\alpha]$ (the \emph{formal sewing map}) of these two sheaves on $\Spec(\mathbb{C}\dbrack{q})$. 
  In this way, using formal gluing (c.f. \cite{Pries00}), one relates the spaces of coinvariants associated to nodal curves with those associated to smooth ones. Hence, if we have assumed  all the involved sheaves of coinvariants are coherent, then the construction of sheaves of coinvariants $\brk[s]{W^{\bullet}}_{(C,\pt_{\bullet},t_{\bullet})}$ gives rise to a vector bundle $\brk[s]{W^{\bullet}}$ on the moduli stack of coordinated stable curves $\widehat{\overline{\Moduli}}_{g,n}$ (cf. \cite[Corollary 5.2.6]{DGK23}).

  Before proceeding, it is worth emphasizing that the significance of the algebraic approach developed in \cite{DGK23} lies not only in extending smoothing constructions to arbitrary genera and non-semisimple settings, but also in providing a fundamentally algebraic approach to sewing and factorization. In particular, many of the constructions remain valid over arbitrary base fields, where complex-analytic methods are unavailable.

\aststitle{Concerns on the strong identity condition}

  Despite these developments, from the perspective of analytic sewing and vertex tensor category theory, several conceptual tensions remain:
  \begin{enumerate}
    \item The smoothing property in \cite{DGK23} is \emph{equivalent} to the strong identity condition. However, this condition does not hold in many situations (e.g. the triplet vertex operator algebra $\mathcal{W}(p)$, see \cite[\S 9.1]{DGK23}) where the sewing is known to be valid in the analytic setting.
    \item In the algebraic approach of \cite{DGT-factorization,DGK23}, the sewing is realized by assigning the \emph{mode transition algebra $\mathfrak{A}$} to the preimages of the \emph{node} in the normalization of a nodal curve, and the sewing isomorphism follows from the isomorphism $\alpha_0$ on the central fibers.
    In contrast, the analytic approach assigns \emph{pairs of dual (contragredient) modules} to the \emph{gluing tubes} on Riemann surfaces, and the sewing is given in terms of product and iterate intertwining operators. The convergence of such formal series is central in these theories.
    \item Since the mode transition algebra is a double induced bimodule of the Zhu algebra, tensor products of modules constructed via algebraic sewing must be \emph{induced modules} (also known as \emph{generalized Verma modules}). 
    However, in non-semisimple cases, this is often not the case.
  \end{enumerate}

  Addressing these concerns is one of the central motivations of this work. 
  The results above may be summarized conceptually as follows:
  \begin{quote}
    The strong identity condition characterizes precisely when algebraic sewing admits the same dual-pair formalism as in the analytic sewing. Moreover, in this case, every admissible $\mathbb{V}$-module is induced from the Zhu algebra.
  \end{quote}

\subsection*{Outline of the paper}

  This paper is organized into four parts, whose logical dependencies are summarized in the following diagram.
  
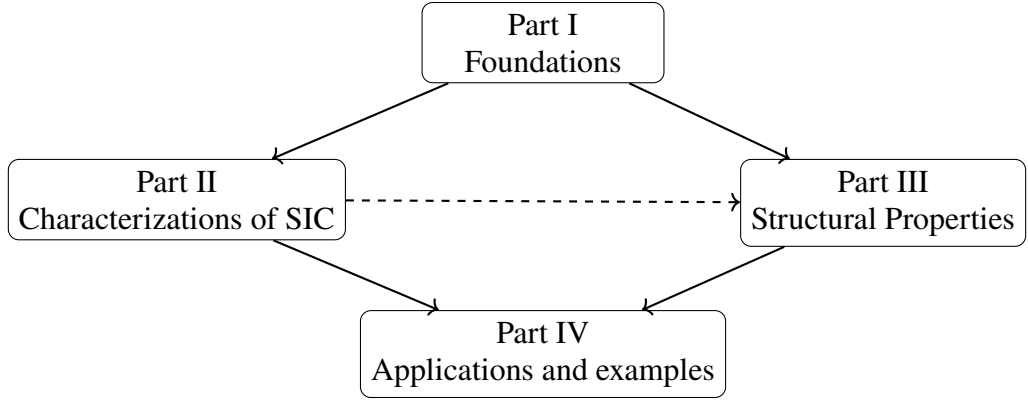
\begin{center}
    \begin{tikzpicture}[
      box/.style={draw, rounded corners, align=center, minimum width=3.2cm, minimum height=1.0cm},
      arrow/.style={->, thick}
    ]

    % Nodes
    \node[box] (I) {Part I\\Foundations};
    \node[box, below left=of I] (II) {Part II\\Characterizations of SIC};
    \node[box, below right=of I] (III) {Part III\\Structural Properties};
    \node[box, below=3cm of I] (IV) {Part IV\\Applications and examples};

    % Arrows
    \draw[arrow] (I) -- (II);
    \draw[arrow] (I) -- (III);
    \draw[arrow] (II) -- (IV);
    \draw[arrow] (III) -- (IV);

    % cross dependency (important logical nuance)
    \draw[arrow, dashed] (II) -- (III);

    \end{tikzpicture}
  \end{center}

  \zcref{part1} develops the basic framework of almost-canonically seminormed rings and their representation theory. 
  \zcref{sec:ACSR} introduces almost-canonically seminormed rings, Zhu algebras, canonical modules, and mode transition algebras, together with the definition of the strong identity condition. 
  \zcref{sec:modules,sec:InducedMod} establish the relevant module categories (such as exhaustive, positively-filtered, graded, and quasi-rigid modules)  and construct the induction functors.

  \zcref{part2} forms the technical core of the paper and is devoted to the characterization of SIC (\zcref{thm:IntroSIC}). 
  \zcref{sec:SIE} introduces strong identity expansions, while \zcref{sec:splitOmega} relates them to the splitting of the $\Omega$-filtration. 
  The subsequent sections, namely \zcref{sec:Morita,sec:Morita_to_SIC,sec:HigherMorita}, establish the equivalence between SIC, projectivity of canonical modules, and Morita-type equivalences arising from ordinary and thick Zhu algebras. 

  \zcref{part3} develops broader structural aspects of our framework whose interaction with SIC becomes transparent through the characterizations established in \zcref{part2}.  
  \zcref{sec:rationality} introduces a generalized notion of rationality and relates it to semisimplicity and SIC (\zcref{thm:IntroRationality}). 
  \zcref{sec:tensor} proves that the framework of almost-canonically seminormed rings is compatible with tensor products, both algebraically and categorically (\zcref{thm:IntroTensor}). 
  Finally, \zcref{sec:End} establishes the end formula for the mode transition algebra (\zcref{thm:IntroEnd}), thereby connecting algebraic smoothing to the dual-pair formalism underlying analytic sewing.
  
  Finally, \zcref{part4} returns to concrete algebraic and VOA-theoretic settings. 
  \zcref{sec:Weyl} studies the Weyl algebra as an explicit prototype, verifying SIC and describing its exhaustive modules in concrete terms. 
  \zcref{sec:VOA} applies the general theory to vertex operator algebras and identifies the precise representation-theoretic obstruction governing algebraic smoothing (\zcref{thm:IntroVOA}). 
  The final section, \zcref{sec:Examples}, concludes with examples of irrational vertex operator algebras for which SIC fails.

% \clearpage
\part{Basics notions}\label{part1}
  In this part, we set up the foundational algebraic and categorical framework.

\section*{Conventions}
  Throughout this paper, we work over a fixed commutative ring $\nota{\k}$: 
  by \textbf{spaces}, we mean $\k$-modules, not necessarily free; 
  by \textbf{algebras}, we mean $\k$-modules equipped with an associative $\k$-bilinear multiplication, not necessarily unital; 
  and by \textbf{rings}, we mean unital algebras.
  All tensor products $\otimes$ are over $\k$ unless otherwise specified. 

  We will consider graded algebras and graded modules. Without specification, all gradings are by $\Z$. 
  For each $n\in\Z$, we write $\nota{[n]}$ for the graded space concentrated in degree $n$ with value $\k$, and for each graded space $M$, we write $\nota{M[n]}$ for $M\otimes[n]$.

\subsection*{Duals and rigidity}
  Let $R$ be an algebra. 
  We write $\nota{(-)^{\vee|R}}$ (resp. $\nota{(-)^{R|\vee}}$) for the functor sending a right (resp. left) $R$-module $M$ to its dual left (resp. right) $R$-module $\Hom_{R}(M,R)$. 
  When $R$ is implied, we omit it from the notation. 
  When the side of the module structure is clear, we also omit the vertical bar. 
  When $R=\k$, we simply write $\nota{M^{\ast}}$ for $M^{\vee|\k}$. 

  An $R$-module $M$ is called \textbf{rigid} if $M^{\vee} \otimes_R (-) \to \Hom_R(M,-)$ is a natural isomorphism (equivalently, the canonical map $M^{\vee} \otimes_R M \to \End_R(M)$ is an isomorphism). Recall that, if $R$ is unital, then an $R$-module $M$ is rigid precisely when it is a finitely generated and projective.

\subsection*{Categories of bimodules}
  For any algebras $R$ and $S$, we will use the following notations:
  \begin{description}[leftmargin = 7.5\parindent]
    \item[\nota{$\Mod(R|S)$}] 
        the category of all $(R|S)$-bimodules, i.e. spaces equipped a compatible left $R$-action and right $S$-action.
    \item[\nota{$\Mod({}_R|S)$}] 
        same as $\Mod(R|S)$, but we emphasize $R$ as the base.
    \item[\nota{$\Mod(R|{}_S)$}] 
        same as $\Mod(R|S)$, but we emphasize $S$ as the base.
    \item[\nota{$\Mod(R|)$}] 
        the category of all left $R$-modules.
    \item[\nota{$\Mod(|R)$}] 
        the category of all right $R$-modules.
    \item[\nota{$\Mod*(R|{}_S)$}] 
        the full subcategory of $\Mod(R|{}_S)$ consisting of all bimodules that are rigid as right $S$-modules.
    \item[\nota{$\Mod*({}_R|S)$}] 
        the full subcategory of $\Mod({}_R|S)$ consisting of all bimodules that are rigid as left $R$-modules.
    \item[\nota{$\Mod*(R|)$}] 
        the full subcategory of $\Mod(R|)$ consisting of all left $R$-modules that are rigid as $\k$-modules.
    \item[\nota{$\Mod*(|R)$}] 
        the full subcategory of $\Mod(|R)$ consisting of all right $R$-modules that are rigid as $\k$-modules.
  \end{description}
  The capital categories $\Mod(\cdots)$ are Grothendieck abelian categories and the last four are full exact subcategories of them. 

\subsection*{Categorical language}
  Results in this paper can also be formulated and proved in the following generality. \emph{Readers who are not interested in generality can ignore this remark and similar ones occasionally appearing.}

  Instead of working over a commutative ring $\k$, one may also work in an abelian category $\mathbb{K}$ with a symmetric monoidal structure $(\otimes,\unit)$ satisfying the following conditions:
  \begin{enumerate}
    \item The abelian category $\mathbb{K}$ is a \emph{Grothendieck abelian category}, i.e., it has a generating set, all colimits exist, and all filtered colimits are exact.
    \item The tensor product $\otimes$ is additive and cocontinuous in each variable.
    \item The monoidal structure $(\otimes,\unit)$ is closed, i.e., for any $x\in\mathbb{K}$, the tensor product functor $-\otimes x$ has a right adjoint $[x,-]$.
    \item There is a generating set $\mathcal{G}$ of $\mathbb{K}$ consists of \emph{rigid} objects: i.e. $[\grp{g},\unit]\otimes(-) \cong [\grp{g},-]$ for all $\grp{g} \in \mathcal{G}$.
  \end{enumerate}
  In practice, one may be interested in the following examples: the category of modules over $\k$; the category of $\Gamma$-graded (where $\Gamma$ is an abelian group) modules over $\k$; the category of super modules over $\k$; the category of differential graded modules over $\k$.

\section{Almost-canonically seminormed rings and strong identities}\label{sec:ACSR}
\subsection{Almost-canonical seminorms} 
  The main object under consideration in this paper is a graded ring (i.e. a unital algebra) $U$ equipped with an {almost-canonical (left) seminorm}: 
  \begin{definition}[{\cite[\S A.6]{DGK23}\footnote{In \cite{DGK23}, almost-canonical seminorms are defined for \emph{split-filtered rings}. But for the purpose of this paper, there is no difference to work with the associated graded algebra of a split-filtered algebra. See \cite[\S A.8]{DGK23}}}]
    An \concept{almost-canonical (left) seminorm} on a graded algebra $\nota{U} = \bigoplus_{n \in \Z} U_{n}$ is a system of neighborhoods $\nota{\NL{U}}$ of $0$ in $U$ verifying the following conditions:
    \begin{enumerate}[label=(\texttt{A}\arabic*)]
      \item\label{A1} Each neighborhood $\NL[n]U$ is a graded subspace of $U$.
      \item\label{A2} The \concept{canonical seminorm} $\nota{\cNL[\bullet]U}:=UU_{\le-\bullet}$ is contained in and degreewise dense in $\NL[\bullet]U$.
      \item\label{A3} $(\NL[n]U_{p})U_{q}\subset\NL[n-q]U_{p+q}$ and $U_{p}(\NL[n]U_{q})\subset\NL[n]U_{p+q}$ for all $n,p,q\in\Z$.
    \end{enumerate}
    Note that, by \zcref[noname]{A3}, each $\NL[n]U$ is a graded left ideal of $U$. Hence, we also call such a system $\NL{U}$ a system of \concept{left neighborhoods}.
  \end{definition}
  \begin{remark}
    In the categorical language, the notion of an almost-canonical seminorm in the above sense is interpreted as an inductive system $\NL{U}$ of \emph{subobjects} of $U$ in the category of graded objects satisfying condition \zcref[noname]{A3} and the following variant of condition \zcref[noname]{A2}:
    \begin{enumerate}[label=(\texttt{A}\arabic*'),start=2]
      \item $\NL[n]U_{p} = \cNL[n]U_{p}+\NL[n+1]U_{p}$ for all $n,p\in\Z$.
      \label{A2p}
    \end{enumerate}
    Indeed, one can show that these two variants are equivalent. The advantage of \zcref[noname]{A2} is that it is more suitable for the topological language that is used throughtout this paper.
  \end{remark}
  \begin{notation}
    The system of \concept{right neighborhoods} associated to $\NL$ is defined by 
    \[
      \nota{\NR[\bullet]U_{\Box}} := \NL[\bullet-\Box]U_{\Box}.
    \]
    From this definition, one can verify that $\NR$ satisfies properties analogous to those of $\NL$ with a changing of sides. 
    For instance, each $\NR[n]U$ is a graded right ideal of $U$.

    Note that, the left and right neighborhoods $\NL$ and $\NR$ coincide on the subalgebra $U_{0}$; that is, $\NL[\bullet]U_{0} = \NR[\bullet]U_{0}$. So we omit the subscripts and simply denote them by $\nota{\NLR[\bullet]U_{0}}$.
  \end{notation}
  \begin{remark}\label{rem:opposite}
    Let $U$ be an almost-canonically seminormed algebra. 
    We can view its opposite algebra $\nota{\opp{U}}$ as a graded algebra equipped with the reversed grading: $\opp{U}_{\Box} = U_{-\Box}$. 
    Then, the seminorm on $U$ induces one on $\opp{U}$ defined by 
    \[
      \NL[\bullet]\opp{U}_{\Box} := \NR[\bullet]U_{-\Box}.
    \]
    This is another way to understand the system of right neighborhoods.
  \end{remark}

  \begin{example}
    The primary motivating example of an almost-canonically seminormed ring is the \emph{enveloping algebra} $\mathscr{U}(\mathbb{V})$ of a vertex operator algebra $\mathbb{V}$ over $\k$. 
    We omit the detailed construction here and leave the discussion to \zcref{sec:VOA}. 
    For further background, see \cite{FZ92,FBZ04,Fre07,NT,MNT10}, especially \cite[\S 2]{DGK23}.
  \end{example}

  We do not expect readers to have backgrounds in vertex algebras. Nevertheless, it is helpful to present some concrete examples out of the VOA-realm to illustrate the general theory. Below, we provide several such examples:

  \begin{example}
    Let $U = \k[x^{\pm 1}]$ be the ring of \emph{Laurent polynomials}, graded by $\deg x = 1$ and $\deg x^{-1} = -1$. In this case, the canonical seminorm is trivial: $\cNL[n]U = \cNR[n]U = U$ for all $n$. Thus, $U$ does not admit any nontrivial almost-canonical seminorm.

    More generally, consider the \emph{Laurent tensor algebra} $U = \bigoplus_{n \in \Z} M^{\otimes n}$ of a space $M$. Here, $M^{\otimes n}$ denotes the $n$-th tensor power of $M$ for $n \geq 0$, and for $n < 0$, we set $M^{\otimes n} := (M^{\ast})^{\otimes -n}$, where $M^{\ast}$ is the dual space of $M$. Multiplication is given by concatenation of tensors, together with the evaluation map $M \otimes M^{\ast} \to \k$. Then, the canonical seminorms are given by $\cNL[n]U = U M^{\otimes -n}$ and $\cNR[n]U = M^{\otimes n} U$. If the evaluation map $M \otimes M^{\ast} \to \k$ is surjective, then we again have trivial seminorms: $\cNL[n]U = \cNR[n]U = U$ for all $n$.
  \end{example}

  \begin{example}\label{eg:Weyl}
    Let $U = \mathcal{D} = \k\brk[a]{x,\partial_{x}}/(\partial_{x}x-x\partial_{x}-1)$ be the \emph{rank-one Weyl algebra}. It is the \emph{noncommutative} ring generated by $x$ and $\partial_{x}$ subject to the relation $\partial_{x}x-x\partial_{x}=1$. 
    There are various useful filtrations on $\mathcal{D}$.
    Here, we consider the \emph{Fuchsian grading}\footnote{The terminology comes from the Fuchsian condition at infinity: a differential operator $P\in\mathcal{D}$ is said to be \emph{Fuchsian at $\infty$}, or to have a \emph{regular singularity at $\infty$}, if the Fuchsian degree is increasing with respect to the orders of monomials in $P=\sum_{i}f_{i}(x)\partial^{i}$: that is, $\deg f_{i} - i > \deg f_{j} - j$ whenever $i>j$. See \cite[\S 5.1.2]{HTTDmodules}}: 
    \[
      \deg x=1
      \txand
      \deg\partial_{x}=-1. 
    \]
    Then, the canonical seminorm can be characterized as follows: for each $n\in\N$,
    \begin{align*}
      \cNL[n]\mathcal{D} &= \mathcal{D}\partial_{x}^{\ge n} = \Set{f\in\mathcal{D}\given\text{each monomial of $f$ contains at least $n$ powers of $\partial_{x}$}},\\
      \cNR[n]\mathcal{D} &= x^{\ge n}\mathcal{D} = \Set{f\in\mathcal{D}\given\text{each monomial of $f$ contains at least $n$ powers of $x$}}.
    \end{align*}
    We see that the left and right neighborhoods are precisely the filtration by \emph{orders of differential operators} and the filtration by \emph{degrees of polynomials} respectively.
    For more details and further discussion, see \zcref{sec:Weyl}.
  \end{example}

  \begin{example}\label{eg:degreewiseHom}
    A ($\k$-linear) \emph{$\Z$-algebra} (cf. \cite{OW11,TensorGrothCats}) is a $\k$-linear category with the set of objects isomorphic to $\Z$. Given a $\Z$-algebra $\mathcal{A}$, its path algebra
    \[
      U:=\bigoplus_{p,q\in\Z}\Hom_{\mathcal{A}}(p,q)
    \]
    is a graded ring with the grading $\deg\Hom_{\mathcal{A}}(p,q) = p-q$. 
    Conversely, any graded ring $U$ gives rise to a $\Z$-algebra $\mathcal{A}$ which is the full subcategory of left graded $U$-modules with objects $\Set{ U[p] }_{p \in \Z}$. 
    Note that, $\Hom_{\mathcal{A}}(U[p],U[q]) = U_{p-q}$.

    More generally, let $\mathcal{A}$ be a $\k$-linear semicategory (i.e. morphisms are composable but not necessarily contain identities) with the set of objects isomorphic to $\Z$. 
    Then, the canonical seminorm on its path algebra $U$ is given by
    \begin{align*}
      \cNL[n]U &= 
        \bigoplus_{p,q\in\Z}
          \Set{ f \colon p \to q \given f \text{ factors through some } r, \text{ where }q-r\ge n },\\
      \cNR[n]U &= 
        \bigoplus_{p,q\in\Z}
          \Set{ f \colon p \to q \given f \text{ factors through some } r, \text{ where }p-r\ge n }.
    \end{align*}
    This example is closely related to the structure of a mode transition algebra.
  \end{example}

  The following example shows why one needs the notion of \emph{almost-canonical seminorms} instead of just considering \emph{canonical seminorms}.
  \begin{example}[{\cite[Lemma A.6.13]{DGK23}}]\label{eg:completion}
    Let $U$ be an almost-canonically seminormed algebra and consider the \emph{degreewise completion} $\widehat{U}$ of $U$ with respect to the almost-canonical seminorm\footnote{Note that, since $\NL[n]U_{\bullet}=\NR[n+\bullet]U_{\bullet}$, replacing $\NL$ by $\NR$ does not change the limit.}:
    \[
      \widehat{U}_{\bullet} := \varprojlim_{n\in\N} U_{\bullet}/\NL[n]U_{\bullet}.
    \]
    Then, the induced seminorm on $\widehat{U}$ is an almost-canonical seminorm. 
    
    Note that, even if we start with the \emph{canonical seminorm} $\cNL[\bullet]U$, the completion $\widehat{U}$ is not necessarily canonically seminormed. That's why we need to generalize the notion of canonical seminorms to \emph{almost-canonical seminorms}.
  \end{example}

\subsection{Zhu algebras and the mode transition algebra}\label{sec:Zhu-mode}
  For each $n\in\N$, define 
  \begin{equation}\label{def:ALR}
    \begin{gathered}
      \nota{\mathfrak{L}^{n}_{\bullet}}:=(U/\mathsf{N}_{\mathsf{L}}^{n+1}U)_{\bullet},\qquad
      \nota{\mathfrak{R}^{n}_{\bullet}}:=(U/\mathsf{N}_{\mathsf{R}}^{n+1}U)_{\bullet},\\
      \txand
      \nota{\mathsf{A}_{n}} := \mathfrak{L}^{n}_{0} = \mathfrak{R}^{n}_{0} = U_0/\mathsf{N}^{n+1}U_0.
    \end{gathered}
  \end{equation}
  For any element $u\in U$, its images in $\mathfrak{L}^{n}$, $\mathfrak{R}^{n}$, and $\mathsf{A}_{n}$ are denoted by $\nota{[u]^{\mathsf{L}}_{n}}$, $\nota{[u]^{\mathsf{R}}_{n}}$, and $\nota{[u]_{n}}$ respectively. 
  When $U$ is unital, we further denote the image of $1\in U$ in them simply by $\nota{1_{n}}$. 
  \begin{definition}\label{def:Zhu}
    The algebra $\mathsf{A}_{n}$ is called the \concept{$n$-th Zhu algebra} of $U$.
  \end{definition}

  Note that each $\mathfrak{L}^{n}$ (resp. $\mathfrak{R}^{n}$) is a discrete $(U|\mathsf{A}_{n})$-bimodule (resp. $(\mathsf{A}_{n}|U)$-bimodule). Hence, the space $\nota{\mathsf{A}_{m,n}}:=\mathfrak{L}^{n}_{m-n}=\mathfrak{R}^{m}_{m-n}$ is a $(\mathsf{A}_{m}|\mathsf{A}_{n})$-bimodule.

  \begin{remark}
    Using these notations, $\widehat{U}$ is both the projective limit of the graded $U$-modules $(\mathfrak{L}^{n}_{\bullet})_{n\in\N}$ and $(\mathfrak{R}^{n}_{\bullet})_{n\in\N}$. This provides a more categorical way to consider the topological ring $U$.
  \end{remark}

  \begin{para}[Relation to the literature]
    The notion of the \emph{(zeroth) Zhu algebra} $A(\mathbb{V})$ of a vertex operator algebra $\mathbb{V}$ was introduced in \cite{ZhuThesis} as the quotient of $\mathbb{V}$ by a suitable subspace $O(\mathbb{V})$. It was shown in \cite{FZ92,NT} that $A(\mathbb{V})$ is isomorphic to an appropriate quotient of the degree-zero part of the enveloping algebra $\mathscr{U}(\mathbb{V})$. In our language, that quotient is precisely $\mathsf{A}_{0}(\mathscr{U}(\mathbb{V}))$.
    \emph{Higher Zhu algebras} $\mathsf{A}_{n}$ were defined in \cite{DLM2} as quotients of $\mathbb{V}$ by subspaces $O_{n}(\mathbb{V})$. In \cite{He17}, it was shown that $A_{n}(\mathbb{V})$ is isomorphic to our $n$-th Zhu algebra $\mathsf{A}_{n}$ of the almost-canonically seminormed ring $\mathscr{U}(\mathbb{V})$.
    Other realizations of $A_{n}(\mathbb{V})$ can be found in \cite{H05b,vEH19}.
    In \cite{DJmz}, a family of $(A_{m}(\mathbb{V})|A_{n}(\mathbb{V}))$-bimodules $A_{m,n}(\mathbb{V})$ was introduced to describe the transition of admissible modules from level $m$ to level $n$. It was shown in \cite{Han22} that the bimodule $A_{m,n}(\mathbb{V})$ can be interpreted as the $(\mathsf{A}_{m}|\mathsf{A}_{n})$-bimodule $\mathsf{A}_{m,n}$.

    These constructions extend to other contexts. For example, in orbifold theory, namely the study of twisted representations (see \cite{DLM1,DJtams}), for the \emph{$g$-twisted} enveloping algebra $\mathscr{U}_{g}(\mathbb{V})$, its Zhu algebras $\mathsf{A}_{n}$ are isomorphic to the \emph{$g$-twisted Zhu algebras} $A_{g,n}(\mathbb{V})$ of $\mathbb{V}$, and the $(\mathsf{A}_{m}|\mathsf{A}_{n})$-bimodules $\mathsf{A}_{m,n}$ are isomorphic to the $(A_{g,m}(\mathbb{V})|A_{g,n}(\mathbb{V}))$-bimodules $A_{g,m,n}(\mathbb{V})$; see \cite{Han20,Han25}.
  \end{para}

  \begin{notation}
    We will also denote the images of $\NL[k]U_{\bullet}$, $\NR[k]U_{\bullet}$, and $\NLR[k]U_{0}$ in the discrete spaces $\mathfrak{L}^{n}_{\bullet}$, $\mathfrak{R}^{n}_{\bullet}$, and $\mathsf{A}_{n}$ by $\nota{\NL[k]\mathfrak{L}^{n}_{\bullet}}$, $\nota{\NR[k]\mathfrak{R}^{n}_{\bullet}}$, and $\nota{\NLR[k]\mathsf{A}_{n}}$ respectively. Be aware that these are not seminorms, but merely finite-length filtrations.
  \end{notation}

  A key observation towards the mode transition algebra is the following perfect pairing:
  \begin{lemma}[{\cite[B.2.1]{DGK23}}]\label{lem:def:ostar}
    If $U$ is unital, the following pairing is an isomorphism:
    \begin{align*}
      \mathfrak{R}^{0}\otimes_{U}\mathfrak{L}^{0}&\overset{\ostar}{\longrightarrow}\mathsf{A}_{0}\\
      [\alpha]^{\mathsf{R}}_{0}\otimes[\beta]^{\mathsf{L}}_{0}&\longmapsto
      \begin{dcases*}
        [\alpha\beta]_{0} & if $\deg(\alpha)+\deg(\beta)=0$,\\
        0 & otherwise.
      \end{dcases*}
    \end{align*}
  \end{lemma}
  \begin{proof}
    The inverse of $\ostar$ is given by the quotient map $U\to U/(\NL[1]U+\NR[1]U)$, where the latter is identified with $\mathfrak{R}^{0}\otimes_{U}\mathfrak{L}^{0}$. We omit the details here.
  \end{proof}

  In the proof of the above lemma, what we need from the assumption of unitality is that $U_{\le -1} \subset \NL[1]U$ and $U_{\ge 1} \subset \NR[1]U$. Hence, we introduce the following weaker notion:
  \begin{definition}\label{def:weakunital}
    A graded algebra $U$ is \concept{weakly unital} if for all $n\in\Z$, we have 
    \[
      U_{0}U_{n} = U_{n} = U_{n}U_{0}.
    \]
  \end{definition}
  Then, the assumption in \zcref{lem:def:ostar} can be weakened to the weak unitality of $U$. 
  Note that, under this assumption, $\mathfrak{L}^{0}$ has only nonnegative degrees, and $\mathfrak{R}^{0}$ has only nonpositive degrees.

  \begin{definition}\label{def:mode-transition-algebra}
    Let $U$ be a weakly unital almost-canonically seminormed algebra. 
    Then, the \concept{mode transition algebra} of $U$ is the bigraded $(U|U)$-bimodule
    \[
      \nota{\mathfrak{A}_{\bullet,-\circ}}:=\mathfrak{L}^{0}_{\bullet}\otimes_{\mathsf{A}_{0}}\mathfrak{R}^{0}_{-\circ}.
    \]
    Its multiplication $\star$ is induced by the isomorphism $\ostar$.
    Under this multiplication, each diagonal block $\mathfrak{A}_{n,-n}$ is a subalgebra, denoted by $\mathfrak{A}_{n}$.
  \end{definition}
  \begin{remark}\label{rem:multiplication-of-fA}
    By \zcref{lem:def:ostar}, for any $a,b,c,d$, we have
    \[
      \mathfrak{A}_{a,-b} \star \mathfrak{A}_{c,-d} \subset 
      \begin{dcases*}
        \mathfrak{A}_{a,-d} & if $b=c$,\\
        0 & otherwise.
      \end{dcases*}
    \]
  \end{remark}

  The notion of the mode transition algebra is introduced in \cite{DGK23} to generalize the sewing isomorphism of coinvariants from \cite[\S 8]{DGT-factorization}. The essential condition for the existence of such an isomorphism is the following (cf. \cite[Theorem 5.0.3]{DGK23}):
  \begin{definition}\label{def:SIC}
    We say that the \concept{strong identity condition} (\nota{\textsf{SIC}} for short) holds (for $U$) if there are \concept{strong identity elements} $\nota{\one_{n}}\in\mathfrak{A}_{n}$ ($n\in\N$) satisfying the identities
    \begin{equation}\label{eq:SIC}
      \one_{n}\star\mathfrak{a}=\mathfrak{a}=\mathfrak{a}\star\one_{m}
      \txforall\mathfrak{a}\in\mathfrak{A}_{n,-m} \,(n,m\in\N).
    \end{equation}

    Note that the isomorphism $\ostar$ induces a left action of $\mathfrak{A}$ on $\mathfrak{L}^{0}$ and a right action of $\mathfrak{A}$ on $\mathfrak{R}^{0}$. We will denote them by $\star$ again. Under these actions, the condition \eqref{eq:SIC} is equivalent to
    \begin{equation}\label{eq:SIC:action}
      \one_{n}\star[\alpha]^{\mathsf{L}}_{0} = [\alpha]^{\mathsf{L}}_{0}
      \txand{} 
      [\beta]^{\mathsf{R}}_{0}\star\one_{m} = [\beta]^{\mathsf{R}}_{0}
      \txforall 
      \alpha\in U_{n}, \beta\in U_{-m} \,(n,m\in\N).
    \end{equation}
  \end{definition}

  \begin{example}
    Let $U=\mathcal{D}$, the rank-one Weyl algebra (cf. \zcref{eg:Weyl}). 
    Then, we have 
    \begin{enumerate}
      \item $\mathfrak{A}_{p,-q} \cong \k x^{p}\otimes \partial^{q}$ (cf. \zcref{prop:modeAofWeyl});
      \item When the rational number field $\Q$ is contained in $\k$, the strong identity condition holds (cf. \zcref{prop:SICegWeyl}): for each $n\in\N$, we can take 
      \[
        \one_{n} = \frac{1}{n!}[x^{n}]^{\mathsf{L}}_{0}\otimes[\partial^{n}]^{\mathsf{R}}_{0}\in\mathfrak{A}_{n}.
      \]
      \item (\zcref{rem:SICnonegWeyl}) On the other hand, if there is any integer $n\in\N$ vanishing in $\k$, then $U$ fails the strong identity condition.
    \end{enumerate}
  \end{example}

  \begin{example}
    One may want to interpret the mode transition algebra $\mathfrak{A}$ as a $\Z$-algebra (cf. \zcref{eg:degreewiseHom}) where $\Hom(p,q) = \mathfrak{A}_{p,-q}$. This is the case precisely when the strong identity condition holds. 
    Nevertheless, we can always grade $\mathfrak{A}$ by setting $\deg\mathfrak{A}_{p,-q} = p-q$.
    With the grading, $\mathfrak{A}$ is weakly unital even when the strong identity condition fails. 

    Now, the canonical seminorm on this graded algebra is given by:
    \[
      \cNL[n]\mathfrak{A} = \bigoplus_{p\in\N, q\ge n}\mathfrak{A}_{p,-q}
      \txand 
      \cNR[n]\mathfrak{A} = \bigoplus_{p\ge n, q\in\N}\mathfrak{A}_{p,-q}.
    \]
    Indeed, for any $p,q \in \N$, by \zcref{rem:multiplication-of-fA}, we have
    \[
      \mathfrak{A}\star\mathfrak{A}_{p,-q} \subset \bigoplus_{r\in\N}\mathfrak{A}_{r,-q}
      \txand 
      \mathfrak{A}_{p,-q}\star\mathfrak{A} \subset \bigoplus_{s\in\N}\mathfrak{A}_{p,-s}.
    \]
    This shows $\cNL[n]\mathfrak{A} \subset \bigoplus_{p\in\N, q\ge n}\mathfrak{A}_{p,-q}$ and $\cNR[n]\mathfrak{A} \subset \bigoplus_{p\ge n, q\in\N}\mathfrak{A}_{p,-q}$.
    On the other hand, for all $p,q\in\N$, we have 
    \[
      \mathfrak{A}_{p,0}\star\mathfrak{A}_{0,-q} = \mathfrak{A}_{p,-q}.
    \] 
    In particular, $\mathfrak{A}_{p,-q} \subset \cNL[n]\mathfrak{A}$ when $q \ge n$ and $\mathfrak{A}_{p,-q} \subset \cNR[n]\mathfrak{A}$ when $p \ge n$.
    
    From this description, the following is evident:
    \begin{enumerate}
      \item $\mathsf{A}_{n}(\mathfrak{A}) = \bigoplus_{p\le n}\mathfrak{A}_{p,-p}$ for all $n\in\N$.
      \item $\mathfrak{L}^{n}(\mathfrak{A}) = \bigoplus_{p\in\N, q\le n}\mathfrak{A}_{p,-q}$ and $\mathfrak{R}^{n}(\mathfrak{A}) = \bigoplus_{p\le n, q\in\N}\mathfrak{A}_{p,-q}$ for all $n\in\N$.
      \item $\mathfrak{A}_{p,-q}(\mathfrak{A}) = \mathfrak{A}_{p,-q}$ for all $p,q\in\N$.
    \end{enumerate}
    Indeed, the first two follow from previous description. The last one follows from the previous two since $\mathfrak{A}_{p,0}\otimes_{\mathfrak{A}_{0,0}}\mathfrak{A}_{0,-q} = \mathfrak{A}_{p,-q}$.

    In particular, if $\mathfrak{A}$ admits strong identity elements $\one_{n}\in\mathfrak{A}_{n}$ ($n\in\N$), then they also give rise to strong identity elements for the mode transition algebra of $\mathfrak{A}$ itself.
  \end{example}

  \begin{egrmk}\label{eg:completions}
    One can also play with the (separated) completion $\widehat{U}$ of an almost-canonically seminormed ring $U$. 
    However, since the quotients $\mathfrak{L}^{n}$, $\mathfrak{R}^{n}$, and $\mathsf{A}_{n}$ are all discrete spaces (cf. \cite[\S A.8]{DGK23}). Therefore, for the purpose of discussing the strong identity condition, there is no difference between working with the completion $\widehat{U}$ and with $U$.
  \end{egrmk}

\section{Module categories}\label{sec:modules}
  In this section, we will introduce several categories of modules that will be used in this paper. 
  Throughout this section, let $U$ be a \emph{weakly unital} almost-canonically seminormed algebra (cf. \zcref{def:weakunital}) and $R$ is a ring\footnote{In practice, $R$ is either $\k$ (our working ground) or $\mathsf{A}_{0}$ (the actual base of $U$).}. 
  We thus have the category $\Mod(U|{}_R)$ of $(U|{}_R)$-bimodules and the category $\Mod({}_R|U)$ of $({}_R|U)$-bimodules. 

\subsection{Exhaustive modules}
  We will mainly focus on certain type of $U$-modules.

  \begin{definition}\label{def:ExMod}
    A left $U$-module $W$ is said to be \concept{exhaustive} if it is the union of the subspaces
    \[
      \nota{\Omega^{\mathsf{L}}_{n}(\nocolor{W})}:=\Set*{w\in W\given \NL[n+1]U\cdot w=0}.
    \] 
    Likewise, a right $U$-module $W$ is said to be \concept{exhaustive} if it is the union of 
    \[
      \nota{\Omega^{\mathsf{R}}_{n}(\nocolor{W})}:=\Set*{w\in W\given w\cdot\NR[n+1]U=0}.
    \]
    We will mainly focus on the following categories:
    \begin{description}[leftmargin = 9\parindent]
      \item[\nota{$\Mod[Ex](U|{}_R)$}] 
          the full subcategory of $\Mod(U|{}_R)$ consisting of all bimodules whose underlying left $U$-modules are exhaustive.
      \item[\nota{$\Mod[Ex]({}_R|U)$}] 
          the full subcategory of $\Mod({}_R|U)$ consisting of all bimodules whose underlying right $U$-modules are exhaustive.
    \end{description}
  \end{definition}
  \begin{proposition}[{Exhaustive = discrete continuous}]\label{prop:ExModDiscCont}
    For a left $U$-module $W$, the following are equivalent:
    \begin{enumerate}
      \item $W$ is exhaustive;
      \item the action map $U\otimes W\to W$ is continuous with respect to the seminorm $\NL$ on $U$ and the discrete topology on $W$.
    \end{enumerate}
    In particular, for such a left $U$-module, we have
    \[
      \Omega^{\mathsf{L}}_{n}(W) = \Set{ w \in W \given U_{\le -n-1} \cdot w = 0 }.
    \]
    Similar assertions hold for right $U$-modules.
  \end{proposition}
  \begin{proof}
    The equivalence of the two conditions is straightforward from the definitions. 
    For the last assertion, note that $U_{\le -n-1} = U_{0}U_{\le -n-1} \subset \NL[n+1]U$. Hence, if $w \in \Omega^{\mathsf{L}}_{n}(W)$, then $U_{\le -n-1}\cdot w = 0$. 
    Conversely, if $w\in W$ satisfies $U_{\le -n-1}\cdot w = 0$, then, $\cNL[n+1]U\cdot w = 0$. Since $\cNL[n+1]U$ is dense in $\NL[n+1]U$, we have $\NL[n+1]U\cdot w = 0$ by the continuity of the action map.
  \end{proof}

  \begin{proposition}\label{prop:ExModGrothendieck}
    The categories $\Mod[Ex](U|{}_R)$, $\Mod[Ex](U_{0}|{}_R)$, $\Mod[Ex]({}_R|U)$, and $\Mod[Ex]({}_R|U_{0})$ are Grothendieck abelian categories. 
  \end{proposition}
  \begin{proof}
    We only prove for $\Mod[Ex](U|{}_R)$, the others being similar.

    We refer to \cite[079A]{stacks-project} for Grothendieck's AB conditions. 
    As a standard example, the module category $\Mod(U|{}_R)$ is a Grothendieck abelian category. 
    Since $\Mod[Ex](U|{}_R)$ is a full subcategory of it, we are leading to show:
    \begin{enumerate}
      \item $\Mod[Ex](U|{}_R)$ is closed under taking colimits; and
      \item there is a generator of $\Mod[Ex](U|{}_R)$.
    \end{enumerate}

    We first show that \textbf{$\Mod[Ex](U|{}_R)$ is closed under taking colimits.}
    For this, note that, for an arbitary morphism $f\colon W\to W'$ in $\Mod[Ex](U|{}_R)$, we have
    \begin{equation}
      \label{eq:pf:ExModGrothendieck}\tag{$\divideontimes$}
      f(\Omega^{\mathsf{L}}_{n}(W))\subset\Omega^{\mathsf{L}}_{n}(W').
    \end{equation}
    Now, consider an arbitrary system of exhaustive $(U|{}_R)$-bimodules $(W^i,f_{ij})_{i,j\in I}$ and let $W$ be their colimit in $\Mod(U|{}_R)$ with canonical maps $f_i \colon W^i \to W$. 
    Then, we have $\bigcup_{i\in I}f_i(W^i) = W$. 
    For each $n\in\N$, by \zcref{eq:pf:ExModGrothendieck}, we have 
    \[
      \bigcup_{i\in I}f_i(\Omega^{\mathsf{L}}_{n}(W^i)) \subset \Omega^{\mathsf{L}}_{n}(W).
    \]
    Since each $W^i$ is exhaustive, we must thus have $W=\bigcup_{n\in\N}\Omega^{\mathsf{L}}_{n}(W)$ as desired.

    Next, we give \textbf{a generator of $\Mod[Ex](U|{}_R)$}. Let $\mathsf{g}$ be a generator of $\Mod(U|{}_R)$. 
    Of course, it has not to be exhaustive. 
    But we can modify it as follows:
    note that, for any exhaustive module $W$, there is an epimorphism $\mathsf{g}^{\bigoplus I} \to W$. Hence, $W$ is the union of the images (which are exhaustive submodules of $W$) of each individual composition $\mathsf{g} \hookrightarrow \mathsf{g}^{\bigoplus I} \to W$. 
    From this observation, we see that the coproduct of all exhaustive quotients of $\mathsf{g}$ is a generator of $\Mod[Ex](U|{}_R)$. 
  \end{proof}
  \begin{remark}%\label{rem:UasRalgebra}
    If the set $\Set{\mathfrak{L}^{n}}_{n\in\N}$ belongs to $\Mod(U|{}_R)$, then it is straightforward to check that it generates $\Mod[Ex](U|{}_R)$ from the definition.
    Similar argument holds for $\Mod[Ex]({}_R|U)$.
  \end{remark}

  Note that each subspace $\Omega^{\mathsf{L}}_{n}(W)$ has a natural left action of $\mathsf{A}_{n}$ on it, given by restricting the left action of $U$ to $U_{0}$. The same is true for right modules. 
  We thus have functors
  \begin{equation}\label{eq:OmegaLR}
    \nota{\Omega^{\mathsf{L}}_{n}}\colon\Mod[Ex](U|{}_R)\to\Mod(\mathsf{A}_{n}|{}_R)\txand
    \nota{\Omega^{\mathsf{R}}_{n}}\colon\Mod[Ex]({}_R|U)\to\Mod({}_R|\mathsf{A}_{n}).
  \end{equation}

\subsection{Positively-filtered modules}
  The subspaces $\Omega^{\mathsf{L}}_{n}$ (resp. $\Omega^{\mathsf{R}}_{n}$) form a positive filtration.
  \begin{definition} 
    A \concept{positive filtration} on a left $U$-module $W$ is an increasing sequence of subspaces $W_{\le\bullet}$ of $W$ such that $U_{p} W_{\le n}\subset W_{\le p+n}$ for all $p,n\in\Z$ and $W_{\le n} = 0$ if $n \ll 0$. 
    Likewise\footnote{Note that, the grading on $\opp{U}$ is reversed: $\opp{U}_{\Box} = U_{-\Box}$. Hence, a positive filtration on a right $U$-module $W$ is an increasing sequence of subspaces $W_{\le\bullet}$ of $W$ such that $W_{\le n}U_{p}\subset W_{\le n-p}$ for all $p,n\in\Z$ and $W_{\le n} = 0$ if $n \ll 0$.}, a \concept{positive filtration} on a right $U$-module $W$ is a positive filtration on the left $\opp{U}$-module $W$. 
    These notions are generalized to $(U|{}_R)$-bimodules (resp. $({}_R|U)$-bimodules) in the obvious way: instead of sequences of subspaces, one considers sequences of $R$-submodules.

    By a \concept{positively-filtered} $(U|{}_R)$-bimodule (resp. $({}_R|U)$-bimodule), we mean a $(U|{}_R)$-bimodule (resp. $({}_R|U)$-bimodule) $W$ together \textbf{with} a positive filtration $W_{\le\bullet}$ on it.
    We are thus given the following categories:
    \begin{description}[leftmargin = 9.5\parindent]
      \item[\nota{$\Mod[Fil](U|{}_R)$}] 
          the category of all positively-filtered $(U|{}_R)$-bimodules.
      \item[\nota{$\Mod[Fil]({}_R|U)$}] 
          the category of all positively-filtered $({}_R|U)$-bimodules.
      \item[\nota{$\Mod[Fil]_{0}(U|{}_R)$}]  
          the subcategory consisting of those satisfying $W_{\le n} = 0$ if $n < 0$.
      \item[\nota{$\Mod[Fil]_{0}({}_R|U)$}]  
          the subcategory consisting of those satisfying $W_{\le n} = 0$ if $n < 0$.
    \end{description}
    Note that the morphisms in these categories are required to preserve the filtrations.
  \end{definition}

  Forgetting the filtration gives functors from the categories of positively-filtered modules to those of modules. 
  \begin{proposition}\label{prop:FilModtoExMod}
    The essential image of the forgetting functor from $\Mod[Fil](U|{}_R)$ to $\Mod(U|{}_R)$ is $\Mod[Ex](U|{}_R)$. 
    Furthermore, the composition 
    \[
      \begin{tikzcd}
         \Mod[Fil]_{0}(U|{}_R) \arrow[r, hook] & \Mod[Fil](U|{}_R) \arrow[r, "\fun{forget}"] & \Mod(U|{}_R)
      \end{tikzcd}
    \]
    has a right adjoint given by taking the $\Omega$-filtration. 
    The same holds for $({}_R|U)$-bimodules. 
  \end{proposition}
  \begin{proof}
    If a left $U$-module $W$ admits a positive filtration $W_{\le\bullet}$, we must have $W_{\le n}\subset\Omega_{n}^{\mathsf{L}}(W)$ up to a shift of degrees, and thus it is exhaustive. 
    Conversely, any exhaustive left $U$-module admits a positive filtration where $W_{\le n}:=\Omega_{n}^{\mathsf{L}}(W)$. 
    The unit for the adjunction is the canonical morphism $W_{\le \bullet} \to \Omega_{\bullet}^{\mathsf{L}}(W)$ and the counit is the identity.
    Similar arguments apply to right modules.
  \end{proof}
  \begin{warning}
    Unlike $\Mod[Ex](U|{}_R)$, the category $\Mod[Fil](U|{}_R)$ is not abelian. This can be seen from the fact that one can shift the filtration while keep the underlying bimodule structure.
  \end{warning}

\subsection{Graded modules}
  As we are considering a graded algebra $U$, it is natural to also consider graded modules over it.
  \begin{definition}\label{def:gMod}
    A \concept{grading} on a left (resp. right) $U$-module $W$ is a family of subspaces $W_{\bullet}$ of $W$ such that $U_{p} W_{n}\subset W_{p+n}$ (resp. $W_{n} U_{p}\subset W_{p+n}$) for all $p,n\in\Z$. 
    Such a grading is called \concept{positive} (resp. \concept{negative}\footnote{Note that the conventions on ``positive'' and ``negative'' differ between filtrations and gradings, as filtrations are conventionally required to be increasing.}) if $W_{n}=0$ for all $n \ll 0$ (resp. $W_{n}=0$ for all $n \gg 0$). 
    A \concept{positively-graded} left $U$-module (resp. \concept{negatively-graded} right $U$-module) is a left (resp. right) $U$-module $W$ together \textbf{with} a positive (resp. negative) grading on it. 

    These notions are generalized to $(U|{}_R)$-bimodules (resp. $({}_R|U)$-bimodules) in the obvious way: instead of families of subspaces, one considers families of $R$-submodules.
    We have the following categories:
    \begin{description}[leftmargin = 9\parindent]
      \item[\nota{$\Mod[Gr](U|{}_R)$}]  
          the category of all positively-graded $(U|{}_R)$-bimodules.% with homogeneous morphisms of degree zero.
      \item[\nota{$\Mod[Gr]({}_R|U)$}]  
          the category of all negatively-graded $({}_R|U)$-bimodules.% with homogeneous morphisms of degree zero.
      \item[\nota{$\Mod[Gr]_{d}(U|{}_R)$}]  
          the subcategory consisting of those satisfying\footnote{The meaning of the subscript $d$ is explained by \zcref{eq:degree-0}.} $W_{n} = 0$ if $n < -d$.
      \item[\nota{$\Mod[Gr]_{d}({}_R|U)$}]  
          the subcategory consisting of those satisfying $W_{n} = 0$ if $n > d$.
    \end{description}
    Note that the morphisms in these categories are required to preserve the gradings.
  \end{definition}
  \begin{remark}
    A \emph{negatively-graded} right $U$-module $W$ is equivalent to a \emph{positively-graded} left $\opp{U}$-module $\nota{\opp{W}}$, whose grading is given by $\opp{W}_{\bullet}:=W_{-\bullet}$. 
  \end{remark}

  \begin{example}\label{eg:LR-GrMod}
    For each $n\in\N$, the $(U|\mathsf{A}_{n})$-bimodule $\mathfrak{L}^{n}$ (see \zcref{def:ALR}) is positively-graded with its inherited grading from $U$. Indeed, $\mathfrak{L}^{n}_{p} = 0$ if $p<-n$. Similarly, the $(\mathsf{A}_{n}|U)$-bimodule $\mathfrak{R}^{n}$ is negatively-graded. 
    Note that the canonical projections $\pi\colon \mathfrak{L}^{n} \to \mathfrak{L}^{n-1}$ and $\pi\colon \mathfrak{R}^{n} \to \mathfrak{R}^{n-1}$ preserve the natural gradings.
  \end{example}

  For a graded $U$-module, its each individual component is a $U_{0}$-module. Thus, we have the following \emph{taking degree-$n$} functors:
  \begin{equation}\label{eq:degree-n}
    \nota{(-)_{n}}\colon\Mod[Gr](U|{}_R)\to\Mod(U_{0}|{}_R)\txand
    \nota{(-)_{n}}\colon\Mod[Gr]({}_R|U)\to\Mod({}_R|U_{0}).
  \end{equation}
  In particular, the following \emph{taking degree-zero} functors are of importance:
  \begin{equation}\label{eq:degree-0}
    \nota{(-)_{0}}\colon\Mod[Gr]_{d}(U|{}_R)\to\Mod(\mathsf{A}_{d}|{}_R)\txand
    \nota{(-)_{0}}\colon\Mod[Gr]_{d}({}_R|U)\to\Mod({}_R|\mathsf{A}_{d}).
  \end{equation}

  \begin{notation}
    Any positively-graded $(U|{}_R)$-bimodule $(W,W_{\bullet})$ admits a positive filtration $W_{\le \bullet}$ given by 
    \[
      W_{\nota{\le\bullet}}:=\bigoplus_{n\le\bullet}W_{n}.
    \]
    Conversely, any positively-filtered $(U|{}_R)$-bimodule $(W,W_{\le\bullet})$ induces a positively-graded bimodule $\gr W$ given by
    \[
      \nota{\gr{\nocolor{W}}_{\bullet}} :=
      W_{\le \bullet}/W_{\le \bullet-1}.
    \]

    Similar construction applies to $({}_R|U)$-bimodules. 
    But note that the positive filtration associated to a negatively-grading $W_{\bullet}$ is given by
    \[
      \opp{W}_{\le\bullet} := \bigoplus_{n\le\bullet}\opp{W}_{n}=\bigoplus_{n\le\bullet}W_{-n}.
    \]
    And conversely, from a positively-filtered $({}_R|U)$-bimodule $(W,W_{\le\bullet})$, one gets a positively-graded $(\opp{U}|{}_{\opp{R}})$-bimodule $\gr W$ and hence a negatively-graded $({}_R|U)$-bimodule $\opp{\gr W}$.
  \end{notation}

  \begin{notation}\label{nota:grOmega}
    Recall that any exhaustive $U$-module $W$ admits an $\Omega$-filtration, we call the associated grading (here and what follows, $\Omega_{\bullet}$ stands for either $\Omega^{\mathsf{L}}_{\bullet}$ or $\Omega^{\mathsf{R}}_{\bullet}$, depends on the context)
    \begin{equation}
      \nota{\gr\Omega_{\bullet}(\nocolor{W})}:=\Omega_{\bullet}(W)/\Omega_{\bullet-1}(W)
    \end{equation}
    the \concept{$\gr\Omega$-grading} of $W$.
  \end{notation} 

  \begin{proposition}\label{prop:FilMod-GrMod}
    The functor $(-)_{\le\bullet}\colon\Mod[Gr](U|{}_R)\to\Mod[Fil](U|{}_R)$ is a right inverse of the functor $\gr(-)_{\bullet}\colon\Mod[Fil](U|{}_R)\to\Mod[Gr](U|{}_R)$.
    The same holds for $({}_R|U)$-bimodules.
  \end{proposition}
  \begin{remark}
    Due to well-known facts on graded modules, the categories $\Mod[Gr](U|{}_R)$ and $\Mod[Gr]({}_R|U)$ are abelian categories. 
  \end{remark}

  Combining \zcref{prop:FilModtoExMod,prop:FilMod-GrMod}, we obtain a forgetting functor from positively-graded modules to exhaustive modules. We summarize the relations between various module categories in \zcref{fig:module_categories}.

  \begin{corollary}\label{coro:Omega-vs-grading}
    The underlying module of a positively-graded $(U|{}_R)$-bimodule $W$ is exhaustive. 
    Furthermore, we have 
    \[
      \Omega^{\mathsf{L}}_{n}(W) = \bigoplus_{k\in\Z}\left( \Omega^{\mathsf{L}}_{n}(W) \cap W_{k} \right).
    \]
    The same holds for negatively-graded $({}_R|U)$-bimodules.
  \end{corollary}
  \begin{proof}
    The first follows from \zcref{prop:FilModtoExMod,prop:FilMod-GrMod}. The second is due to the fact that the neighborhood $\NL[n+1]U$ is homogeneous.
  \end{proof}

  \begin{warning}
    Despite the above corollary, the $\gr\Omega$-grading of a positively-graded $U$-module $W$ may differ significantly from its original grading.
  \end{warning}

  \begin{definition}\label{def:gradable}
    An exhaustive $U$-module is said to be \concept{gradable} if it is isomorphic to the underlying module of a (positively or negatively) graded $U$-module.
  \end{definition}
  \begin{remark}
    From the definition, it is not clear whether ``exhaustive'' = ``gradable'' or not.
  \end{remark}

  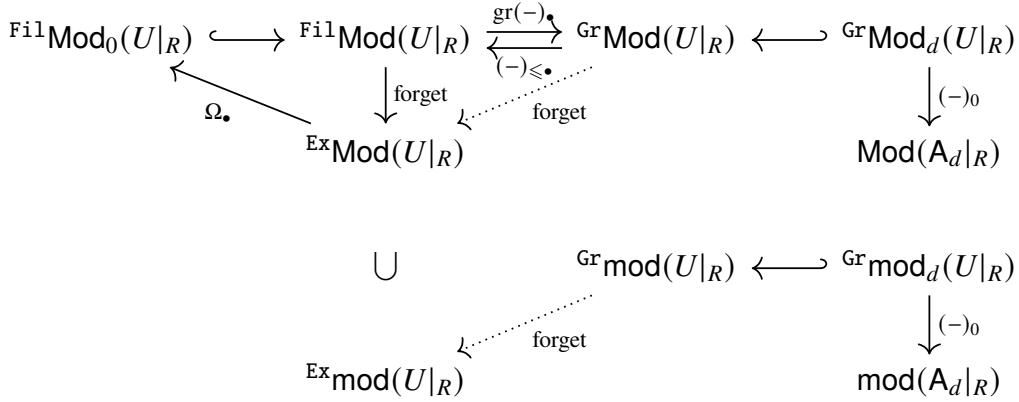
\begin{figure}[htbp]
    \centering
    \begin{tikzcd}
      \Mod[Fil]_{0}(U|{}_R) 
          \arrow[r, hook] &
      \Mod[Fil](U|{}_R) 
          \arrow[r, shift left, "\gr(-)_\bullet"] 
          \arrow[d, "\text{forget}"] &
      \Mod[Gr](U|{}_R) 
          \arrow[l, shift left, "(-)_{\le\bullet}"]
          \arrow[dl, "\text{forget}", dotted] &
      \Mod[Gr]_d(U|{}_R) 
          \arrow[l, hook']
          \arrow[d, "(-)_{0}"] \\
      &
      \Mod[Ex](U|{}_R) 
          \arrow[ul,"\Omega_{\bullet}"] &&
      \Mod(\mathsf{A}_d|{}_R) \\
      && 
      \Mod*[Gr](U|{}_R) 
          \arrow[dl, "\text{forget}", dotted] &
      \Mod*[Gr]_d(U|{}_R)
          \arrow[l, hook'] 
          \arrow[d, "(-)_{0}"] \\
      &
      \Mod*[Ex](U|{}_R) 
          \arrow[uu, phantom, "\bigcup" description] &&
      \Mod*(\mathsf{A}_d|{}_R) 
    \end{tikzcd}
    \caption{Relations between various module categories introduced in \zcref{sec:modules}. The bottom ones are categories of quasi-rigid modules. Similar relations hold for right module categories.}
    \label{fig:module_categories}
  \end{figure}

\subsection{Quasi-rigid modules and duality}\label{sec:ordinary}
  In the context of almost-canonically seminormed algebras, we give the following definitions mimicing objects in $\Mod*(-)$'s.
  \begin{definition}
    \label{def:ordinary}
    A positively-graded $(U|{}_R)$-bimodule $W$ is said to be \concept{quasi-rigid} if each right $R$-submodule $W_{n}$ is rigid. 
    An exhaustive $(U|{}_R)$-bimodule $W$ is said to be \concept{quasi-rigid} if it is the underlying module of a quasi-rigid graded $(U|{}_R)$-bimodule. 
    Similar definitions apply to $({}_R|U)$-bimodules.
    We have the following full subcategories
    \[
      \nota{\Mod*[Ex](U|{}_R)},\,
      \nota{\Mod*[Ex]({}_R|U)},\,
      \nota{\Mod*[Gr](U|{}_R)},\,
      \nota{\Mod*[Gr]({}_R|U)},\,
      \nota{\Mod*[Gr]_{d}(U|{}_R)},\,
      \nota{\Mod*[Gr]_{d}({}_R|U)}
    \]
    consisting of all quasi-rigid bimodules in the corresponding module categories.
  \end{definition}

  The functors in \zcref[noname]{eq:degree-n,eq:degree-0} restrict to the quasi-rigid module categories:
  \begin{align}
    (-)_{n}\colon\Mod*[Gr](U|{}_R)\to\Mod*(U_{0}|{}_R)
    &\qquad
    (-)_{n}\colon\Mod*[Gr]({}_R|U)\to\Mod*({}_R|U_{0}) \label{eq:degree-n-ordinary}\\
    (-)_{0}\colon\Mod*[Gr]_{d}(U|{}_R)\to\Mod*(\mathsf{A}_{d}|{}_R)
    &\qquad
    (-)_{0}\colon\Mod*[Gr]_{d}({}_R|U)\to\Mod*({}_R|\mathsf{A}_{d}) \label{eq:degree-0-ordinary}
  \end{align}
  \begin{remark}
    The functors $\Omega^{\mathsf{L}}_{n}$ and $\Omega^{\mathsf{R}}_{n}$ in \zcref[noname]{eq:OmegaLR} do not automatically restrict to the quasi-rigid module categories $\Mod*[Ex](U|{}_R)\to\Mod*(\mathsf{A}_{n}|{}_R)$ and $\Mod*[Ex]({}_R|U)\to\Mod*({}_R|\mathsf{A}_{n})$.
  \end{remark}

  The following lemma is useful.
  \begin{lemma}\label{lem:rigidmodule}
    Let $A$ and $R$ be two rings. 
    For any right $A$-module $M$ and any $(A|R)$-bimodule $N$, if $M$ is rigid over $A$ and $N$ are rigid over $R$, then $M\otimes_{A}N$ is rigid over $R$.
  \end{lemma}
  \begin{proof}
    Indeed, we have $(M\otimes_{A}N)^{\vee|R} = N^{\vee|R}\otimes_{A}M^{\vee|A}$.
  \end{proof}

% \subsection{Duals of graded modules}
  \begin{definition}\label{def:degreewise-dual} 
    Let $W$ be a graded $(U|{}_R)$-bimodule (resp. $({}_R|U)$-bimodule). 
    Then, its \concept{graded dual module} is the graded left (resp. right) $R$-module $W^{\dagger|R}$ (resp. $W^{R|\dagger}$) with components
    \[
      \nota{W^{\dagger|R}_{\bullet}} := 
      (W_{-\bullet})^{\vee|R}
      \qquad
      \left(
        \text{resp. }
        \nota{W^{R|\dagger}_{\bullet}} :=
        (W_{-\bullet})^{R|\vee}
      \right).
    \]
    When $R=\k$, we simply denote it by $\nota{W'}$. 
  \end{definition}
  \begin{lemma}
    The space $W^{\dagger|R}$ is an $({}_R|U)$-submodule of $W^{\vee|R}$. Likewise, the space $W^{R|\dagger}$ is a $(U|{}_R)$-submodule of $W^{R|\vee}$.
  \end{lemma}
  \begin{proof}
    We need to show that for any $\alpha\in U$ and $\varphi\in W^{\dagger|R}$, the $R$-linear functional $\varphi\cdot\alpha \colon W \to R$ vanishes on all but finitely many components $W_{n}$. We may assume $\alpha\in U_{p}$. Then, we have 
    \[ 
      (\varphi\cdot\alpha)(W_{n})=\varphi(\alpha \cdot W_{n})\subset\varphi(W_{p+n}).
    \]
    Since $\varphi\in W^{\dagger|R}$, the statement follows. 
  \end{proof}

  \begin{recollection}
    Let $M$ be a graded right $R$-module and $N$ be a graded left $R$-module. 
    Their \concept{complete graded tensor product} is the graded space
    \[
      \nota{M \ctensor_{R} N} := 
      \bigoplus_{n\in\Z} \left( \prod_{p+q=n} M_{p} \otimes_{R} N_{q} \right).
    \]
    This can be thought of as the degreewise completion of the algebraic tensor product $M \otimes_{R} N$.

    Note that, when $M$ is positively-graded and $N$ is negatively-graded, the above completion can also be obtained via the following \emph{canonical seminorms}: 
    \[
      \NL[n](M \otimes_{R} N) = \bigoplus_{q \le -n} (M_{p} \otimes_{R} N_{q})
      \txand
      \NR[n](M \otimes_{R} N) = \bigoplus_{p \ge n} (M_{p} \otimes_{R} N_{q}).
    \]
    This is similar to what we did in \zcref{sec:ACSR}.
  \end{recollection}
  \begin{recollection}
    Let $M$ and $N$ be two positively-graded right $R$-modules. 
    Then, the graded $R$-module of homogeneous $R$-linear homomorphisms from $M$ to $N$ is given by
    \[
      \nota{\Hom^{\tt Gr}_{|R}(M,N)} := 
      \bigoplus_{n\in\Z} \Hom_{\Mod[Gr](|R)}(M[n],N).
    \]
  \end{recollection}

  The meaning of \emph{quasi-rigid modules} is justified by the following proposition:
  \begin{proposition}
    \label{prop:ordinary}
    Let $W$ be a positively-graded $(U|{}_R)$-module. 
    Then, $W$ is quasi-rigid if and only if the canonical evaluation map 
    \[
      (-)\ctensor_{R}W^{\dagger|R}
      \longrightarrow
      \Hom^{\tt Gr}_{|R}(W,-)
    \]
    is a natural isomorphism on $\Mod[Gr](U|{}_R)$.
    Similar holds for negatively-graded $({}_R|U)$-modules.
  \end{proposition}
  \begin{proof}
    Since
    \begin{align*}
      (- \ctensor_{R} W^{\dagger|R})_{n} 
      &= \prod_{k\in\Z} \left( (-)_{k} \otimes_{R} W^{\dagger|R}_{n-k} \right) = \prod_{k\in\Z} \left( (-)_{k} \otimes_{R} (W_{k-n})^{\vee|R} \right),\\
      \Hom_{\Mod[Gr](|R)}(W[n],-) 
      &= \prod_{k\in\Z} \Hom_{|R}(W_{k-n},(-)_{k}).
    \end{align*}
    We see that the two are isomorphic if and only if $W$ is quasi-rigid.
  \end{proof}

% \subsection{Dual of exhaustive modules}
  The dual of exhaustive modules is more subtle: the full dual module is usually not exhaustive.
  Despite this, we have the following useful description.
  \begin{lemma}\label{lem:OmegaDual}
    $\Omega^{\mathsf{R}}_{n}(W^{\vee|R}) = (\mathfrak{R}^{n} \otimes_{U} W)^{\vee|R}$ and $\Omega^{\mathsf{L}}_{n}(W^{R|\vee}) = (W \otimes_{U} \mathfrak{L}^{n})^{R|\vee}$.
  \end{lemma}
  \begin{proof}
    We prove the first equality; the second one is similar.
    By the right $U$-action on $W^{\vee|R}$, we see that $\Omega^{\mathsf{R}}_{n}(W^{\vee|R})$ consists of those $R$-linear functionals on $W$ which vanish on $\NR[n+1]U\cdot W$. 
    Hence, it is isomorphic to the space of $R$-linear functionals on the quotient $W/(\NR[n+1]U\cdot W)$. This is precisely $(\mathfrak{R}^{n} \otimes_{U} W)^{\vee|R}$.
  \end{proof}

\section{Induced modules}\label{sec:InducedMod}
  The functors $\Omega^{\mathsf{L}}_{n}$ and $\Omega^{\mathsf{R}}_{n}$ (cf. \zcref{def:ExMod,eq:OmegaLR}) admit left adjoints. 
  To make the construction more transparent, we assume $U$ is unital. For the weakly unital case, one can work with its \emph{unitalization}.

  To state the adjointness, we first consider the following construction on $U_{0}$-modules.
  \begin{definition}\label{def:Phi0}
    Let $M$ be a $(U_{0}|{}_R)$-bimodule. The \concept{induced (bi)module generated by} $M$ is the following $(U|{}_R)$-bimodule 
    \[
      \nota{\Phi^{\mathsf{L}}_{n}(\nocolor{M})}:=\mathfrak{L}^{n}\otimes_{U_{0}}M.
    \]
    We view $\Phi^{\mathsf{L}}_{n}(M)$ as a positively-graded bimodule via the grading on $\mathfrak{L}^{n}$.
    Likewise, for an $({}_R|U_{0})$-bimodule $M$, the \concept{induced (bi)module generated by} $M$ is the following $({}_R|U)$-bimodule
    \[
      \nota{\Phi^{\mathsf{R}}_{n}(\nocolor{M})}:=M\otimes_{U_{0}}\mathfrak{R}^{n}.
    \]
    We view $\Phi^{\mathsf{R}}_{n}(M)$ as a negatively-graded bimodule via the grading on $\mathfrak{R}^{n}$.
  \end{definition}

  By abuse of notations, we will also use $\Phi^{\mathsf{L}}_{n}$ (resp. $\Phi^{\mathsf{R}}_{n}$) to denote the composition of it with the restriction functor from $\Mod(U_{0}|{}_R)$ to $\Mod(\mathsf{A}_{n}|{}_R)$ (resp. from $\Mod({}_R|U_{0})$ to $\Mod({}_R|\mathsf{A}_{n})$).
  \begin{proposition}\label{prop:adj:Phideg}
    The following functors form an adjoint pair:
    \[
      \begin{tikzcd}
        \Phi^{\mathsf{L}}_{n}\colon \Mod(\mathsf{A}_{n}|{}_R) & 
        \Mod[Gr]_{n}(U|{}_R) \colon(-)_{0}.
        \ar[from=1-1, to=1-2, phantom, "\scriptstyle\bot" description]
        \ar[from=1-1, to=1-2, shift left=1ex] 
        \ar[from=1-2, to=1-1, shift left=1ex] 
      \end{tikzcd}
    \]
    The similar holds for the pair $\Phi^{\mathsf{R}}_{n}$ and $(-)_{0}$. 
  \end{proposition}
  \begin{proof}
    The unit map $\eta_{M}\colon M\to(\Phi^{\mathsf{L}}_{0}(M))_{0}$ is given by $m\mapsto [1_{n}]^{\mathsf{L}}_{n}\otimes m$, which is clearly an isomorphism.
    The counit map $\Phi^{\mathsf{L}}_{n}(W_{0})\to W$ is given by the $U$-action map $[\alpha]^{\mathsf{L}}_{n}\otimes w\mapsto \alpha w$. Note that this map is homogeneous since $U_{\bullet}\cdot W_{0}\subset W_{\bullet}$.  
  \end{proof}

  By abuse of notation, we will also use $\Phi^{\mathsf{L}}_{n}$ (resp. $\Phi^{\mathsf{R}}_{n}$) to denote the composition of it with the forgetful functor from $\Mod[Gr](U|{}_R)$ to $\Mod[Ex](U|{}_R)$ (resp. from $\Mod[Gr]({}_R|U)$ to $\Mod[Ex]({}_R|U)$).
  \begin{proposition}\label{prop:adj:PhiOmega}
    The following functors form an adjoint pair:
    \[
      \begin{tikzcd}
        \Phi^{\mathsf{L}}_{n}\colon \Mod(\mathsf{A}_{n}|{}_R) & 
        \Mod[Ex](U|{}_R)
        \colon\Omega^{\mathsf{L}}_{n}.
        \ar[from=1-1, to=1-2, phantom, "\scriptstyle\bot" description]
        \ar[from=1-1, to=1-2, shift left=1ex] 
        \ar[from=1-2, to=1-1, shift left=1ex] 
      \end{tikzcd}
    \]
    The similar holds for the pair $\Phi^{\mathsf{R}}_{n}$ and $\Omega^{\mathsf{R}}_{n}$. 
  \end{proposition}
  \begin{proof}
    The proof is similar to that of \zcref{prop:adj:Phideg}. 
    We also provide another way to see this adjointness. The key observation is that $\Omega^{\mathsf{L}}_{n}(W) = \Hom_{U|}(\mathfrak{L}^{n}, W)$. Then, the adjunction is a consequence of the Tensor-Hom adjunction for the $(U|\mathsf{A}_{n})$-bimodule $\mathfrak{L}^{n}$.
  \end{proof}
  \begin{remark}\label{rem:OmegaHomLR}
    From the proof, we see that 
    \[
      \Omega^{\mathsf{L}}_{n}(-) = \Hom_{U|}(\mathfrak{L}^{n}, -)
      \txand 
      \Omega^{\mathsf{R}}_{n}(-) = \Hom_{|U}(\mathfrak{R}^{n}, -)
    \]
    as functors on $\Mod(U|{}_R)$ and $\Mod({}_R|U)$ respectively.
  \end{remark}

  The induction functors restrict to the quasi-rigid module categories. 
  \begin{proposition}\label{prop:adj:ordinary}
    Let $n\in\N$. Suppose either 
    \begin{enumerate}
      \item $\mathfrak{L}^{n} \in \Mod*[Gr](U|{}_{\mathsf{A}_{n}})$ and $\mathfrak{R}^{n} \in \Mod*[Gr]({}_{\mathsf{A}_{n}}|U)$; or 
      \item $R$ is semisimple and components of $\mathfrak{L}^{n},\mathfrak{R}^{n}$ are finite over $\mathsf{A}_{n}$.
    \end{enumerate}
    Then, the adjoint pairs in \zcref{prop:adj:Phideg} restrict to those between the quasi-rigid module categories:
    \[
      \begin{tikzcd}[row sep=tiny]
        \Phi^{\mathsf{L}}_{n}\colon \Mod*(\mathsf{A}_{n}|{}_R) & 
        \Mod*[Gr]_{n}(U|{}_R)  \colon(-)_{0}.
        \ar[from=1-1, to=1-2, phantom, "\scriptstyle\bot" description]
        \ar[from=1-1, to=1-2, shift left=1ex] 
        \ar[from=1-2, to=1-1, shift left=1ex] 
      \end{tikzcd}
    \]
    The similar holds for the pair $({}_R|U)$. 
  \end{proposition}
  \begin{proof}
    We need to show: if an $(\mathsf{A}_{n}|{}_R)$-bimodule $M$ is rigid over $R$, then the graded bimodule $\Phi^{\mathsf{L}}_{n}(M)$ is quasi-rigid. 
    In case (i), this is due to \zcref{lem:rigidmodule}.
    In case (ii), since $R$ is semisimple, being rigid over $R$ is equivalent to be finite over $R$. Then, the finiteness of $\Phi^{\mathsf{L}}_{n}(M)$ follows from the finiteness of $\mathfrak{L}^{n}$.
  \end{proof}

  \begin{definition}\label{def:quasi}
    Mimicking the notion of quasi-finiteness in \cite{MNT10}, we propose to call $U$ \concept{quasi-rigid} if the assumption (i) of \zcref{prop:adj:ordinary} is satisfied, and \concept{weakly quasi-finite} if (ii) holds.
    We call it \emph{weak} since being finite over $\mathsf{A}_{n}$ is weaker than being finite over $\k$. 
  \end{definition}

\subsection{Thick Zhu algebras and thick induced modules}\label{sec:ThickZhu}
  The following construction may be thought of as a \emph{thick} variant of the Zhu algebra.
  \begin{definition}\label{def:thickZhu}
    The $n$-th left and right \concept{thick Zhu algebras} are 
    \[
      \nota{\mathscr{A}^{\mathsf{L}}_{n}} := \Omega^{\mathsf{L}}_{n}(\mathfrak{L}^{n})
      \txand 
      \nota{\mathscr{A}^{\mathsf{R}}_{n}} := \Omega^{\mathsf{R}}_{n}(\mathfrak{R}^{n}).
    \]
  \end{definition}
  \begin{lemma}\label{lem:ThickZhu}
    The aboves are rings and 
    $\mathfrak{L}^{n}$ is a $(U|\mathscr{A}^{\mathsf{L}}_{n})$-bimodule, $\mathfrak{R}^{n}$ is an $(\mathscr{A}^{\mathsf{R}}_{n}|U)$-bimodule.
  \end{lemma}
  \begin{proof}
    First, we show that $1_{n} \in \Omega^{\mathsf{L}}_{n}(\mathfrak{L}^{n})$. 
    Indeed, for any $u \in \NL[n+1]U$, we have $u\cdot 1 = u \in \NL[n+1]U$. Hence, $u\cdot 1_{n} = [u]^{\mathsf{L}}_{n} = 0$.

    Next, we show that
    the multiplication of $\mathscr{A}^{\mathsf{L}}_{n}$ and 
    the right action of $\mathscr{A}^{\mathsf{L}}_{n}$ on $\mathfrak{L}^{n}$ given by the multiplication of $U$ is well-defined. 
    Note that $\mathscr{A}^{\mathsf{L}}_{n} = \Set*{ u\in U \given \NL[n+1]\cdot u \subset \NL[n+1]U }/\NL[n+1]U$. 
    For any $[u]^{\mathsf{L}}_{n}, [v]^{\mathsf{L}}_{n} \in \mathscr{A}^{\mathsf{L}}_{n}$ and $[\alpha]^{\mathsf{L}}_{n} \in \mathfrak{L}^{n}$, we have
    \begin{itemize}
      \item When $u \in \NL[n+1]U$, we have $v\cdot u \in \NL[n+1]U$ and $\alpha \cdot u \in \NL[n+1]U$ since $\NL[n+1]U$ is a left ideal.
      \item When $v \in \NL[n+1]U$, we have $v\cdot u \in \NL[n+1]U$ since $u$ satisfies $\NL[n+1]\cdot u \subset \NL[n+1]U$.
      \item When $\alpha \in \NL[n+1]U$, we have $\alpha \cdot u \in \NL[n+1]U$ since $u$ satisfies $\NL[n+1]\cdot u \subset \NL[n+1]U$.
    \end{itemize}
    This finishes verifying the well-definedness. The associativity of the multiplication of $\mathscr{A}^{\mathsf{L}}_{n}$ and the compatibility of the right action of $\mathscr{A}^{\mathsf{L}}_{n}$ on $\mathfrak{L}^{n}$ with its left $U$-action follow from the associativity of the multiplication of $U$.

    Same arguments apply to $\mathscr{A}^{\mathsf{R}}_{n}$ and $\mathfrak{R}^{n}$.
  \end{proof}
  \begin{lemma}
    For any left $U$-module $W$, the subspace $\Omega^{\mathsf{L}}_{n}(W)$ is naturally a left $\mathscr{A}^{\mathsf{L}}_{n}$-module.
    Similarly for right modules.
  \end{lemma}
  \begin{proof}
    This follows from $\mathscr{A}^{\mathsf{L}}_{n} = \Set*{ u\in U \given \NL[n+1]\cdot u \subset \NL[n+1]U }/\NL[n+1]U$.
  \end{proof}

  So, we have functors similar to \zcref[noname]{eq:OmegaLR}:
  \begin{equation}
    \nota{\Omega^{\mathsf{L}}_{n}}\colon\Mod[Ex](U|{}_R)\to\Mod(\mathscr{A}^{\mathsf{L}}_{n}|{}_R)\txand
    \nota{\Omega^{\mathsf{R}}_{n}}\colon\Mod[Ex]({}_R|U)\to\Mod({}_R|\mathscr{A}^{\mathsf{R}}_{n}).
  \end{equation}
  As before, these functors admit left adjoints given by the following \emph{thick induced module} construction.
  \begin{definition}\label{def:thickPhi}
    Let $M$ be a $(\mathscr{A}^{\mathsf{L}}_{n}|{}_R)$-bimodule. The \concept{thick induced (bi)module generated by} $M$ is the following $(U|{}_R)$-bimodule 
    \[
      \nota{\mathbf{\Phi}^{\mathsf{L}}_{n}(\nocolor{M})}:=\mathfrak{L}^{n}\otimes_{\mathscr{A}^{\mathsf{L}}_{n}}M.
    \]
    Likewise, for an $({}_R|\mathscr{A}^{\mathsf{R}}_{n})$-bimodule $M$, the \concept{thick induced (bi)module generated by} $M$ is the following $({}_R|U)$-bimodule
    \[
      \nota{\mathbf{\Phi}^{\mathsf{R}}_{n}(\nocolor{M})}:=M\otimes_{\mathscr{A}^{\mathsf{R}}_{n}}\mathfrak{R}^{n}.
    \]
  \end{definition}

  \begin{proposition}\label{prop:adj:bPhiOmega}
    The following functors form an adjoint pair:
    \[
      \begin{tikzcd}
        \mathbf{\Phi}^{\mathsf{L}}_{n}\colon \Mod(\mathscr{A}^{\mathsf{L}}_{n}|{}_R) & 
        \Mod[Ex](U|{}_R) \colon\Omega^{\mathsf{L}}_{n}.
        \ar[from=1-1, to=1-2, phantom, "\scriptstyle\bot" description]
        \ar[from=1-1, to=1-2, shift left=1ex] 
        \ar[from=1-2, to=1-1, shift left=1ex] 
      \end{tikzcd}
    \]
    The similar holds for the pair $\mathbf{\Phi}^{\mathsf{R}}_{n}$ and $\Omega^{\mathsf{R}}_{n}$. 
  \end{proposition}
  \begin{proof}
    This follows from the same argument as in \zcref{prop:adj:PhiOmega}.
  \end{proof}

  We summarize the adjoint pairs we have constructed in this section in \zcref{fig:adjoints}.

  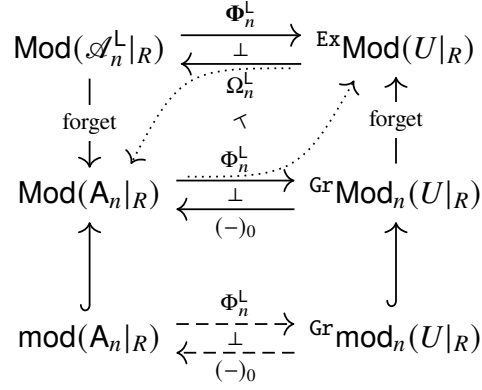
\begin{figure}[htbp]
    \centering
    \begin{tikzcd}[sep = large]
      \Mod(\mathscr{A}^{\mathsf{L}}_{n}|{}_R) &
      \Mod[Ex](U|{}_R) \\
      \Mod(\mathsf{A}_{n}|{}_R) &
      \Mod[Gr]_{n}(U|{}_R) \\
      \Mod*(\mathsf{A}_{n}|{}_R) &
      \Mod*[Gr]_{n}(U|{}_R)
      %%%% Arrows %%%%
      \arrow[from=1-1, to=1-2, phantom, "\scriptstyle\bot" description]
      \arrow[from=1-1, to=1-2, shift left=1ex, "\mathbf{\Phi}^{\mathsf{L}}_{n}"]
      \arrow[from=1-2, to=1-1, shift left=1ex, "\Omega^{\mathsf{L}}_{n}"]
      \arrow[from=2-1, to=2-2, phantom, "\scriptstyle\bot" description]
      \arrow[from=2-1, to=2-2, shift left=1ex, "\Phi^{\mathsf{L}}_{n}"]
      \arrow[from=2-2, to=2-1, shift left=1ex, "(-)_{0}"]
      \arrow[from=3-1, to=3-2, phantom, "\scriptstyle\bot" description]
      \arrow[from=3-1, to=3-2, shift left=1ex, dashed, "\Phi^{\mathsf{L}}_{n}"]
      \arrow[from=3-2, to=3-1, shift left=1ex, dashed, "(-)_{0}"]
      \arrow[from=1-1, to=2-1, "\text{forget}" description]
      \arrow[from=2-2, to=1-2, "\text{forget}" description]
      \arrow[from=2-1, to=1-2, phantom, sloped, "\scriptstyle\top" description]
      \arrow[from=2-1, to=1-2, between={0.05}{0.95}, shift left=1ex, curve={height=25pt,pos=0.5}, dotted]
      \arrow[from=1-2, to=2-1, between={0.05}{0.95}, shift left=1ex, curve={height=25pt,pos=0.5}, dotted]
      \arrow[from=3-1, to=2-1, hook]
      \arrow[from=3-2, to=2-2, hook]
    \end{tikzcd}
    \caption{Adjoint pairs between module categories. The dotted one is given by compositions with forgetting functors. The dashed one only exists under the assumptions in \zcref{prop:adj:ordinary}. A similar diagram holds for the right module categories.} 
    \label{fig:adjoints}
  \end{figure}

\part{Characterizations of the strong identity condition}\label{part2}
  In this part, we study various characterizations of the \emph{strong identity condition} (\SIC).
  Here, we give an outline. 

  For an almost-canonically seminormed algebra $U$, we will consider the following conditions:
  \begin{enumerate}[leftmargin=5\parindent]
    \item[\textup{(\textbf{SIC})}] 
        Its mode transition algebra $\mathfrak{A}$ (see \zcref{def:mode-transition-algebra}) has a family of diagonal elements $\one_{n}\in\mathfrak{A}_{n,-n}$ verifying the \emph{strong identity condition} (see \zcref{def:SIC}):
        \[
          \one_{n}\star\mathfrak{a}=\mathfrak{a}=\mathfrak{a}\star\one_{m}
          \txforall\mathfrak{a}\in\mathfrak{A}_{n,-m} \,(n,m\in\N).
        \]
    \item[\textup{(\textbf{SIE})}] 
        The identity $1_{U}$ can be expressed as an orthogonal series 
        \[
          1_{U} = \sum_{n=0}^{\infty}e_{n} \in \widehat{U_{0}},
        \] 
        called a \emph{strong identity expansion} (see \zcref{def:SIE}).
    \item[\textup{(\textbf{$\Omega$split})}] 
        The \emph{$\Omega$-filtrations} on \emph{exhaustive} $U$-modules (see \zcref{def:ExMod}) split and the canonical projections 
        \begin{align*}
          \cdots \longrightarrow 
          \mathfrak{L}^{n} \overset{\pi_{n}}{\longrightarrow} 
          \mathfrak{L}^{n-1} \longrightarrow 
          \cdots \longrightarrow 
          \mathfrak{L}^{0} \longrightarrow 
          0 \\
          \cdots \longrightarrow 
          \mathfrak{R}^{n} \overset{\pi_{n}}{\longrightarrow} 
          \mathfrak{R}^{n-1} \longrightarrow 
          \cdots \longrightarrow 
          \mathfrak{R}^{0} \longrightarrow 
          0
        \end{align*} 
        (see \zcref{def:ALR} for the definition) respect the splittings in the sense that they preserve the $\gr\Omega$-gradings.
    \item[\textup{(\textbf{$\mathfrak{L}$Proj})}] 
        The left $U$-modules $\mathfrak{L}^{n}$ are projective in the category of exhaustive $U$-modules.
    \item[\textup{(\textbf{$\mathfrak{R}$Proj})}] 
        The right $U$-modules $\mathfrak{R}^{n}$ are projective in the category of exhaustive $U$-modules.
    \item[\textup{(\textbf{$\Phi\Omega\mathsf{L}_{R}$})}]
        The adjoint pair 
        \[
          \begin{tikzcd}
            \Phi^{\mathsf{L}}_{0}\colon \Mod(\mathsf{A}_{0}|{}_R) & 
            \Mod[Ex](U|{}_R)
            \colon\Omega^{\mathsf{L}}_{0}.
            \ar[from=1-1, to=1-2, phantom, "\scriptstyle\bot" description]
            \ar[from=1-1, to=1-2, shift left=1ex] 
            \ar[from=1-2, to=1-1, shift left=1ex] 
          \end{tikzcd}
        \]
        (see \zcref{def:Phi0} for the \emph{induced module} functor $\Phi^{\mathsf{L}}_{0}$)
        is an equivalence. 
        Here, the letter $R$ in the subscript refers to a ring.
    \item[\textup{(\textbf{$\Phi\Omega\mathsf{R}_{R}$})}]
        The adjoint pair 
        \[
          \begin{tikzcd}
            \Phi^{\mathsf{R}}_{0}\colon \Mod({}_R|\mathsf{A}_{0}) & 
            \Mod[Ex]({}_R|U)
            \colon\Omega^{\mathsf{R}}_{0}.
            \ar[from=1-1, to=1-2, phantom, "\scriptstyle\bot" description]
            \ar[from=1-1, to=1-2, shift left=1ex] 
            \ar[from=1-2, to=1-1, shift left=1ex] 
          \end{tikzcd}
        \]
        is an equivalence.
    \item[\textup{(\textbf{$\mathbf{\Phi}\Omega\mathsf{L}^{n}_{R}$})}]
        The adjoint pair 
        \[
          \begin{tikzcd}
            \mathbf{\Phi}^{\mathsf{L}}_{n}\colon \Mod(\mathscr{A}^{\mathsf{L}}_{n}|{}_R) & 
            \Mod[Ex](U|{}_R) \colon\Omega^{\mathsf{L}}_{n}.
            \ar[from=1-1, to=1-2, phantom, "\scriptstyle\bot" description]
            \ar[from=1-1, to=1-2, shift left=1ex] 
            \ar[from=1-2, to=1-1, shift left=1ex] 
          \end{tikzcd}
        \] 
        (see \zcref{def:thickZhu} for the \emph{thick Zhu algebra} $\mathscr{A}^{\mathsf{L}}_{n}$ and \zcref{def:thickPhi} for the \emph{thick induced module} functor $\mathbf{\Phi}^{\mathsf{L}}_{n}$)
        is an equivalence. 
    \item[\textup{(\textbf{$\mathbf{\Phi}\Omega\mathsf{R}^{n}_{R}$})}]
        The adjoint pair 
        \[
          \begin{tikzcd}
            \mathbf{\Phi}^{\mathsf{R}}_{n}\colon \Mod({}_R|\mathscr{A}^{\mathsf{R}}_{n}) & 
            \Mod[Ex]({}_R|U) \colon\Omega^{\mathsf{R}}_{n}.
            \ar[from=1-1, to=1-2, phantom, "\scriptstyle\bot" description]
            \ar[from=1-1, to=1-2, shift left=1ex] 
            \ar[from=1-2, to=1-1, shift left=1ex] 
          \end{tikzcd}
        \] 
        is an equivalence.
  \end{enumerate}
  Then, we have the following unconditional implications:
  
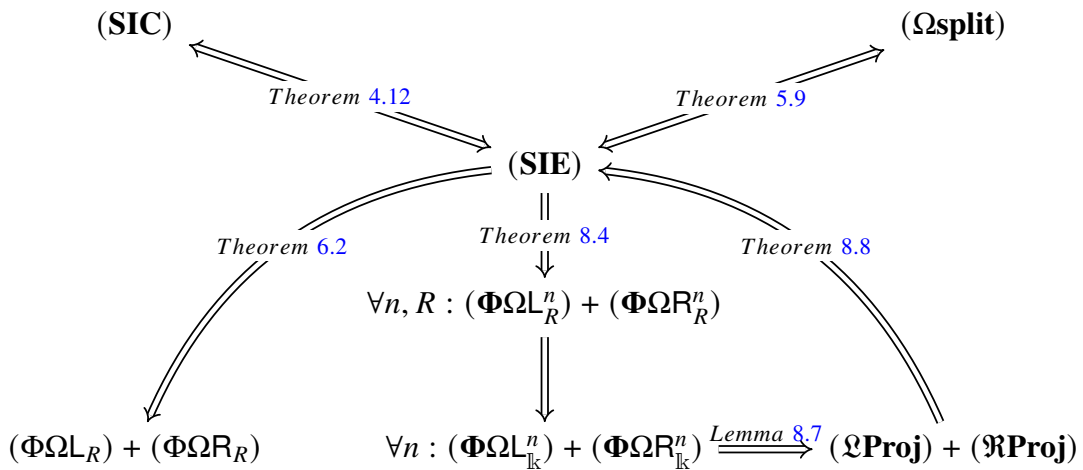
\[
    \begin{tikzcd}[row sep=large]
      (\mathbf{SIC}) & & 
      (\Omega\mathbf{split}) \\
      &
      (\mathbf{SIE}) 
      & \\
      &
      \forall n, R: 
      (\mathbf{\Phi}\Omega\mathsf{L}^{n}_{R}) + (\mathbf{\Phi}\Omega\mathsf{R}^{n}_{R})
      & \\
      (\Phi\Omega\mathsf{L}_{R}) + 
      (\Phi\Omega\mathsf{R}_{R})
      &
      \forall n: 
      (\mathbf{\Phi}\Omega\mathsf{L}^{n}_{\k}) + (\mathbf{\Phi}\Omega\mathsf{R}^{n}_{\k})
      &
      (\mathfrak{L}\mathbf{Proj}) + 
      (\mathfrak{R}\mathbf{Proj})
      \ar[from=1-1, to=2-2, Leftrightarrow, "{\zcref{thm:SIC-SIE}}"description]
      \ar[from=2-2, to=1-3, Leftrightarrow, "{\zcref{thm:SIE-Omega}}"description]
      \ar[from=2-2, to=4-1, Rightarrow, bend right, "{\zcref{thm:PhiOmega0}}"description]
      \ar[from=2-2, to=3-2, Rightarrow, "{\zcref{thm:bPhiOmega}}"description]
      \ar[from=4-2, to=4-3, Rightarrow, "{\zcref{lem:bPhiOmega-Proj}}"]
      \ar[from=4-3, to=2-2, Rightarrow, bend right, "{\zcref{thm:Proj-SIE}}"description]
      \ar[from=3-2, to=4-2, Rightarrow]
    \end{tikzcd}
  \]
  and the following conditional implications:
  
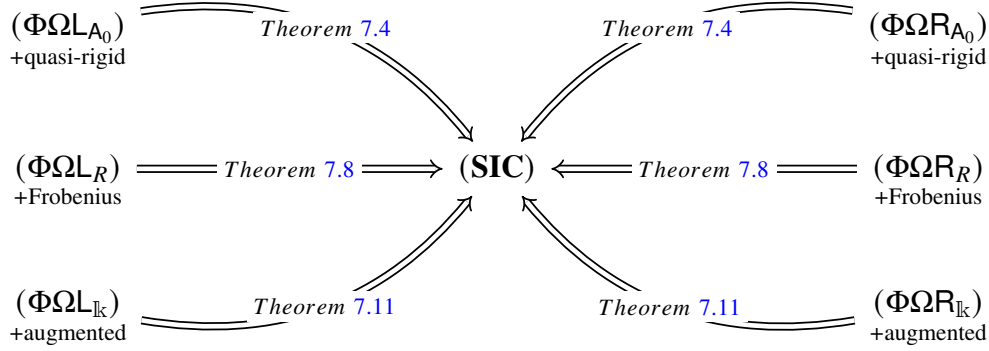
\[
    \begin{tikzcd}[column sep = 4cm]
      \underset{+\text{quasi-rigid}}{(\Phi\Omega\mathsf{L}_{\mathsf{A}_{0}})} & & 
      \underset{+\text{quasi-rigid}}{(\Phi\Omega\mathsf{R}_{\mathsf{A}_{0}})} \\ 
      \underset{+\text{Frobenius}}{(\Phi\Omega\mathsf{L}_{R})} & 
      (\mathbf{SIC}) &
      \underset{+\text{Frobenius}}{(\Phi\Omega\mathsf{R}_{R})} \\
      \underset{+\text{augmented}}{(\Phi\Omega\mathsf{L}_{\k})} & & 
      \underset{+\text{augmented}}{(\Phi\Omega\mathsf{R}_{\k})}
      \ar[from=1-1, to=2-2, Rightarrow, "{\zcref{thm:Adjeq-SIC}}"description, bend left]
      \ar[from=1-3, to=2-2, Rightarrow, "{\zcref{thm:Adjeq-SIC}}"description, bend right]
      \ar[from=2-1, to=2-2, Rightarrow, "{\zcref{thm:AdjeqR-SIC}}"description]
      \ar[from=2-3, to=2-2, Rightarrow, "{\zcref{thm:AdjeqR-SIC}}"description]
      \ar[from=3-1, to=2-2, Rightarrow, "{\zcref{thm:Adjeqk-SIC}}"description, bend right]
      \ar[from=3-3, to=2-2, Rightarrow, "{\zcref{thm:Adjeqk-SIC}}"description, bend left]
    \end{tikzcd}
  \]

\section{Strong identity expansion}\label{sec:SIE}
  In this section, we assume $U$ is unital. 
\subsection{Decomposition via the mode transition algebra}
  The following proposition can be found in \cite[\S B.3]{DGK23}.
  \begin{proposition}\label{prop:decomposition_of_fA}
    Let $\mu$ denote the map $\mathfrak{A}_{n}\to\mathsf{A}_{n}\colon[\alpha]^{\mathsf{L}}_{0}\otimes[\beta]^{\mathsf{R}}_{0}\mapsto[\alpha\beta]_{n}$. Then, the sequence
    \begin{equation}\label{eq:ES:fA}
      \mathfrak{A}_{n}\overset{\mu}{\longrightarrow}\mathsf{A}_{n}\overset{\pi}{\longrightarrow}\mathsf{A}_{n-1}\longrightarrow0
    \end{equation}
    is exact. Furthermore, if the strong identity condition is verified, then $\mu$ is injective and the sequence splits, giving a ring product $\mathsf{A}_{n}\cong\mathfrak{A}_{n}\times\mathsf{A}_{n-1}$.
  \end{proposition}
  Repeating the above ring decomposition, we have 
  \begin{equation}\label{eq:decomposition-fA}
    \widehat{U_{0}}=\varprojlim \mathsf{A}_{n} \cong \prod_{n\in\N}\mathfrak{A}_{n},
  \end{equation}
  with the product topology. 
  That is to say, an element of $\widehat{U_{0}}$ can be expressed as a sequence $(\alpha_{n})_{n\in\N}$, where $\alpha_{n}\in\mathfrak{A}_{n}$ for each $n\in\N$.
  In particular, considering the sequences whoses $n$-th component is the strong identity element $\one_{n}$ for a fixed $n\in\N$, and all the other components are zero, we find elements $e_{n}\in\widehat{U_{0}}$ such that
  \begin{enumerate}[label=(\texttt{si}\arabic*)]
    \item $\sum\limits_{n=0}^{\infty}e_{n}$ converges to $1\in\widehat{U_{0}}$.
    \item $\Set{e_{n}}_{n\in\N}$ are orthogonal to each other.   
  \end{enumerate}
  \begin{remark}\label{rem:Nn}
    Suppose we are given a sequence $(e_{n})_{n\in\N}$ of elements in $\widehat{U_{0}}$ satisfying the above two conditions. Then, the convergence of $\sum\limits_{n=0}^{\infty}e_{n}$ implies, in particular, that for every $n\in\N$, the tail sum $\sum\limits_{k=N}^{\infty}e_{k}$ belongs to the $n$-th neighborhood $\NLR[n]\widehat{U_{0}}$ for all sufficiently large $N$. We denote by $N_{n}$ the minimal such $N$. Clearly, $N_{n}\ge n$ for all $n\in\N$.
  \end{remark}
  \begin{remark}\label{rem:GTC_hatU0}
    In categorical language, $\widehat{U_{0}}$ is a \emph{pro-object} of algebras in the underlying category. The series $\sum\limits_{n=0}^{\infty}e_{n}$ should be understood as the sequence $(e_{\le n}:=\sum\limits_{k=0}^{n}e_{k})_{n\in\N}$ of partial sums.
  \end{remark}

  \zcref{prop:decomposition_of_fA} has the following generalization (recall definitions of the $(U|\mathsf{A}_{n})$-bimodule $\mathfrak{L}^{n}$ and the $(\mathsf{A}_{n}|U)$-bimodule $\mathfrak{R}^{n}$ in \eqref{def:ALR}).
  \begin{theorem}\label{thm:decomposition_of_fL}
    Let $\mu$ denote either the map $\mathfrak{A}_{\bullet+n,-n}\to\mathfrak{L}^{n}_{\bullet}$ or the map $\mathfrak{A}_{n,\bullet-n}\to\mathfrak{R}^{n}_{\bullet}$ given by the formula $[\alpha]^{\mathsf{L}}_{0}\otimes[\beta]^{\mathsf{R}}_{0}\mapsto[\alpha\beta]_{n}$. Then, the sequences
    \begin{equation}\label{eq:ES:fL}
      \mathfrak{A}_{\bullet+n,-n}\overset{\mu}{\longrightarrow}\mathfrak{L}^{n}_{\bullet}\overset{\pi}{\longrightarrow}\mathfrak{L}^{n-1}_{\bullet}\longrightarrow0
      \txand
      \mathfrak{A}_{n,\bullet-n}\overset{\mu}{\longrightarrow}\mathfrak{R}^{n}_{\bullet}\overset{\pi}{\longrightarrow}\mathfrak{R}^{n-1}_{\bullet}\longrightarrow0
    \end{equation}
    are exact. Furthermore, if the strong identity condition is verified, then $\mu$ are injective and the sequences split, giving module products $\mathfrak{L}^{n}_{\bullet}\cong\mathfrak{A}_{\bullet+n,-n}\times\mathfrak{L}^{n-1}_{\bullet}$ and $\mathfrak{R}^{n}_{\bullet}\cong\mathfrak{A}_{n,\bullet-n}\times\mathfrak{R}^{n-1}_{\bullet}$.
  \end{theorem}
  \begin{remark}
    In viewing of \zcref{eq:ModeDecomposition}, the above theorem can be thought of a generalization of \cite[Theorem 4.16]{DJmz}. A more faithful generalization is given by combining with \zcref{eq:Mode-End-Phi}.
  \end{remark}
  \begin{proof}
    The statements for $\mathfrak{L}$ and $\mathfrak{R}$ are dual. We only focus on $\mathfrak{L}$. 

    We first verify that the map $\mu$ is well-defined. This blows down to verify:
    \begin{enumerate}
      \item $\NL[1]U_{\bullet+n}\cdot U_{-n}\subset\NL[n+1]U_{\bullet}$.
      \item $U_{\bullet+n}\cdot \NR[1]U_{-n}\subset\NR[\bullet+n+1]U_{\bullet}=\NL[n+1]U_{\bullet}$.
      \item For all $u\in U_{0}$, $\alpha\in U_{\bullet+n}$, and $\beta\in U_{-n}$, we have 
      \[
        \mu([\alpha u]^{\mathsf{L}}_{0}\otimes[\beta]^{\mathsf{R}}_{0}) = \mu([\alpha]^{\mathsf{L}}_{0}\otimes[u\beta]^{\mathsf{R}}_{0}).
      \]
    \end{enumerate}
    All of them are evident.

    For the exactness, note that the image of $\mu$ is spanned by $[\alpha\beta]^{\mathsf{L}}_{n}$ with $\deg\alpha=\bullet+n$ and $\deg\beta=-n$. Hence, the image of $\mu$ consists of $\cNL[n]U_{\bullet}\mod\NL[n+1]U$. 
    On the other hand, the kernel of $\pi$ is $\NL[n]U_{\bullet}\mod\NL[n+1]U$. Since $\NL[n]U_{\bullet}=\cNL[n]U_{\bullet}+\NL[n+1]U_{\bullet}$, the sequence is exact.

    Before moving on, we need a lemma. 
    \begin{lemma}\label{lem:Act_of_fA_via_A}
      For any $U_0$-module $M$, the action map $\mathfrak{A}_{\bullet+n,-n}\times\Phi^{\mathsf{L}}_{0}(M)_{n}\to\Phi^{\mathsf{L}}_{0}(M)_{\bullet+n}$ induced from the action of $\mathfrak{A}$, factors through the action map $\mathfrak{L}^{n}_{\bullet}\times\Phi^{\mathsf{L}}_{0}(M)_{n}\to\Phi^{\mathsf{L}}_{0}(M)_{\bullet+n}$ induced from the action of $U$, via $\mu$.
    \end{lemma}
    \begin{proof}
      Recall that the action of $\mathfrak{A}$ on $\Phi^{\mathsf{L}}_{0}(M)$ is given by 
      \[
        ([\alpha]^{\mathsf{L}}_{0}\otimes[\alpha']^{\mathsf{R}}_{0})\star([\beta]^{\mathsf{L}}_{0}\otimes m)=[\alpha]^{\mathsf{L}}_{0}\otimes ([\alpha']^{\mathsf{R}}_{0}\ostar[\beta]^{\mathsf{L}}_{0})m.
      \]
      This applies specially to $\mathfrak{A}_{\bullet+n,-n}\times\Phi^{\mathsf{L}}_{0}(M)_{n}\to\Phi^{\mathsf{L}}_{0}(M)_{\bullet+n}$ yielding
      \[
        [\alpha]^{\mathsf{L}}_{0}\otimes ([\alpha']^{\mathsf{R}}_{0}\ostar[\beta]^{\mathsf{L}}_{0})m = [\alpha]^{\mathsf{L}}_{0}\otimes ([\alpha'\beta]_{0})m = [\alpha\alpha'\beta]^{\mathsf{L}}_{0}\otimes m.
      \]
    
      On the other hand, the action of $U$ on $\Phi^{\mathsf{L}}_{0}(M)$ induces a map $\mathfrak{L}^{n}_{\bullet}\times\Phi^{\mathsf{L}}_{0}(M)_{n}\to\Phi^{\mathsf{L}}_{0}(M)_{\bullet+n}$ since $\NL[n+1]U_n\cdot U_{n}\subset\NL[1]U_{\bullet+n}$. Under this map, we have
      \[
        \mu([\alpha]^{\mathsf{L}}_{0}\otimes[\alpha']^{\mathsf{R}}_{0})\star([\beta]^{\mathsf{L}}_{0}\otimes m)=[\alpha\alpha']^{\mathsf{L}}_{n}\star([\beta]^{\mathsf{L}}_{0}\otimes m)=[\alpha\alpha'\beta]^{\mathsf{L}}_{0}\otimes m
      \]
      Hence, the statement follows.
    \end{proof}

    Now, back to the proof of the theorem and suppose that the strong identity condition holds. By the above lemma, the action of $\mathfrak{a}\in\ker{\mu}\subset\mathfrak{A}_{\bullet+n,-n}$ on any $\Phi^{\mathsf{L}}_{0}(M)_{n}$ factors through $\mu$, so it should be trivial. In particular, for $\mathfrak{A}=\Phi^{\mathsf{L}}_{0}(\Phi^{\mathsf{R}}_{0}(\mathsf{A}_{0}))$, we have $\mathfrak{a}=\mathfrak{a}\star\mathbf{1}_{n}=0$. This shows the injectivity of $\mu$. 

    Finally, we show that the sequence splits. 
    Note that each term of the sequence \eqref{eq:ES:fL} is equipped with the right action of the corresponding term in \eqref{eq:ES:fA}, and these actions are compatible with the maps $\mu$ and $\pi$. Since the sequence \eqref{eq:ES:fA} splits, so does \eqref{eq:ES:fL}.
  \end{proof}

  Repeating the above module decomposition, we have (recall \zcref{eg:completion})
  \begin{equation}\label{eq:decomposition-fL}
    \widehat{U}_{\bullet}=\varprojlim \mathfrak{L}^{n}_{\bullet} \cong \prod_{n\in\N}\mathfrak{A}_{\bullet+n,-n}
    \txand 
    \widehat{U}_{-\bullet}=\varprojlim \mathfrak{R}^{n}_{-\bullet} \cong \prod_{n\in\N}\mathfrak{A}_{n,\bullet-n},
  \end{equation}
  with the product topology and the continuous actions\footnote{Recall that, each $\mathfrak{L}^{n}$ is a $(U|\mathsf{A}_{n})$-bimodule and each $\mathfrak{R}^{n}$ is a $(\mathsf{A}_{n}|U)$-bimodule.} of $\widehat{U_{0}}$ on the right and on the left respectively. In particular, by \eqref{eq:SIC:action}, there is a sequence $(e_{n})_{n\in\N}$ in $\widehat{U_{0}}$ (for instance, the sequence in the paragraph below \zcref[noname]{eq:decomposition-fA}) satisfying \zcref[noname]{si1,si2} and the following: 
  \begin{enumerate}[label=(\texttt{si}\arabic*),resume]
    \item For any $n\in\N$, we have:
    \[
      \widehat{U}_{-n}\cdot e_{<N_{n}}=0
      \txand
      e_{<N_{n}}\cdot\widehat{U}_{n}=0,
    \]
    where $e_{<N_{n}}=\sum\limits_{k=0}^{N_{n}-1}e_{k}$ is the partial sum and $N_n$ is the smallest integer such that the remainder $e_{\ge N_n}:=\sum\limits_{k=n}^{\infty}e_{k}$ belongs to $\NLR[n]\widehat{U_{0}}$ (cf. \zcref{rem:Nn}).
  \end{enumerate}

\subsection{Strong identity expansion}
  We summarize the above into the following definition.
  \begin{definition}\label{def:SIE}
    A \concept{strong identity expansion} (\nota{\textsf{SIE}} for short) for a unital $U$ is a sequence $(e_{n})_{n\in\N}$ in $\widehat{U_{0}}$ satisfying the conditions 
    \begin{enumerate}[label=(\texttt{si}\arabic*)]
      \item\label{si1} $\sum\limits_{n=0}^{\infty}e_{n}$ converges to $1\in\widehat{U_{0}}$.
      \item\label{si2} $\Set{e_{n}}_{n\in\N}$ are orthogonal to each other.  
      \item\label{si3} For any $n\in\N$, we have:
      \[
        \widehat{U}_{-n}\cdot e_{<N_{n}}=0
        \txand
        e_{<N_{n}}\cdot\widehat{U}_{n}=0,
      \]
      where $e_{<N_{n}}=\sum\limits_{k=0}^{N_{n}-1}e_{k}$ is the partial sum and $N_n$ is the smallest integer (cf. \zcref{rem:Nn}) such that the remainder $e_{\ge N_n}:=\sum\limits_{k=n}^{\infty}e_{k}$ belongs to $\NLR[n]\widehat{U_{0}}$.
    \end{enumerate}
  \end{definition}

  Now, assume $U$ admits a strong identity expansion, saying $(e_{n})_{n\in\N}$. Let $f_n$ be the sum $e_{N_n}+\cdots+e_{N_{n+1}-1}$. We see that $(f_{n})_{n\in\N}$ also satisfy \zcref[noname]{si1,si2,si3}. Hence, we may assume:
  \begin{enumerate}[label=(\texttt{si}\arabic*),resume]
    \item\label{si4} 
    For each $n\in\N$, we have $e_{n}\in\NLR[n]\widehat{U_{0}}$. 
  \end{enumerate}
  \emph{This convention will be insisted in the rest of the paper.}
  \begin{remark}
    The sequence in the paragraph below \zcref[noname]{eq:decomposition-fA} already satisfies \zcref[noname]{si1,si2,si3,si4}.
  \end{remark}

  \begin{theorem}\label{thm:SIC-SIE}
    The strong identity condition (\SIC) is verified for a unital $U$ if and only if $U$ admits a strong identity expansion (\SIE).
  \end{theorem}
  \begin{proof}
    The implication ${\SIC}\implies {\SIE}$ has been discussed in the previous subsection. We only focus on the converse direction. 

    The remainder $e_{\ge n}$ serves as a left identity on $U_{n}$ and a right identity on $U_{-n}$. 
    In particular, we can find sequences $e_{n}^{+}$ and $e_{n}^{-}$ in $U_{n}$ and $U_{-n}$ respectively such that
    \begin{equation}\label{eq:exp-egen}
      e_{\ge n} \equiv e_{n}^{+}e_{n}^{-} \left(:= \sum_{i} e_{n}^{+}(i)e_{n}^{-}(i) \right)  \mod\NLR[n+1]U_{0}.
    \end{equation}
    Hence, $[e_{\ge n}]_{n}$ is equal to the image of the following element under the map $\mu$ in \zcref{prop:decomposition_of_fA}:
    \[
      \one_{n} := \sum_{i} [e_{n}^{+}(i)]^{\mathsf{L}}_{0}\otimes[e_{n}^{-}(i)]^{\mathsf{R}}_{0} \in \mathfrak{A}_{n}.
    \]
    Then, for any $\alpha\in U_{n}$, by \zcref{lem:Act_of_fA_via_A} and \zcref[noname]{si3}, we have 
    \[
      \one_{n}\star[\alpha]^{\mathsf{L}}_{0} = [e_{n}^{+}e_{n}^{-}]_{n}\cdot[\alpha]^{\mathsf{L}}_{0} = [e_{\ge n}]_{n}\cdot[\alpha]^{\mathsf{L}}_{0} = [1]_{n}[\alpha]^{\mathsf{L}}_{0} = [\alpha]^{\mathsf{L}}_{0}.
    \]
    Similarly, for any $\beta\in U_{-n}$, we have $[\beta]^{\mathsf{R}}_{0}\star\one_{n}=[\beta]^{\mathsf{R}}_{0}$. This shows that $\one_{n}$ is a strong identity element in $\mathfrak{A}_{n}$.
  \end{proof}

\section{Splitting of the \texorpdfstring{$\Omega$}{Omega}-filtration}\label{sec:splitOmega}
  In this section, we give another characterization of the strong identity condition.

\subsection{Consequences of SIE}
  We assume that the almost-canonically seminormed ring $U$ admits a strong identity expansion (\SIE) $(e_{n})_{n\in\N}$. 
  \begin{proposition}\label{lem:splitOmega}
    Under our assumptions, any exhaustive left $U$-module $W$ is gradable and its $\Omega$-filtration splits.
  \end{proposition}
  \begin{proof}
    Since each $\Omega_{n}(W)$ is a left $\mathsf{A}_{n}$-module, we have a continuous action of $\widehat{U_{0}}=\varprojlim \mathsf{A}_{n}$ on their union $W$ that is compatible with the original left action of $U$. 
    Now, set
    \begin{equation}\label{eq:splitOmega}
      \nota{W_{(n)}}:=\Set{e_{n} w\given w\in W}
    \end{equation}
    for $n\in\N$ and $W_{(n)}=0$ for $n<0$. Note that they are $\widehat{U_{0}}$-submodules of $W$ by \zcref[noname]{si1,si2}.
    
    First, we show that $W=\bigoplus_{n\in\N}W_{(n)}$:
    \begin{enumerate}
      \item $W_{(m)}\cap W_{(n)}=0$ whenever $m\neq n$. This is due to \ref{si2}. 
      \item $W=\sum_{n\in\N}W_{(n)}$. Indeed, we will show that 
      \[\Omega_{n}(W)\subset W_{(0)}+\cdots+W_{(n)}.\] 
      This is due to \ref{si4}. By which, $e_{\ge n+1}$ is a limit of elements from $\NL[n+1]U$. Therefore,  $e_{\ge n+1} \Omega_{n}(W)=0$. Then, for any $w\in\Omega_{n}(W)$, we have $w=1\cdot w=e_0 w + \cdots + e_n w$.
    \end{enumerate}

    Next, we show that $W_{(n)}\subset\Omega_{n}(W)$. Indeed, by \ref{si3}, $U_{-n-1}=U_{-n-1} e_{\ge n+1}$. On the other hand, by \ref{si2}, $e_{\ge n+1} e_{n}=0$. Hence, $U_{-n-1} W_{(n)}=0$ as desired. 
    
    So far, we have 
    \[
      \Omega_{n}(W)=W_{(0)} \oplus \cdots \oplus W_{(n)}.
    \]
    This shows that the $\Omega$-filtration splits as $\widehat{U_{0}}$-modules.
    
    Finally, we show that this does provide a grading on $W$. If $p<-n$, since $U_{p}\Omega_{n}(W)=0$, we have $U_{p}W_{(n)}=0$. Now, let $p\ge -n$. Then, $U_{p}W_{(n)}\subset U_{p}\Omega_{n}(W)\subset\Omega_{p+n}(W)$. Thus, the action of $\widehat{U_{0}}$ on them factors through $\mathsf{A}_{p+n}$. Then, for any $\alpha\in U_{p}$ and $w\in W$, by \zcref[noname]{si3,eq:exp-egen}, and note that $\alpha U_{n}\subset U_{p+n}$, we have
    \begin{align*}
      \alpha(e_{n}w) &= (\alpha e_{n})w = [\alpha e_{n}^{+}e_{n}^{-}]_{p+n}w = [\alpha e_{n}^{+}]_{p+n}e_{n}^{-}w \\
      &= [e_{p+n}\alpha e_{n}^{+}]_{p+n}e_{n}^{-}w = e_{p+n}(\alpha e_{n}w).
    \end{align*}
    This shows $U_{p}W_{(n)}\subset W_{(p+n)}$ as desired.
  \end{proof}

  \begin{remark}
    From the proof, we see that the grading $W_{(\bullet)}$ is canonically isomorphic to the $\gr\Omega$-grading in \zcref{nota:grOmega}.
  \end{remark}
  \begin{warning}
    The above proposition says that any exhaustive $U$-module is \emph{gradable}. However, this does not mean that the forgetting functor $\Mod[Gr](U|{}_R)\to\Mod[Ex](U|{}_R)$ is an equivalence. 
    The functor sending an exhaustive $U$-module $W$ to its grading $W_{(\bullet)}$ is only a right inverse of the forgetting functor, but not a left inverse. 
  \end{warning}

  \begin{corollary}
    Under our assumptions, the canonical projections 
    \[
      \cdots \longrightarrow 
      \mathfrak{L}^{n} \overset{\pi_{n}}{\longrightarrow} 
      \mathfrak{L}^{n-1} \longrightarrow 
      \cdots \longrightarrow 
      \mathfrak{L}^{0} \longrightarrow 
      0
    \]
    respect the splitting of $\Omega$-filtrations in the sense that they preserve the $\gr\Omega$-gradings. 
    The same holds for $\mathfrak{R}^{n}$'s.
  \end{corollary}
  \begin{proof}
    We need to show $\pi_{n}(\mathfrak{L}^{n}_{(\bullet)}) = \mathfrak{L}^{n-1}_{(\bullet)}$. Indeed, due to the description of the $\gr\Omega$-grading in the proof of \zcref{lem:splitOmega}, we have $\pi_{n}(\mathfrak{L}^{n}_{(m)}) = \pi_{n}(e_{m}\mathfrak{L}^{n}) = e_{m}\pi_{n}(\mathfrak{L}^{n}) = e_{m}\mathfrak{L}^{n-1} = \mathfrak{L}^{n-1}_{(m)}$ for any $m\in\N$.
  \end{proof}

  \begin{lemma}\label{lem:grOmegaPhi}
    Under our assumptions, for any left $\mathsf{A}_{0}$-module $M$, the natural grading on $\Phi^{\mathsf{L}}_{0}(M)$ agrees with its $\gr\Omega$-grading.
  \end{lemma}
  \begin{proof}
    It suffices to verify the statement for $M=\mathsf{A}_{0}$. Note that $\Phi^{\mathsf{L}}_{0}(\mathsf{A}_{0})_{\bullet}=\mathfrak{L}^{0}_{\bullet}\otimes_{U_{0}}\mathsf{A}_{0}=\mathfrak{L}^{0}_{\bullet}$. 
    On one hand, by \ref{si3}, $e_{\ge n}$ serves as a left identity on $U_{n}$. On the other hand, since $\mathfrak{L}^{0}_{n}\subset\Omega_{n}(\mathfrak{L}^{0})$, the action of $\widehat{U_{0}}$ on $\mathfrak{L}^{0}_{n}$ factors through $\mathsf{A}_{n}$. Since $[e_{\ge n}]_{n}=[e_{n}]_{n}$. We thus conclude that $\mathfrak{L}^{0}_{n}=e_{n}\mathfrak{L}^{0}_{n}$. Then, $e_{n}\mathfrak{L}^{0}=\bigoplus_{m\in\N}e_{n}\mathfrak{L}^{0}_{m}=\bigoplus_{m\in\N}e_{n}e_{m}\mathfrak{L}^{0}_{m}=e_{n}\mathfrak{L}^{0}_{n}=\mathfrak{L}^{0}_{n}$. 
  \end{proof}

  Similar statements fails for $\Phi^{\mathsf{L}}_{n}$ ($n>0$). Instead, we have 
  \begin{lemma}\label{lem:grOmegaPhin}
    Under our assumptions, for any left $\mathsf{A}_{n}$-module $M$, the $\gr\Omega$-grading on $\Phi^{\mathsf{L}}_{n}(M)$ is given by
    \[
      \Phi^{\mathsf{L}}_{n}(M)_{(m)}=\bigoplus_{p=m-n}^{m}(\Phi^{\mathsf{L}}_{n}(M)_{(m)}\cap\Phi^{\mathsf{L}}_{n}(M)_{p}).
    \]
  \end{lemma}
  \begin{proof}
    It suffices to verify the statement for $M=\mathsf{A}_{n}$. Recall that $\Phi^{\mathsf{L}}_{n}(\mathsf{A}_{n})_{\bullet}=\mathfrak{L}^{n}_{\bullet}$. 
    For any $p\ge-n$, since $\mathfrak{L}^{n}_{p}\subset\Omega_{p+n}(\mathfrak{L}^{n})$, the action of $\widehat{U_{0}}$ on $\mathfrak{L}^{n}_{p}$ factors through $\mathsf{A}_{p+n}$. 
    On the other hand, since $\mathfrak{L}^{n}_{p} = \mathfrak{R}^{n+p}_{p}$, by \zcref[noname]{si3}, $e_{<p}$ annihilates it on the left. Thus, we have 
    \[
      \mathfrak{L}^{n}_{p}=\bigoplus_{i=\max\Set{p,0}}^{p+n}e_{i}\mathfrak{L}^{n}_{p}.
    \]
    Therefore, for any $m\in\N$, we have
    \[
      e_{m}\mathfrak{L}^{n}=\bigoplus_{p=-n}^{\infty}e_{m}\mathfrak{L}^{n}_{p}=
      \bigoplus_{p=-n}^{\infty}\bigoplus_{i=\max\Set{p,0}}^{p+n}e_{m}e_{i}\mathfrak{L}^{n}_{p}=\bigoplus_{p=m-n}^{m}e_{m}\mathfrak{L}^{n}_{p}.
    \]
    as desired.
  \end{proof}

\subsection{Consequences of the splitting}
  Now, we study the converse. 
  We assume:
  \begin{enumerate}[leftmargin=5\parindent,font=\nota]
    \myItem[($\Omega$split)] \label{Omegasplit}
        The $\Omega$-filtrations on exhaustive $U$-modules split and the canonical projections 
        \begin{align*}
          \cdots \longrightarrow 
          \mathfrak{L}^{n} \overset{\pi_{n}}{\longrightarrow} 
          \mathfrak{L}^{n-1} \longrightarrow 
          \cdots \longrightarrow 
          \mathfrak{L}^{0} \longrightarrow 
          0 \\
          \cdots \longrightarrow 
          \mathfrak{R}^{n} \overset{\pi_{n}}{\longrightarrow} 
          \mathfrak{R}^{n-1} \longrightarrow 
          \cdots \longrightarrow 
          \mathfrak{R}^{0} \longrightarrow 
          0
        \end{align*}
        respect the splittings in the sense that they preserve the $\gr\Omega$-gradings\footnote{Although $\pi_{n}$'s always preserve the $\Omega$-filtrations, they do not automatically preserve the $\gr\Omega$-grading.}.
  \end{enumerate}
  Although we assume the $\Omega$-filtrations on all exhaustive modules split, we only focus on those in the sequences.
  \begin{proposition}
    Under our assumption, these projections split as homomorphisms of $U$-modules. 
    The same holds for $\mathfrak{R}^{n}$'s.
  \end{proposition}
  \begin{proof}
    First, since $1_{n} \in \Omega^{\mathsf{L}}_{n}(\mathfrak{L}^{n})$ (by \zcref{lem:ThickZhu}), it can be expressed as $x_{0} + \cdots + x_{n}$, where each $x_k \in \mathfrak{L}^{n}_{(k)}$.
    Since $\pi$ preserves the $\Omega$-filtrations and is strict, we have $\pi(x_{k}) \subset \mathfrak{L}^{n-1}_{(k)}$.
    On the other hand, $1_{n-1} \in \Omega^{\mathsf{L}}_{n-1}(\mathfrak{L}^{n-1})$.
    Therefore, $x_{n}$ must belongs to the kernel of $\pi$. 
    The assignment $1_{n-1} \mapsto 1_{n} - x_{n}$ gives a splitting of $\pi$ as desired.
  \end{proof}

  \begin{corollary}
    Under the assumptions of the above proposition, $U$ admits a strong identity expansion (\SIE).
  \end{corollary}
  \begin{proof}
    First note that, the canonical projections $\pi_{n}$ are morphisms of graded modules. Hence, if any $\pi_{n}$ splits, it would give a decomposition of graded $U$-modules $\mathfrak{L}^{n} = \ker(\pi_n)\oplus\mathfrak{L}^{n-1}$. 
    In particular, there are $\mathfrak{e}_{n} \in \ker(\pi_n)$ such that $1_{n} = \mathfrak{e}_{n} + 1_{n-1}$ through this decomposition.
    The same argument applies to $\mathfrak{R}^{n}$'s.
    Passing to the limit, we have decompositions of graded left and right $\widehat{U}$-modules respectively:
    \[
      \widehat{U} = \mathfrak{L}^{n} \oplus \NL[n+1]\widehat{U}
      \txand
      \widehat{U} = \mathfrak{R}^{n} \oplus \NR[n+1]\widehat{U}.
    \]
    This provides elements $(e_n)_{n\in\N}$ and $(f_n)_{n\in\N}$ in $\widehat{U_{0}}$ by tracing the image of $1_U$ under the above decompositions. Furthermore, since $\mathfrak{L}^{n}_{0} = \mathfrak{R}^{n}_{0} = \mathsf{A}_{0}$, we must have $e_n = f_n$.

    We claim that the elements $(e_n)_{n\in\N}$ verify the strong identity expansion. 
    Indeed, \zcref[noname]{si1,si2,si4} are clear from the construction. 
    It remains to check \zcref[noname]{si3}, which is due to the fact that $U_{-n} \subset \NL[n]U$ and $U_{n} \subset \NR[n]U$.
  \end{proof}

  Summarize the discussion in this section, we have
  \begin{theorem}\label{thm:SIE-Omega}
    An almost-canonically seminormed ring $U$ admits a strong identity expansion (\SIE) if and only if the condition \zcref[noname]{Omegasplit} is satisfied.
  \end{theorem}

\section{Strong identity expansion implies Morita-type equivalences}\label{sec:Morita}
  In this section, we will show that {\SIE} implies the following Morita-type equivalences: 
  \begin{enumerate}[leftmargin=5\parindent,font=\nota]
    \myItem[\textup{($\Phi\Omega\mathsf{L}_{R}$)}]\label{PhiOmegaL} 
        The adjoint pair 
        $\Phi^{\mathsf{L}}_{0}\colon \Mod(\mathsf{A}_{0}|{}_R) \rightleftharpoons \Mod[Ex](U|{}_R)\colon\Omega^{\mathsf{L}}_{0}$ 
        is an equivalence. 
    \myItem[\textup{($\Phi\Omega\mathsf{R}_{R}$)}]\label{PhiOmegaR} 
        The adjoint pair 
        $\Phi^{\mathsf{R}}_{0}\colon \Mod({}_R|\mathsf{A}_{0}) \rightleftharpoons \Mod[Ex]({}_R|U)\colon\Omega^{\mathsf{R}}_{0}$ 
        is an equivalence.
  \end{enumerate}
  Note that, one cannot expect similar statements to hold using the pairs $(-)_{0}\dashv\Phi^{\mathsf{L}}_{0}$ and $(-)_{0}\dashv\Phi^{\mathsf{R}}_{0}$ in \zcref{prop:adj:Phideg}. Indeed, one can always shift the grading on a gradable $U$-module to obtain non-isomorphic graded $U$-modules with the same underlying $U$-module.

  In this section, we assume that $U$ admits a strong identity expansion, saying $(e_{n})_{n\in\N}$.

\subsection{Morita-type equivalence}
  Now, we turn to the Morita-type equivalence. 
  The following lemma will serve as the general framework for the proof of Morita-type equivalences in this paper.

  \begin{lemma}\label{lem:AdjEq}
    Let $L\dashv R$ be an adjoint pair between abelian categories such that 
    \begin{itemize}
      \item the unit map $\eta\colon\id\to R\circ L$ is a natural isomorphism; and
      \item the counit map $\varepsilon\colon L\circ R\to \id$ is a natural epimorphism.
    \end{itemize}
    Then, $L\dashv R$ is an adjoint equivalence.
  \end{lemma}
  \begin{proof}
    It suffices to show that $\varepsilon$ is monic. Indeed, let $K$ be the kernel of the counit map $\varepsilon\colon L\circ R\to \id$. Then, we have a commutative diagram with exact rows and columns:
    \[
      \begin{tikzcd}
        & & 0\ar[d] & K\ar[d] \\
        & LRK \ar[r]\ar[d] & LRLR \ar[r]\ar[d,equal] & LR \ar[r]\ar[d] & 0\\
        0\ar[r] & K\ar[r]\ar[d] & LR\ar[r]\ar[d] & \id\ar[r]\ar[d] & 0 \\
        & 0 & 0 & 0
      \end{tikzcd}
    \]
    Then, the snake lemma implies that $K=0$.
  \end{proof}

  Applying \zcref{lem:AdjEq} to our case, we have:
  \begin{theorem}\label{thm:PhiOmega0}
    If $U$ admits a strong identity expansion (\SIE), then the adjoint pairs $\Phi^{\mathsf{L}}_{0}\dashv\Omega^{\mathsf{L}}_{0}$ and $\Phi^{\mathsf{R}}_{0}\dashv\Omega^{\mathsf{R}}_{0}$ are adjoint equivalence (i.e. \zcref[noname]{PhiOmegaL,PhiOmegaR} are verified).
  \end{theorem}
  \begin{proof}
    We only focus on the adjoint pair $\Phi^{\mathsf{L}}_{0}\dashv\Omega^{\mathsf{L}}_{0}$. The other one is similar. By \zcref{lem:AdjEq}, it boils down to verifying that 
    \begin{enumerate}
      \item the unit map $\eta_{M}\colon M\mapsto\Omega^{\mathsf{L}}_{0}\Phi^{\mathsf{L}}_{0}(M)$ is bijective for any left $\mathsf{A}_{0}$-module $M$; and
      \item the counit map $\varepsilon_{W}\colon\Phi^{\mathsf{L}}_{0}\Omega^{\mathsf{L}}_{0}(W)\to W$ is surjective for any exhaustive left $U$-module $W$.
    \end{enumerate}
    For (1), this is due to $\Omega^{\mathsf{L}}_{0}\Phi^{\mathsf{L}}_{0}(M)=\Phi^{\mathsf{L}}_{0}(M)_{(0)}=\Phi^{\mathsf{L}}_{0}(M)_{0}$ (by \zcref{lem:grOmegaPhi}). 
    For (2), we argue as follows.
    For any $w\in W_{(n)}$, since the action of $\widehat{U_{0}}$ on it factors through $\mathsf{A}_{n}$, we have 
    \[
      w=1\cdot w=[e_n]_{n}w=[e_{\ge n}]_{n}w=[e_{n}^{+}e_{n}^{-}]_{n}w.
    \] 
    Then, $w$ is the image of 
    \[
      \sum_{i}[e_{n}^{+}(i)]^{\mathsf{L}}_{0}\otimes(e_{n}^{-}(i) w)\in\Phi^{\mathsf{L}}_{0}(\Omega^{\mathsf{L}}(W))_{n},
    \]
    where the summation ranges over all $i$ such that $[e_{n}^{+}(i)e_{n}^{-}(i)]_{n}\neq0$.
  \end{proof}

  \begin{warning}\label{rem:noPhiOmegan}
    The above adjoint equivalence has no chance to be generalized to the adjoint pairs $\Phi^{\mathsf{L}}_{n}\dashv\Omega^{\mathsf{L}}_{n}$ for $n>0$. Indeed, for a left $\mathsf{A}_{n}$-module $M$, the unit map $\eta_{M}\colon M\to\Omega^{\mathsf{L}}_{n}\Phi^{\mathsf{L}}_{n}(M)$ is still injective, but its image lives in $\Phi^{\mathsf{L}}_{n}(M)_{n}$, which is often not the whole $\Omega^{\mathsf{L}}_{n}(\Phi^{\mathsf{L}}_{n}(M))$. To see this, we have the following lemma inspired by \cite[4.1]{DLM2}:
  \end{warning}
  \begin{lemma}\label{lem:bottem}
    Let $M$ be a left $\mathsf{A}_{n}$-module. If $M$ is not annihilated by $\NLR[n]\mathsf{A}_{n}$, then $\Phi_{n}(M)_{-n}\neq0$.
  \end{lemma}
  \begin{proof}
    Otherwise, we have $U_{-n}\Phi_{n}(M)_{0}\subset\Phi_{n}(M)_{-n}=0$. Then, we see that $U_{n}U_{-n}\subset U_0$ annihilates $\Phi_{n}(M)_{0}$ on the left and hence annihilates $M$. 
    By \zcref[noname]{A2p}, this shows that $\NLR[n]\mathsf{A}_{n}$ annihilates $M$, a contradiction.
  \end{proof}
  Hence, by \zcref{lem:grOmegaPhin}, the unit map cannot be surjective in general.

\subsection{For quasi-rigid modules}
  Applying \zcref{prop:adj:ordinary}, we have the following corollary for quasi-rigid modules.
  \begin{corollary}\label{coro:PhiOmega0-ord}
    Suppose either 
    \begin{enumerate}
      \item $U$ is \emph{quasi-rigid}: $\mathfrak{L}^{n} \in \Mod*[Gr](U|{}_{\mathsf{A}_{n}})$ and $\mathfrak{R}^{n} \in \Mod*[Gr]({}_{\mathsf{A}_{n}}|U)$; or 
      \item $R$ is semisimple and $U$ is \emph{weakly quasi-finite}: components of $\mathfrak{L}^{n},\mathfrak{R}^{n}$ are finite over $\mathsf{A}_{n}$.
    \end{enumerate}
    If $U$ admits a strong identity expansion (\SIE), then the adjoint pairs
    \[
      \Phi^{\mathsf{L}}_{0}\colon \Mod*(\mathsf{A}_{0}|{}_R) \rightleftharpoons \Mod*[Ex](U|{}_R)\colon\Omega^{\mathsf{L}}_{0}
      \txand
      \Phi^{\mathsf{R}}_{0}\colon \Mod*({}_R|\mathsf{A}_{0}) \rightleftharpoons \Mod*[Ex]({}_R|U)\colon\Omega^{\mathsf{R}}_{0}
    \]
    are adjoint equivalences.
  \end{corollary}
  \begin{proof}
    First, by \zcref{prop:adj:ordinary}, the functor $\Phi^{\mathsf{L}}_{0}|_{\Mod*(\mathsf{A}_{0}|{}_R)}$ lands in $\Mod*[Gr](U|{}_R)$, and hence in $\Mod*[Ex](U|{}_R)$. 
    In particular, for any $M \in \Mod*(\mathsf{A}_{0}|{}_R)$, each component $\Phi^{\mathsf{L}}_{0}(M)_{n}$ is rigid over $R$.
    On the other hand, by \zcref{lem:splitOmega}, we see that $\Omega^{\mathsf{L}}_{0}$ is the composition 
    \[
      \Mod[Ex](U|{}_R) \overset{(-)_{(\bullet)}}{\longrightarrow} 
      \Mod[Gr](U|{}_R) \overset{(-)_{0}}{\longrightarrow} 
      \Mod(\mathsf{A}_{0}|{}_R).
    \]
    Then, by \zcref{lem:grOmegaPhi}, we conclude that all $\Omega^{\mathsf{L}}_{n}(\Phi^{\mathsf{L}}_{0}(M))$ are rigid over $R$. 
    In particular, the functor $\Omega^{\mathsf{L}}_{0}|_{\Mod*[Ex](U|{}_R)}$ lands in $\Mod*(\mathsf{A}_{0}|{}_R)$. 
    Finally, by \zcref{thm:PhiOmega0}, the adjoint pair is an adjoint equivalence. The other side is similar.
  \end{proof}

\section{Morita-type equivalence implies the strong identity condition}\label{sec:Morita_to_SIC}
  In this section, we will deduce the strong identity condition from the Morita-type equivalences \zcref[noname]{PhiOmegaL,PhiOmegaR} plus certain rigidity/finiteness conditions on $U$.

\subsection{Duals of induced modules}
  We need some preparations on graded dual modules.
  
  \begin{lemma}
    There are homogeneous homomorphisms 
    \[
      \psi^{\mathsf{L}}_{M} \colon \Phi^{\mathsf{L}}_{0}(M^{R|\vee}) \longrightarrow \Phi^{\mathsf{R}}_{0}(M)^{R|\dagger}
      \txand
      \psi^{\mathsf{R}}_{M} \colon \Phi^{\mathsf{R}}_{0}(M^{\vee|R}) \longrightarrow \Phi^{\mathsf{L}}_{0}(M)^{\dagger|R},
    \]
    for $({}_R|\mathsf{A}_{0})$-bimodules $M$ and $(\mathsf{A}_{0}|{}_R)$-bimodules $M$ respectively.
  \end{lemma}
  \begin{proof}
    By the adjunction $\Phi^{\mathsf{L}}_{0}\dashv(-)_{0}$ in \zcref{prop:adj:Phideg}, for any $({}_R|\mathsf{A}_{0})$-bimodule $M$, we have a \emph{homogeneous} homomorphism to $\Phi^{\mathsf{R}}_{0}(M)^{R|\dagger}$ from 
    \[
      \Phi^{\mathsf{L}}_{0}\left((\Phi^{\mathsf{R}}_{0}(M)^{R|\dagger})_{0}\right) = \Phi^{\mathsf{L}}_{0}\left((\Phi^{\mathsf{R}}_{0}(M)_{0})^{R|\vee}\right) = \Phi^{\mathsf{L}}_{0}(M^{R|\vee}).
    \]
    This shows the first. The other one is similar. 
  \end{proof}

  \begin{lemma}\label{lem:induced-dual}
    Suppose \zcref[noname]{PhiOmegaL} (resp. \zcref[noname]{PhiOmegaR}) is verified. 
    Then, for any $({}_R|\mathsf{A}_{0})$-bimodule (resp. $(\mathsf{A}_{0}|{}_R)$-bimodule) $M$, the homomorphism $\psi^{\mathsf{L}}_{M}$ (resp. $\psi^{\mathsf{R}}_{M}$) is an isomorphism.
  \end{lemma}
  \begin{proof}
    By \zcref{prop:adj:PhiOmega}, we have the following commutative diagram.
    \[
      \begin{tikzcd}
        {(\Phi^{\mathsf{L}}_{0}(M^{R|\vee}))_{0}} \ar[d,hookrightarrow] \ar[r,equal] & 
        {M^{R|\vee}} \ar[r,equal] &
        {(\Phi^{\mathsf{R}}_{0}(M)^{R|\dagger})_{0}} \ar[d,hookrightarrow] \\
        {\Omega^{\mathsf{L}}_{0}(\Phi^{\mathsf{L}}_{0}(M^{R|\vee}))} \ar[rr,"{\Omega^{\mathsf{L}}_{0}(\psi^{\mathsf{L}})}"] & &
        {\Omega^{\mathsf{L}}_{0}(\Phi^{\mathsf{R}}_{0}(M)^{R|\dagger})} \ar[d,hookrightarrow] \\
        & & {\Omega^{\mathsf{L}}_{0}(\Phi^{\mathsf{R}}_{0}(M)^{R|\vee})}
      \end{tikzcd}
    \]
    By \zcref{lem:OmegaDual}, the last term is 
    \[
      \Omega^{\mathsf{L}}_{0}(\Phi^{\mathsf{R}}_{0}(M)^{R|\vee}) = \left( \Phi^{\mathsf{R}}_{0}(M) \otimes_{U} \mathfrak{L}^{0} \right)^{R|\vee} = \left( M \otimes_{\mathsf{A}_{0}} \mathfrak{R}^{0} \otimes_{U} \mathfrak{L}^{0} \right)^{R|\vee}.
    \]
    Tracing through the diagram, we see that the composition $M^{R|\vee} \to \left( M \otimes_{\mathsf{A}_{0}} \mathfrak{R}^{0} \otimes_{U} \mathfrak{L}^{0} \right)^{R|\vee}$ is given by the dual of the right action map of $\mathfrak{A}_{0}$ on $M$ (cf. \zcref{lem:def:ostar}), which is an isomorphism.
    Thus, $\psi^{\mathsf{L}}_{M}$ is an isomorphism by \zcref[noname]{PhiOmegaL}. The other case is similar.
  \end{proof}

  What we need is the following special case.
  \begin{lemma}\label{lem:LisDualOfR}
    Suppose \zcref[noname]{PhiOmegaL} (resp. \zcref[noname]{PhiOmegaR}) is verified for $R=\mathsf{A}_{0}$. 
    Then, we have 
    \[
      \mathfrak{L}^{0} \cong (\mathfrak{R}^{0})^{\mathsf{A}_{0}|\dagger}
      \qquad 
      \left(\text{resp. }
        \mathfrak{R}^{0} \cong (\mathfrak{L}^{0})^{\dagger|\mathsf{A}_{0}}
      \right).
    \]
  \end{lemma}
  \begin{proof}
    Applying \zcref{lem:induced-dual} to $M=\mathsf{A}_{0}$ and note that $\mathsf{A}_{0}^{\mathsf{A}_{0}|\vee} \cong \mathsf{A}_{0} \cong \mathsf{A}_{0}^{\vee|\mathsf{A}_{0}}$, we obtain the desired isomorphisms.
  \end{proof}

\subsection{Morita-type equivalence implies strong identity condition}
  Now, we need some assumptions. 
  \begin{theorem}\label{thm:Adjeq-SIC}
    Suppose either 
    \begin{enumerate}
      \item $\mathfrak{L}^{0} \in \Mod*[Gr](U|{}_{\mathsf{A}_{0}})$ and \zcref[noname]{PhiOmegaR} is verified for $R=\mathsf{A}_{0}$; or
      \item $\mathfrak{R}^{0} \in \Mod*[Gr]({}_{\mathsf{A}_{0}}|U)$ and \zcref[noname]{PhiOmegaL} is verified for $R=\mathsf{A}_{0}$.
    \end{enumerate} 
    Then, the strong identity condition (\SIC) holds.
  \end{theorem}
  \begin{remark}
    Be aware that $(\Phi\Omega\mathsf{L}_{\mathsf{A}_{0}})+(\Phi\Omega\mathsf{R}_{\mathsf{A}_{0}})$ only implies that each block $\mathfrak{A}_{p,-q}$ is reflexive (i.e. isomorphic to its double dual), which is \emph{weaker} than the above assumption.
    On the other hand, the above assumption is \emph{weaker} than \emph{quasi-rigidity} in \zcref{def:quasi}.
  \end{remark}
  \begin{proof}
    Let's assume (i). The proof in case (ii) is similar. 
    By \zcref{lem:LisDualOfR}, we have \emph{homogeneous} isomorphisms
    \begin{equation}\label{eq:ModeDecomposition}
      \mathfrak{A}_{\bullet,-\circ} 
      = 
      \mathfrak{L}^{0}_{\bullet}\otimes_{\mathsf{A}_{0}}\mathfrak{R}^{0}_{-\circ}
      \cong
      \mathfrak{L}^{0}_{\bullet}\otimes_{\mathsf{A}_{0}}(\mathfrak{L}^{0})^{\dagger|\mathsf{A}_{0}}_{-\circ} 
      =
      \mathfrak{L}^{0}_{\bullet}\otimes_{\mathsf{A}_{0}}(\mathfrak{L}^{0}_{\circ})^{\vee|\mathsf{A}_{0}}.
    \end{equation}
    In particular, for each $n\ge 0$, we have
    \[
      \mathfrak{A}_{n} \cong 
      \mathfrak{L}^{0}_{n}\otimes_{\mathsf{A}_{0}}(\mathfrak{L}^{0}_{n})^{\vee|\mathsf{A}_{0}}
      =
      \operatorname{End}_{|\mathsf{A}_{0}}(\mathfrak{L}^{0}_{n}),
    \]
    where the last equality is the only place we need $\mathfrak{L}^{0} \in \Mod*[Gr](U|{}_{\mathsf{A}_{0}})$.

    Let $\one_{n}$ be the element of $\mathfrak{A}_{n}$ corresponding to $\id\in\operatorname{End}_{|\mathsf{A}_{0}}(\mathfrak{L}^{0}_{n})$. We conclude that the elements $(\one_{n})_{n\in\N}$ verify the strong identity condition.
  \end{proof}
  \begin{remark}
    Under the assumptions, all modules involved in the proof are in the small categories $\Mod*[Ex](U|{}_{\mathsf{A}_{0}})$, $\Mod*[Ex]({}_{\mathsf{A}_{0}}|U)$ and $\Mod*(\mathsf{A}_{0}|\mathsf{A}_{0})$. Thus, we can also deduce the strong identity condition from one of the following assumptions:
    \begin{enumerate}
      \item
          The adjoint pair 
          $\Phi^{\mathsf{L}}_{0}\colon \Mod*(\mathsf{A}_{0}|\mathsf{A}_{0}) \rightleftharpoons \Mod*[Ex](U|{}_{\mathsf{A}_{0}})\colon\Omega^{\mathsf{L}}_{0}$ 
          is an equivalence. 
      \item
          The adjoint pair 
          $\Phi^{\mathsf{R}}_{0}\colon \Mod*(\mathsf{A}_{0}|\mathsf{A}_{0}) \rightleftharpoons \Mod*[Ex]({}_{\mathsf{A}_{0}}|U)\colon\Omega^{\mathsf{R}}_{0}$ 
          is an equivalence.
    \end{enumerate}
  \end{remark}

% \subsection{Morita-type equivalence + quasi-finiteness implies strong identity condition}

  The assumption in \zcref{thm:Adjeq-SIC} can be generalized into the following form.

  \begin{theorem}\label{thm:AdjeqR-SIC}
    Suppose either 
    \begin{enumerate}
      \item $\mathfrak{L}^{0} \in \Mod*[Gr](U|{}_R)$, $\mathsf{A}_{0}^{R|\vee} \cong \mathsf{A}_{0}$ as $(\mathsf{A}_{0}|R)$-bimodules, and \zcref[noname]{PhiOmegaR} is verified; or
      \item $\mathfrak{R}^{0} \in \Mod*[Gr]({}_R|U)$, $\mathsf{A}_{0}^{\vee|R} \cong \mathsf{A}_{0}$ as $(R|\mathsf{A}_{0})$-bimodules, and \zcref[noname]{PhiOmegaL} is verified.
    \end{enumerate} 
    Then, the strong identity condition (\SIC) holds.
  \end{theorem}
  \begin{proof}
    Let's assume (i). The proof in case (ii) is similar. 
    With the assumptions, the following variant of \zcref{lem:LisDualOfR} is true:
    \[
      \mathfrak{R}^{0} \cong (\mathfrak{L}^{0})^{\dagger|R}.
    \]

    Consider the following \emph{homogeneous} isomorphisms
    \[
      \mathfrak{L}^{0}_{\bullet}\otimes_{R}\mathfrak{R}^{0}_{-\circ}
      \cong
      \mathfrak{L}^{0}_{\bullet}\otimes_{R}(\mathfrak{L}^{0})^{\dagger|R}_{-\circ} 
      =
      \mathfrak{L}^{0}_{\bullet}\otimes_{R}(\mathfrak{L}^{0}_{\circ})^{\vee|R}.
    \]
    In particular, for each $n\ge 0$, we have
    \[
      \mathfrak{L}^{0}_{n}\otimes_{R}\mathfrak{R}^{0}_{-n} \cong 
      \mathfrak{L}^{0}_{n}\otimes_{R}(\mathfrak{L}^{0}_{n})^{\vee|R}
      =
      \operatorname{End}_{|R}(\mathfrak{L}^{0}_{n}).
    \]
    Let $\one_{n}$ be the image of $\id\in\operatorname{End}_{|R}(\mathfrak{L}^{0}_{n})$ under the map $\mathfrak{L}^{0}_{n}\otimes_{R}\mathfrak{R}^{0}_{-n} \to \mathfrak{A}_{n}$. We conclude that the elements $(\one_{n})_{n\in\N}$ verify the strong identity condition.
  \end{proof}

  \begin{remark}
    All modules involved in the proof of \zcref{thm:AdjeqR-SIC} are in the small categories $\Mod*[Ex](U|{}_R)$, $\Mod*[Ex]({}_R|U)$, $\Mod*(\mathsf{A}_{0}|{}_R)$, and $\Mod*({}_R|\mathsf{A}_{0})$. Thus, we can also deduce the strong identity condition from one of the following assumptions:
    \begin{enumerate}
      \item
          The adjoint pair 
          $\Phi^{\mathsf{L}}_{0}\colon \Mod*(\mathsf{A}_{0}|{}_R) \rightleftharpoons \Mod*[Ex](U|{}_R)\colon\Omega^{\mathsf{L}}_{0}$ 
          is an equivalence. 
      \item
          The adjoint pair 
          $\Phi^{\mathsf{R}}_{0}\colon \Mod*({}_R|\mathsf{A}_{0}) \rightleftharpoons \Mod*[Ex]({}_R|U)\colon\Omega^{\mathsf{R}}_{0}$ 
          is an equivalence.
    \end{enumerate}
  \end{remark}
  \begin{remark}
    The two conditions in \zcref{thm:AdjeqR-SIC} are equivalent. This can be seen from the proof. In literatures (e.g. \cite{Kadison99,Lam99}), the equivalent conditions ``$\mathsf{A}_{0}^{R|\vee} \cong \mathsf{A}_{0}$ as $(\mathsf{A}_{0}|R)$-bimodules'' and ``$\mathsf{A}_{0}^{\vee|R} \cong \mathsf{A}_{0}$ as $(R|\mathsf{A}_{0})$-bimodules'' characterize \emph{Frobenius extensions}. 
    For this reason, we propose to call them the \concept{Frobenius conditions}.
    For example, if both $R$ and $\mathsf{A}_{0}$ are semisimple, then, by \emph{Wedderburn-Artin}, these two conditions are satisfied.
  \end{remark}

  Recall that a $\k$-algebra is \emph{augmented} if it admits a $\k$-algebra homomorphism to $\k$. 
  Frobenius algebras are augmented. But the latter is more general.
  \begin{theorem}\label{thm:Adjeqk-SIC}
    Suppose $\mathsf{A}_{0}$ is augmented and either
    \begin{enumerate}
      \item $\Phi^{\mathsf{L}}_{0}(\k) \in \Mod*[Gr](U|)$ and \zcref[noname]{PhiOmegaR} is verified; or
      \item $\Phi^{\mathsf{R}}_{0}(\k) \in \Mod*[Gr](|U)$ and \zcref[noname]{PhiOmegaL} is verified.
    \end{enumerate} 
    Then, the strong identity condition (\SIC) holds.
  \end{theorem}
  \begin{proof}
    Let's assume (i). The proof in case (ii) is similar. 
    With the assumptions, by \zcref{lem:induced-dual}, we have the \emph{homogeneous} isomorphism
    \[
      \Phi^{\mathsf{\k}}_{0}(\k) \cong \Phi^{\mathsf{L}}_{0}(\k)^{\dagger|\k}.
    \]
    Consequently, we have the following homogeneous isomorphism
    \[
      \mathfrak{L}^{0}_{\bullet}\otimes_{\mathsf{A}_{0}}\k\otimes_{\mathsf{A}_{0}}\mathfrak{R}^{0}_{-\circ}
      = 
      \Phi^{\mathsf{L}}_{0}(\k)_{\bullet}\otimes\Phi^{\mathsf{R}}_{0}(\k)_{-\circ}
      \cong
      \Phi^{\mathsf{L}}_{0}(\k)_{\bullet}\otimes\Phi^{\mathsf{L}}_{0}(\k)^{\dagger|\k}_{-\circ} 
      =
      \Phi^{\mathsf{L}}_{0}(\k)_{\bullet}\otimes(\Phi^{\mathsf{L}}_{0}(\k)_{\circ})^{\vee|\k}.
    \]
    In particular, we have the following identification:
    \[
      \mathfrak{L}^{0}_{n}\otimes_{\mathsf{A}_{0}}\k\otimes_{\mathsf{A}_{0}}\mathfrak{R}^{0}_{-m} \cong 
      \Phi^{\mathsf{L}}_{0}(\k)_{n}\otimes(\Phi^{\mathsf{L}}_{0}(\k)_{m})^{\vee|\k}
      =
      \operatorname{Hom}_{\k}(\Phi^{\mathsf{L}}_{0}(\k)_{m},\Phi^{\mathsf{L}}_{0}(\k)_{n}),
    \]
    where the last equality is the only place we need $\Phi^{\mathsf{L}}_{0}(\k) \in \Mod*[Gr](U|)$.

    Then, identities in $\operatorname{End}_{\k}(\Phi^{\mathsf{L}}_{0}(\k)_{n})$ produces a family $(\mathfrak{e}_{n})_{n\in\N}$ of elements in $\mathfrak{L}^{0}_{\bullet}\otimes_{\mathsf{A}_{0}}\k\otimes_{\mathsf{A}_{0}}\mathfrak{R}^{0}_{-\bullet}$. 
    Since $\mathsf{A}_{0}$ is augmented, we have a split surjective ring homomorphism $\mathsf{A}_{0} \to \k$. Via (more precisely, the splitting of) the augmentation, we can lift the family $(\mathfrak{e}_{n})_{n\in\N}$ to a family of elements $(\one_{n})_{n\in\N}$ in $\mathfrak{A}_{\bullet}$.
    We conclude that the elements $(\one_{n})_{n\in\N}$ verify the strong identity condition.
    To see this, one only need to apply the augmentation to the constructions $\mathfrak{L}^{0}_{n}\otimes_{\mathsf{A}_{0}}-\otimes_{\mathsf{A}_{0}}\mathfrak{R}^{0}_{-m}$.
  \end{proof}
  \begin{remark}
    All modules involved in the proof of \zcref{thm:Adjeqk-SIC} are in the small categories $\Mod*[Ex](U|)$, $\Mod*[Ex](|U)$, $\Mod*(\mathsf{A}_{0}|)$, and $\Mod*(|\mathsf{A}_{0})$. Thus, we can also deduce the strong identity condition from one of the following assumptions:
    \begin{enumerate}
      \item
          The adjoint pair 
          $\Phi^{\mathsf{L}}_{0}\colon \Mod*(\mathsf{A}_{0}|) \rightleftharpoons \Mod*[Ex](U|)\colon\Omega^{\mathsf{L}}_{0}$ 
          is an equivalence. 
      \item
          The adjoint pair 
          $\Phi^{\mathsf{R}}_{0}\colon \Mod*(|\mathsf{A}_{0}) \rightleftharpoons \Mod*[Ex](|U)\colon\Omega^{\mathsf{R}}_{0}$ 
          is an equivalence.
    \end{enumerate}
  \end{remark}

\section{Morita-type equivalences of higher-levels }\label{sec:HigherMorita}
  We are going to extend the Morita-type equivalence (\zcref{thm:PhiOmega0}) to higher levels. However, due to \zcref{rem:noPhiOmegan}, we cannot expect $\Phi^{\mathsf{L}}_{n}\dashv\Omega^{\mathsf{L}}_{n}$ to be an adjoint equivalence when $n>0$. Some modifications must be made. 
  Although there are many possible ways to do so, we will only present one of them here. 
  Based on \zcref{prop:adj:bPhiOmega}, one may ask for:
  \begin{enumerate}[leftmargin=5\parindent,font=\nota]
    \myItem[\textup{($\mathbf{\Phi}\Omega\mathsf{L}_{R}^{n}$)}]\label{PhiOmegaLn} 
        The adjoint pair 
        $\mathbf{\Phi}^{\mathsf{L}}_{n}\colon \Mod(\mathscr{A}^{\mathsf{L}}_{n}|{}_R) \rightleftharpoons \Mod[Ex](U|{}_R)\colon\Omega^{\mathsf{L}}_{n}$ 
        is an equivalence. 
    \myItem[\textup{($\mathbf{\Phi}\Omega\mathsf{R}_{R}^{n}$)}]\label{PhiOmegaRn} 
        The adjoint pair 
        $\mathbf{\Phi}^{\mathsf{R}}_{n}\colon \Mod({}_R|\mathscr{A}^{\mathsf{R}}_{n}) \rightleftharpoons \Mod[Ex]({}_R|U)\colon\Omega^{\mathsf{R}}_{n}$ 
        is an equivalence.
  \end{enumerate}

  In this section, we assume that $U$ admits a strong identity expansion, saying $(e_{n})_{n\in\N}$. 

\subsection{Actions of mode blocks}
  First note that, the mode transition algebra $\mathfrak{A}$ acts on exhaustive modules:
  \begin{lemma}\label{lem:action-of-fA-on-Omega}
    For any exhaustive $(U|{}_R)$-bimodule $W$, the left action of $U$ on $W$ induces well-defined action maps:
    \[
      \star\colon\mathfrak{A}_{p,-q}\otimes_{U_{0}}\Omega^{\mathsf{L}}_{n}(W)\longrightarrow 
      \begin{dcases*}
        \Omega^{\mathsf{L}}_{p}(W) & if $q\le n$,\\
        0 & otherwise.
      \end{dcases*}
    \]
    The similar statement holds for $({}_R|U)$-bimodule.
  \end{lemma}
  \begin{proof}
    Recall that, we have the counit map $\varepsilon_{W}\colon\Phi^{\mathsf{L}}_{q}\Omega^{\mathsf{L}}_{q}(W)\to W$ given by the $U$-action on $W$. Note that $\Phi^{\mathsf{L}}_{q}\Omega^{\mathsf{L}}_{q}(W)=\mathfrak{L}^{q}\otimes_{U_{0}}\Omega^{\mathsf{L}}_{q}(W)$. Restricting the counit map to $\Omega^{\mathsf{L}}_{q}(W)_{p-q}$, we obtain an action map:
    \[
      \mathfrak{L}^{q}_{p-q}\otimes_{U_{0}}\Omega^{\mathsf{L}}_{q}(W)\longrightarrow\Omega^{\mathsf{L}}_{p}(W).
    \]
    Composing it with map $\mu\colon\mathfrak{A}_{p,-q} \to \mathfrak{L}^{q}_{p-q}$ in \zcref{thm:decomposition_of_fL}, we obtain the desired action map.
    Since $\Omega^{\mathsf{L}}_{q}(W)\subset\Omega^{\mathsf{L}}_{n}(W)$ whenever $q\le n$, the action maps $\mathfrak{A}_{p,-q}\otimes_{U_{0}}\Omega^{\mathsf{L}}_{q}(W)\to\Omega^{\mathsf{L}}_{p}(W)$ extends to $\Omega^{\mathsf{L}}_{n}(W)$ for $q\le n$.
  \end{proof}
  \begin{lemma}\label{lem:thickZhu}
    Assuming $U$ admits a strong identity expansion. Then, the algebras $\mathscr{A}^{\mathsf{L}}_{n}$ and $\mathscr{A}^{\mathsf{R}}_{n}$ are isomorphic to the following block matrix algebras:
    \[
      \prod_{p,q=0}^{n}\mathfrak{A}_{p,-q}.
    \]
    Furthermore, the action of the blocks $\mathfrak{A}_{p,-q}$ on each $\Omega^{\mathsf{L}}_{n}(W)$ (resp. $\Omega^{\mathsf{R}}_{n}(W)$) coincides with that in \zcref{lem:action-of-fA-on-Omega}.
  \end{lemma}
  \begin{proof}
    The identification is due to the decompositions in \zcref{thm:decomposition_of_fL}. The second assertion is straightforward.
  \end{proof}
  Hence, under the assumption of {\SIC}, we will omit the superscripts and simply write $\mathscr{A}_{n}$ for the above algebra.

  \begin{lemma}\label{lem:grOmegabPhin}
    Assuming $U$ admits a strong identity expansion.
    Then, for any left $\mathscr{A}_{n}$-module $M$, the $\gr\Omega$-grading on $\mathbf{\Phi}^{\mathsf{L}}_{n}(M)$ is given by
    \[
      \mathbf{\Phi}^{\mathsf{L}}_{n}(M)_{(m)}=\bigoplus_{k=0}^{n}\mathfrak{A}_{m,-k}\star M.
    \]
    Similar for right $\mathscr{A}_{n}$-modules.
  \end{lemma}
  \begin{proof}
    By \zcref{lem:splitOmega}, $\mathbf{\Phi}^{\mathsf{L}}_{n}(M)_{(m)}$ can be interpreted as $e_{m}\mathfrak{L}^{n}\otimes_{\mathscr{A}_{n}}M$. 
    Furthermore, we have
    \begin{align*}
      \mathbf{\Phi}^{\mathsf{L}}_{n}(M)_{(m)} 
      &\overset{\ref{lem:grOmegaPhin}}{=}
      (\bigoplus_{p=m-n}^{m}e_{m}\mathfrak{L}^{n}_{p})\otimes_{\mathscr{A}_{n}} M \\ 
      &\overset{\ref{thm:decomposition_of_fL}}{=}
      (\bigoplus_{q=0}^{n}\mathfrak{A}_{m,-q})\otimes_{\mathscr{A}_{n}} M \\
      &\overset{\ref{lem:thickZhu}}{=}
      (\bigoplus_{q=0}^{n}\mathfrak{A}_{m,-q})\bigotimes_{\prod_{p,q=0}^{n}\mathfrak{A}_{p,-q}} (\bigoplus_{k=0}^{n} \one_{k}M) 
      = \bigoplus_{k=0}^{n}\mathfrak{A}_{m,-k}\star M.\qedhere
    \end{align*}
  \end{proof}

\subsection{Morita-type equivalences for higher levels}
  \begin{theorem}\label{thm:bPhiOmega}
    If $U$ admits a strong identity expansion (\SIE), then \zcref[noname]{PhiOmegaLn,PhiOmegaRn} are verified for all $n \in \N$.
  \end{theorem}
  \begin{proof}
    We only focus on the adjoint pair $\mathbf{\Phi}^{\mathsf{L}}_{n}\dashv\Omega^{\mathsf{L}}_{n}$. The other one is similar. By \zcref{lem:AdjEq}, it boils down to verifying that 
    \begin{enumerate}
      \item the unit map $\eta_{M}\colon M\mapsto\Omega^{\mathsf{L}}_{n}\mathbf{\Phi}^{\mathsf{L}}_{n}(M)$ is bijective for any left $\mathscr{A}_{n}$-module $M$; and
      \item the counit map $\varepsilon_{W}\colon\mathbf{\Phi}^{\mathsf{L}}_{n}\Omega^{\mathsf{L}}_{n}(W)\to W$ is surjective for any exhaustive left $U$-module $W$.
    \end{enumerate}
    The proof of (1) can be deduced from \zcref{lem:grOmegabPhin}. 
    But a more straightforward way is to note that $\Omega^{\mathsf{L}}_{n}\mathbf{\Phi}^{\mathsf{L}}_{n}(M) = \Omega^{\mathsf{L}}_{n}(\mathfrak{L}^{n})\otimes_{\mathscr{A}_{n}}M = M$ under the assumption of \SIE.
    The proof of (2) is similar to that in \zcref{thm:PhiOmega0}.
  \end{proof}

  \begin{remark}
    Combining this theorem with \zcref{thm:PhiOmega0}, which can be seen as the case $n=0$ of the above, we obtain an equivalence:
    \begin{equation}\label{eq:EquivA0-An}
      \begin{tikzcd}
        \Omega^{\mathsf{L}}_{n}\circ\Phi^{\mathsf{L}}_{0}\colon 
        \Mod(\mathsf{A}_{0}|{}_R) \ar[r,"\simeq"] & 
        \Mod(\mathscr{A}_{n}|{}_R) \ar[l] \colon\Omega^{\mathsf{L}}_{0}\circ\mathbf{\Phi}^{\mathsf{L}}_{n}.
      \end{tikzcd}
    \end{equation}
    Be aware that this is NOT an adjoint pair.
  \end{remark}

\subsection{For quasi-rigid modules}
  We have the following corollary for quasi-rigid modules.
  \begin{corollary}\label{coro:HigherMoritaOrdinary}
    Let $n\in\N$. Suppose either 
    \begin{enumerate}
      \item $U$ is \emph{quasi-rigid}: $\mathfrak{L}^{n} \in \Mod*[Gr](U|{}_{\mathsf{A}_{n}})$ and $\mathfrak{R}^{n} \in \Mod*[Gr]({}_{\mathsf{A}_{n}}|U)$; or 
      \item $R$ is semisimple and $U$ is \emph{weakly quasi-finite}: components of $\mathfrak{L}^{n},\mathfrak{R}^{n}$ are finite over $\mathsf{A}_{n}$.
    \end{enumerate}
    If $U$ admits a strong identity expansion (\SIE), then the adjoint pairs
    \[
      \mathbf{\Phi}^{\mathsf{L}}_{n}\colon \Mod*(\mathscr{A}_{n}|{}_R) \rightleftharpoons \Mod*[Ex](U|{}_R)\colon\Omega^{\mathsf{L}}_{n}
      \txand
      \mathbf{\Phi}^{\mathsf{R}}_{n}\colon \Mod*({}_R|\mathscr{A}_{n}) \rightleftharpoons \Mod*[Ex]({}_R|U)\colon\Omega^{\mathsf{R}}_{n}
    \]
    are adjoint equivalences.
  \end{corollary}
  \begin{proof}
    We need to show that $\mathbf{\Phi}^{\mathsf{L}}_{n}$ sends a $(\mathscr{A}_{n}|{}_R)$-bimodule $M$ with rigid underlying $R$-module to a quasi-rigid exhaustive $(U|{}_R)$-bimodule. 
    It suffices to show that each $\mathbf{\Phi}^{\mathsf{L}}_{n}(M)_{(m)}$ is rigid. 
    This follows from \zcref{lem:grOmegabPhin} and the assumption on $M$.  
    Indeed, in case (i), by \zcref{lem:rigidmodule}, it remains to show that each block $\mathfrak{A}_{p,-q}$ is rigid over $\mathfrak{A}_{q,-q}$. This is done by \zcref{prop:decomposition_of_fA} and the definition. Case (ii) is straightforward as one only needs to check finiteness.
  \end{proof}

\subsection{Projectivity of $\mathfrak{L}^{n}$ and $\mathfrak{R}^{n}$}
  We end this section with another characterization of the strong identity condition. 

  First, note that any equivalence of abelian categories must be exact. We conclude that
  \begin{lemma}\label{lem:bPhiOmega-Proj}
    If \zcref[noname]{PhiOmegaLn} (resp. \zcref[noname]{PhiOmegaRn}) is verified, then $\mathfrak{L}^{n}\otimes R$ (resp. $R\otimes\mathfrak{R}^{n}$) is a projective object in $\Mod[Ex](U|{}_R)$ (resp. $\Mod[Ex]({}_R|U)$).
  \end{lemma}
  \begin{proof}
    Suppose 
    $\Omega^{\mathsf{L}}_{n} \colon \Mod[Ex](U|{}_R) \to \Mod(\mathscr{A}^{\mathsf{L}}_{n}|{}_R)$ is an equivalence. 
    Then, in particular, $\Omega^{\mathsf{L}}_{n}$ is exact. 
    By \zcref{rem:OmegaHomLR}, this means $\mathfrak{L}^{n}\otimes R$ is a projective object in $\Mod[Ex](U|{}_R)$.
  \end{proof}

  Now, we look at the sequence of canonical projections:
  \begin{align*}
    \cdots \longrightarrow 
    \mathfrak{L}^{n} \overset{\pi_{n}}{\longrightarrow} 
    \mathfrak{L}^{n-1} \longrightarrow 
    \cdots \longrightarrow 
    \mathfrak{L}^{0} \longrightarrow 
    0 \\
    \cdots \longrightarrow 
    \mathfrak{R}^{n} \overset{\pi_{n}}{\longrightarrow} 
    \mathfrak{R}^{n-1} \longrightarrow 
    \cdots \longrightarrow 
    \mathfrak{R}^{0} \longrightarrow 
    0
  \end{align*}
  \begin{theorem}\label{thm:Proj-SIE}
    Suppose all $\mathfrak{L}^{n}$ are projective in $\Mod[Ex](U|)$ and all $\mathfrak{R}^{n}$ are projective in $\Mod[Ex](|U)$. 
    Then, $U$ admits a strong identity expansion (\SIE).
  \end{theorem}
  \begin{proof}
    Under the assumption, all projections $\pi_{n}\colon\mathfrak{L}^{n}\to\mathfrak{L}^{n-1}$ and $\pi_{n}\colon\mathfrak{R}^{n}\to\mathfrak{R}^{n-1}$ split as morphisms of graded left $U$-modules and graded right $U$-modules, respectively. 
    Then, following the same argument in the proof of \zcref{thm:SIE-Omega}, we can construct a strong identity expansion.
  \end{proof}

  As a consequence, we have the following unconditional implication.
  \begin{corollary}
    If \zcref[noname]{PhiOmegaLn,PhiOmegaRn} are verified for all $n \in \N$ with $R=\k$, then $U$ admits a strong identity expansion (\SIE).
  \end{corollary}

  Comparing to \zcref{thm:Adjeq-SIC,thm:AdjeqR-SIC}, this needs Morita-type equivalences at all levels, but does not need any assumption on quasi-rigidity or quasi-finiteness.

\part{Structural Properties for almost-canonically seminormed rings}\label{part3}
In this part, we develops structural features of almost-canonically seminormed rings and their module categories, such as the notion of rationality, the compatibility with tensor product, and the end formula for the mode transition algebra.

\section{Rationality}\label{sec:rationality}
  The notion of \emph{rationality} refers to the semisimplicity of certain category of good modules.
  \begin{definition}\label{def:rational}
    We say that $U$ is \emph{$R$-rational} if any positively-graded $(U|{}_R)$-bimodule is a direct sum of quasi-rigid simple ones. 
  \end{definition} 
  \begin{remark}\label{rem:rational=semisimple}
    The definition mimics the rationality of vertex operator algebras introduced by Zhu in \cite[Definition 1.2.4]{Z96}. Later, in \cite[Theorem 8.1]{DLM1}, Dong-Li-Mason proved that Zhu's rationality is equivalent to the semisimplicity of the category of admissible modules.
  \end{remark}

  In this section, we discuss the relation between rationality, semisimplicity of graded modules, and the strong identity condition.

\subsection{Preparations on simple modules}
  The structure of simple modules are stiff. 
  \begin{lemma}\label{lem:S-grOmega}
    Let $S$ be a simple $(U|{}_R)$-bimodule that is gradable (see \zcref{def:gradable}). Then, its $\Omega$-filtration splits, and the $\gr\Omega$-grading coincides with any grading on $S$ up to a degree-shifting.
    Similar for simple $({}_R|U)$-bimodules.
  \end{lemma}
  \begin{proof}
    This proof is inspired by \cite[Proposition 5.4]{DLM1}. It turns out to be valid in our general setting. 
    
    For simplicity, we start with a grading on $S$ such that $S_{n} = 0$ for $n<0$ and $S_{0} \neq 0$. Then, we have $\bigoplus_{k=0}^{n}S_{k} \subset \Omega^{\mathsf{L}}_{n}(S)$. The assertions  boil down to show that the inclusion is an equality. Indeed, if there is a nonzero $w\in S_{m}$ ($m>n$) belongs to $\Omega^{\mathsf{L}}_{n}(S)$, then
    \[ S = U wR = U_{\ge -n} wR, \]
    where the first equality follows from the simplicity of $S$ and the second follows from $w\in\Omega^{\mathsf{L}}_{n}$. 
    But $U_{\ge -n} wR \subset S_{\ge m-n}$, which is a proper subspace of $S$ since it does not contain $S_{0}$. This is a contradiction. 
  \end{proof}
  \begin{remark}
    The lemma requires $S$ to be gradable. In general, a simple exhaustive module needs not to be gradable.
  \end{remark}

  \begin{lemma}\label{lem:S-component}
    Let $S$ be a simple graded $(U|{}_R)$-bimodule. Then, any components $S_{n}$ is a simple $(U_{0}|{}_{R})$-bimodule.
    Similar for simple $({}_R|U)$-bimodules.
  \end{lemma}
  \begin{proof}
    Suppose there is a nonzero $(U_{0}|{}_R)$-submodule $M$ of $S_{n}$. Then, we have 
    \[
      S = U M = \bigoplus_{p\in\Z} U_{p}M \subset \bigoplus_{p\in\Z} S_{p+n} = S,
    \]
    where the inclusion follows since $M\subset S_{n}$. This forces $M=S_{n}$. 
  \end{proof}
  \begin{warning}
    Being simple as a graded module is weaker than being simple as an exhaustive module. 
    It could happen that a graded module has non-graded proper submodules, but no graded proper submodules.
  \end{warning}
  \begin{remark}\label{rem:semisimple-bimodules}
    For an $(U_{0}|{}_R)$-bimodule, being simple as a left $U_{0}$-module implies being simple as an $(U_{0}|{}_R)$-bimodule since any $(U_{0}|{}_R)$-submodule is in particular a left $U_{0}$-submodule. 
    By the similar logic, for an $\mathsf{A}_{n}$-module, being simple as a $U_{0}$-module implies that being simple as an $\mathsf{A}_{n}$-module.
  \end{remark}

  \begin{lemma}\label{lem:lifting-isomorphism}
    Let $f\colon W^{\mathtt{I}}\to W^{\mathtt{II}}$ be a homomorphism between two simple $(U|{}_R)$-bimodules. Then, $f$ is an isomorphism if and only if $\Omega_{0}(f)$ is an isomorphism.
    Similar for simple $({}_R|U)$-bimodules.
  \end{lemma}
  \begin{proof}
    First note that, since $W^{\mathtt{II}}$ is simple, $f$ is either surjective or zero. Also, since $W^{\mathtt{I}}$ is simple, $f$ is either injective or zero. Hence, the statement follows.
  \end{proof}

\subsection{Rationality and semisimplicity}
  Now, we declare the relation between rationality and semisimplicity.
  \begin{lemma}\label{lem:semisimple}
    The semisimplicity of $\Mod[Gr](U|{}_R)$ implies semisimplicity of $\Mod(\mathsf{A}_{n}|{}_R)$ for all $n$. 
    For each $n\in\N$, suppose either 
    \begin{enumerate}
      \item $U$ is \emph{quasi-rigid}: $\mathfrak{L}^{n} \in \Mod*[Gr](U|{}_{\mathsf{A}_{n}})$ and $\mathfrak{R}^{n} \in \Mod*[Gr]({}_{\mathsf{A}_{n}}|U)$; or 
      \item $R$ is semisimple and $U$ is \emph{weakly quasi-finite}: components of $\mathfrak{L}^{n},\mathfrak{R}^{n}$ are finite over $\mathsf{A}_{n}$.
    \end{enumerate}
    Then, the semisimplicity of $\Mod*[Gr](U|{}_R)$ implies that of $\Mod*(\mathsf{A}_{n}|{}_R)$. 
    Similar statements hold for the $({}_R|U)$-bimodules.
  \end{lemma}
  \begin{proof}
    For any $(\mathsf{A}_{n}|{}_R)$-bimodule $M$, there is the graded $(U|{}_R)$-bimodule $\Phi^{\mathsf{L}}_{n}(M)$ satisfying $\Phi^{\mathsf{L}}_{n}(M)_{0}\cong M$. 
    Hence, $M$ is completely reducible by \zcref{lem:S-component}.
    The second statement follows from \zcref{prop:adj:ordinary}.
  \end{proof}

  \begin{lemma}\label{lem:lifting-simple}
    The induction functors $\Phi^{\mathsf{L}}_{n}$ and $\Phi^{\mathsf{R}}_{n}$ send indecomposable modules to indecomposable graded (or exhaustive) modules. 
  \end{lemma}
  \begin{proof}
    This is essentially \cite[Lemma 1.2]{DGK22}, stated in the generality of almost-canonically seminormed rings. We generalize the proof to our setting.

    Suppose $M$ is an indecomposable $(\mathsf{A}_{n}|{}_R)$-bimodule and suppose $\Phi^{\mathsf{L}}_{n}(M)=W^{\mathtt{I}}\oplus W^{\mathtt{II}}$ is a decomposition into graded submodules. Then, we have $M\cong\Phi^{\mathsf{L}}_{n}(M)_{0} = W^{\mathtt{I}}_{0}\oplus W^{\mathtt{II}}_{0}$. 
    Since $M$ is indecomposable, we must have either $W^{\mathtt{I}}_{0}=0$ or $W^{\mathtt{II}}_{0}=0$. Suppose $W^{\mathtt{II}}_{0}=0$. Consider the following commutative diagram provided by the counit $\varepsilon$ of the adjoint pair $\Phi^{\mathsf{L}}_{n}\dashv(-)_{0}$:
    \[
      \begin{tikzcd}
        \Phi^{\mathsf{L}}_{n}(W^{\mathtt{I}}_{0})\ar[r,"\varepsilon_{W^{\mathtt{I}}}"]\ar[d] & W^{\mathtt{I}}\ar[d] \\
        \Phi^{\mathsf{L}}_{n}(W^{\mathtt{I}}_{0}\oplus W^{\mathtt{II}}_{0})\ar[r,"\varepsilon_{W^{\mathtt{I}}\oplus W^{\mathtt{II}}}"] & W^{\mathtt{I}}\oplus W^{\mathtt{II}}
      \end{tikzcd}
    \]
    where the vertical arrows are given by the inclusion $W^{\mathtt{I}}\subset W^{\mathtt{I}}\oplus W^{\mathtt{II}}$. 
    The left vertical arrow is an isomorphism since $W^{\mathtt{II}}_{0}=0$. The bottom horizontal arrow is an isomorphism since $\Phi^{\mathsf{L}}_{n}(W^{\mathtt{I}}_{0}\oplus W^{\mathtt{II}}_{0})$ is precisely $\Phi^{\mathsf{L}}_{n}(M)$. Therefore, the inclusion $W^{\mathtt{I}}\subset W^{\mathtt{I}}\oplus W^{\mathtt{II}}=\Phi^{\mathsf{L}}_{n}(M)$ has to be an isomorphism. This shows that $\Phi^{\mathsf{L}}_{n}(M)$ is indecomposable as a graded $(U|{}_R)$-bimodule.

    Now, we show that $\Phi^{\mathsf{L}}_{n}(M)$ is also indecomposable as an exhaustive $(U|{}_R)$-bimodule.
    Suppose $\Phi^{\mathsf{L}}_{n}(M) = W^{\mathtt{I}}\oplus W^{\mathtt{II}}$ is a decomposition into exhaustive submodules. 
    Then, we have $M\subset \Omega^{\mathsf{L}}_{n}\Phi^{\mathsf{L}}_{n}(M) = \Omega^{\mathsf{L}}_{n}(W^{\mathtt{I}}) \oplus \Omega^{\mathsf{L}}_{n}(W^{\mathtt{II}})$. Since $M$ is indecomposable, we must have either $M \subset \Omega^{\mathsf{L}}_{n}(W^{\mathtt{I}})$ or $M \subset \Omega^{\mathsf{L}}_{n}(W^{\mathtt{II}})$. Suppose $M \subset \Omega^{\mathsf{L}}_{n}(W^{\mathtt{I}})$. Consider the following commutative diagram provided by the counit $\varepsilon$ of the adjoint pair $\Phi^{\mathsf{L}}_{n}\dashv\Omega^{\mathsf{L}}_{n}$:
    \[
      \begin{tikzcd}
        & 
        \Phi^{\mathsf{L}}_{n}(\Omega^{\mathsf{L}}_{n}(W^{\mathtt{I}}))\ar[r,"\varepsilon_{W^{\mathtt{I}}}"]\ar[d] & W^{\mathtt{I}}\ar[d] \\
        \Phi^{\mathsf{L}}_{n}(M) \ar[ur,bend left]\ar[r,"\Phi^{\mathsf{L}}_{n}(\eta_{M})"] &
        \Phi^{\mathsf{L}}_{n}(\Omega^{\mathsf{L}}_{n}(\Phi^{\mathsf{L}}_{n}(M)))\ar[r,"\varepsilon_{W^{\mathtt{I}}\oplus W^{\mathtt{II}}}"] & W^{\mathtt{I}}\oplus W^{\mathtt{II}}
      \end{tikzcd}
    \]
    The bottom composition equals the identity of $\Phi^{\mathsf{L}}_{n}(M)$ by the triangle identity of the adjunction. Hence, the last vertical arrow is surjective. But this forces $W^{\mathtt{II}} = 0$. This shows that $\Phi^{\mathsf{L}}_{n}(M)$ is indecomposable as an exhaustive $(U|{}_R)$-bimodule.
  \end{proof}

  \begin{lemma}\label{lem:semisimple-rigid}
    If $U$ is $R$-rational, then any simple $(\mathsf{A}_{n}|{}_R)$-bimodule is rigid over $R$. 
  \end{lemma}
  \begin{proof}
    Suppose $M$ is a simple $(\mathsf{A}_{n}|{}_R)$-bimodule. By \zcref{lem:lifting-simple}, $\Phi^{\mathsf{L}}_{n}(M)$ is indecomposable and hence simple since $\Mod[Gr](U|{}_R)$ is semisimple. It is further quasi-rigid by the $R$-rationality of $U$. In particular, $M = \Phi^{\mathsf{L}}_{n}(M)_{0}$ is rigid over $R$.
  \end{proof}

  We are now able to declare the relation between rationality and semisimplicity. This can be viewed as an analogue of the VOA fact in \zcref{rem:rational=semisimple}.
  \begin{theorem}\label{thm:rationality-semisimplicity}
    Then, the following are equivalent:
    \begin{enumerate}
      \item $U$ is $R$-rational; and
      \item $\Mod[Gr](U|{}_R)$ is semisimple and any simple $(\mathsf{A}_{n}|{}_R)$-module is rigid over $R$.
    \end{enumerate}
    Furthermore, for $R = \mathsf{A}_{0}$, these conditions are equivalent to:
    \begin{enumerate}[resume]
      \item $\Mod[Gr](U|{}_{\mathsf{A}_{0}})$ is semisimple; and
      \item $\Mod[Ex](U|{}_{\mathsf{A}_{0}})$ is semisimple.
    \end{enumerate}
  \end{theorem}
  \begin{proof}
    (i) $\implies$ (ii) by \zcref{lem:semisimple-rigid}. 
    (ii) $\implies$ (iii) and (iv) $\implies$ (iii) are clear.

    Suppose either (ii), (iii), or (iv). To show (i), it boils down to showing that: any simple graded $(U|{}_R)$-bimodule is quasi-rigid. 
    This is clear under the condition (ii) since, by \zcref{lem:S-component}, its each component is simple.
    Hence, in what follows, we take $R = \mathsf{A}_{0}$ and assume (iii). 

    Let $S$ be a simple graded $(U|{}_{\mathsf{A}_{0}})$-bimodule and let $S_{0}$ be its bottom degree, which is a simple $(\mathsf{A}_{0}|{}_{\mathsf{A}_{0}})$-bimodule by \zcref{lem:S-component}.
    Note that, by \zcref{lem:semisimple}, $\mathsf{A}_{0}$ is semisimple. Hence, any simple $(\mathsf{A}_{0}|{}_{\mathsf{A}_{0}})$-bimodule is a rigid right $\mathsf{A}_{0}$-module.
    By the adjunction \zcref{prop:adj:Phideg}, we have a morphism 
    \[
      \phi \colon \Phi^{\mathsf{L}}_{0}(S_{0}) \longrightarrow S.
    \]
    By \zcref{lem:lifting-simple} and the assumption, $\Phi^{\mathsf{L}}_{0}(S_{0})$ is simple. The above morphism cannot be zero, hence it has to be an isomorphism.

    Now, consider the canonical map 
    \[
      S = \Phi^{\mathsf{L}}_{0}(S_{0}) \longmapsto \left(\Phi^{\mathsf{L}}_{0}(S_{0})^{\dagger|R}\right)^{R|\dagger}.
    \]
    It is homogeneous and its restriction to the bottom is an isomorphism (since $S_{0}$ is rigid over $R$). Hence, by \zcref{lem:lifting-isomorphism}, it is an isomorphism. 
    This shows that each component $S_{n}$ of $S$ is a reflexive right $\mathsf{A}_{0}$-module. 
    Since $\mathsf{A}_{0}$ is semisimple, this can only happen when $S_{n}$ is a rigid right $\mathsf{A}_{0}$-module. 

    It remains to show (i) $\implies$ (iv). We leave it later to \zcref{coro:Ex(UA)-semisimple}.
  \end{proof}

  \begin{corollary}\label{coro:k-rational-A-rational}
    If $\Mod[Ex](U|)$ is semisimple, then $U$ is $\mathsf{A}_{0}$-rational.
  \end{corollary}
  \begin{proof}
    \zcref{rem:semisimple-bimodules} shows that, if $\Mod[Gr](U|)$ is semisimple, then $\Mod[Gr](U|{}_R)$ must also be semisimple for all $R$. 
    In particular, the condition (iii) in \zcref{thm:rationality-semisimplicity} is satisfied. Hence, $U$ is $\mathsf{A}_{0}$-rational.
  \end{proof}

\subsection{Rationality and the strong identity condition}\label{sec:rationality-SIC}
  As in the VOA situation (cf. \cite[Remark 3.4.6]{DGK23}), rationality implies the strong identity condition.
  \begin{theorem}\label{thm:rational}
    If $U$ is $\mathsf{A}_{0}$-rational, then the strong identity condition holds.
  \end{theorem}
  \begin{proof}
    First, by \zcref{lem:semisimple}, $\mathsf{A}_{0}$ is semisimple. Hence, by \emph{Wedderburn-Artin}, we have an idempotent decomposition of rings
    \[
      \mathsf{A}_{0} = B_{0} \times \cdots \times B_{s},
    \]
    where $B_{i}$'s are simple rings that are non-isomorphic as $(\mathsf{A}_{0}|\mathsf{A}_{0})$-bimodules.
    In particular, we have $(B_{i})^{\vee|\mathsf{A}_{0}} = (B_{i})^{\vee|B_{i}} \cong B_{i}$ and $(B_i)^{\vee|\mathsf{A}_{0}}\otimes_{\mathsf{A}_{0}}B_{j} = 0$ for all $i \neq j$. 
    
    Consider the $({}_{\mathsf{A}_{0}}|U)$-bimodules $\Phi^{\mathsf{L}}_{0}(B_{i})^{\dagger|\mathsf{A}_{0}}$ and $\Phi^{\mathsf{R}}_{0}(B_{i})$. 
    By \zcref{lem:lifting-simple}, both of them are simple. 
    Then, by \zcref{lem:lifting-isomorphism}, we see that the identification $(B_{i})^{\vee|\mathsf{A}_{0}} \cong B_{i}$ lifts to an isomorphism of simple $({}_{\mathsf{A}_{0}}|U)$-bimodules $\Phi^{\mathsf{L}}_{0}(B_{i})^{\dagger|\mathsf{A}_{0}} \cong \Phi^{\mathsf{R}}_{0}(B_{i})$. 
    
    Now, we have, for all $n\in\N$,
    \begin{equation}
    \label{eq:Mode-End-Phi}
    \begin{aligned}
      \mathfrak{A}_{n} 
      &= 
      \Phi^{\mathsf{L}}_{0}(\mathsf{A}_{0})_{n} \otimes_{\mathsf{A}_{0}} \Phi^{\mathsf{R}}_{0}(\mathsf{A}_{0})_{-n} 
      = 
      \bigoplus_{i=0}^{s}
        \Phi^{\mathsf{L}}_{0}(B_{i})_{n} \otimes_{\mathsf{A}_{0}} \Phi^{\mathsf{R}}_{0}(B_{i})_{-n} \\
      &= 
      \bigoplus_{i=0}^{s}
        \Phi^{\mathsf{L}}_{0}(B_{i})_{n} \otimes_{\mathsf{A}_{0}} \left(\Phi^{\mathsf{L}}_{0}(B_{i})\right)^{\dagger|\mathsf{A}_{0}}_{-n}
      = 
      \bigoplus_{i=0}^{s}
        \Phi^{\mathsf{L}}_{0}(B_{i})_{n} \otimes_{\mathsf{A}_{0}} \left(\Phi^{\mathsf{L}}_{0}(B_{i})_{n}\right)^{\vee|\mathsf{A}_{0}}
      \\
      &= 
      \bigoplus_{i=0}^{s}
        \End_{\mathsf{A}_{0}}(\Phi^{\mathsf{L}}_{0}(B_{i})_{n}),
    \end{aligned}
    \end{equation}
    where the last equality follows from that $\Phi^{\mathsf{L}}_{0}(B_{i})$ is quasi-rigid as a $(U|\mathsf{A}_{0})$-bimodule. This is the only place where we need this. 
    Let $\one_{n}$ be the element of $\mathfrak{A}_{n}$ corresponding to $\sum_{i=0}^{s}\id_{\Phi^{\mathsf{L}}_{0}(B_{i})_{n}}$. We conclude that the elements $(\one_{n})_{n\in\N}$ verify the strong identity condition.
  \end{proof}

  We give some immediate consequences of this theorem.
  \begin{corollary}\label{coro:rational-ex-gr}
    If $U$ is $\mathsf{A}_{0}$-rational, then any exhaustive $U$-module is gradable.
  \end{corollary}
  \begin{proof}
    This is due to \zcref{lem:splitOmega}.
  \end{proof}

  \begin{corollary}\label{coro:Ex(UA)-semisimple}
    If $U$ is $\mathsf{A}_{0}$-rational, then the category $\Mod[Ex](U|{}_{\mathsf{A}_{0}})$ is semisimple.
  \end{corollary}
  \begin{proof}
    This is due to \zcref{thm:PhiOmega0}, which says the category $\Mod[Ex](U|{}_{\mathsf{A}_{0}})$ is equivalent to $\Mod(\mathsf{A}_{0}|\mathsf{A}_{0})$, which is semisimple.
  \end{proof}

  Combining the theorem with the discussions before, we have:
  \begin{corollary}\label{coro:k-rational-coros}
    If the category $\Mod[Gr](U|)$ is semisimple, then 
    \begin{enumerate}
      \item All $\mathsf{A}_{n}$ are semisimple;
      \item $U$ is $\mathsf{A}_{0}$-rational;
      \item the strong identity condition holds;
      \item any exhaustive $U$-module is gradable; and
      \item the category $\Mod[Ex](U|)$ is semisimple.
    \end{enumerate}
    If we furthermore assume $U$ is $\k$-rational, then the $U$-modules $\mathfrak{L}^{n},\mathfrak{R}^{n}$ are quasi-rigid.
  \end{corollary}
  \begin{proof}
    \zcref{lem:semisimple} shows (i). 
    \zcref{coro:k-rational-A-rational} shows (ii).
    \zcref{thm:rational} plus (ii) shows (iii). 
    Then, (iii) with \zcref{lem:splitOmega} shows (iv).
    Combining (i) and (iii), by \zcref{thm:PhiOmega0}, we get (v). 
    Finally, assume $U$ is $\k$-rational. 
    Then, $\mathfrak{L}^{0} = \Phi^{\mathsf{L}}_{0}(\mathsf{A}_{0})$ and $\mathfrak{R}^{0} = \Phi^{\mathsf{R}}_{0}(\mathsf{A}_{0})$ are finite direct sums of simple ones and are hence quasi-rigid.
    Then, each block $\mathfrak{A}_{p,-q}$ is rigid over $\k$ by its construction.
    Finally, by the decompositions in \zcref{thm:decomposition_of_fL}, we conclude that $\mathfrak{L}^{n},\mathfrak{R}^{n}$ are quasi-rigid.
  \end{proof}
  \begin{remark}
    From this corollary, we see that, if $\k$ is semisimple, then $\k$-rational almost-canonically seminormed rings satisfy the assumption in \zcref{thm:AdjeqR-SIC}. 
    In fact, the Frobenius properties ($\mathsf{A}_{0}^{\k|\vee} \cong \mathsf{A}_{0}$ as $(\mathsf{A}_{0}|\k)$-bimodules and $\mathsf{A}_{0}^{\vee|\k} \cong \mathsf{A}_{0}$ as $(\k|\mathsf{A}_{0})$-bimodules) follow from the semisimplicity of $\mathsf{A}_{0}$.
    This gives another proof of ``$\k$-rationality $\implies$ \SIC''.
  \end{remark}

\subsection{For quasi-rigid modules}
  \zcref{thm:rational} has a variant for quasi-rigid modules, due to the following lemma:
  \begin{lemma}\label{lem:lifting-simple-ordinary}
    If $\mathfrak{L}^{0} \in \Mod*[Gr](U|{}_{\mathsf{A}_{0}}), \mathfrak{R}^{0} \in \Mod*[Gr]({}_{\mathsf{A}_{0}}|U)$, and if $\Mod*[Gr](U|{}_{R})$ is semisimple, then $\Phi^{\mathsf{L}}_{0}$ sends simple $(\mathsf{A}_{0}|{}_R)$-bimodules to quasi-rigid simple $(U|{}_{R})$-bimodules. 
  \end{lemma}
  \begin{proof}
    Suppose $M$ is a simple $(\mathsf{A}_{0}|{}_R)$-bimodule. 
    By \zcref{lem:lifting-simple}, it remains to show that $\Phi^{\mathsf{L}}_{0}(M)$ is quasi-rigid. By the assumption, each $\mathfrak{L}^{0}_{n}$ is a rigid right $\mathsf{A}_{0}$-module. 
    Since $M$ is simple, it is a right rigid $R$-module.
    Hence, by \zcref{lem:rigidmodule}, each $\Phi^{\mathsf{L}}_{0}(M)_{n} = \mathfrak{L}^{0}_{n}\otimes_{\mathsf{A}}M$ is also rigid over $R$.
  \end{proof}

  With the help of this lemma, the same argument of \zcref{thm:rational} shows the following.
  \begin{theorem}
    If $\mathfrak{L}^{0} \in \Mod*[Gr](U|{}_{\mathsf{A}_{0}}), \mathfrak{R}^{0} \in \Mod*[Gr]({}_{\mathsf{A}_{0}}|U)$, and if $\Mod*[Gr](U|{}_{\mathsf{A}_{0}})$ is semisimple, then the strong identity condition holds. 
  \end{theorem}

  As a consequence, we have the following:
  \begin{corollary}
    If $\mathfrak{L}^{0} \in \Mod*[Gr](U|{}_{\mathsf{A}_{0}}), \mathfrak{R}^{0} \in \Mod*[Gr]({}_{\mathsf{A}_{0}}|U)$, and if $\Mod*[Gr](U|{}_{\mathsf{A}_{0}})$ is semisimple, then $U$ is $\mathsf{A}_{0}$-rational.
  \end{corollary}
  \begin{proof}
    Since the strong identity condition holds. We can apply \zcref{thm:PhiOmega0} to conclude that the category of exhaustive modules, and thus also the category $\Mod[Gr](U|{}_{\mathsf{A}_{0}})$, is semisimple, and that all the simple modules are given by applying $\Phi^{\mathsf{L}}_{0}(-)$ to simple $(\mathsf{A}_{0}|\mathsf{A}_{0})$-bimodules, and are hence quasi-rigid by \zcref{lem:lifting-simple-ordinary}.
  \end{proof}

\section{Tensor products}\label{sec:tensor}
  Let $U^{\tt I},U^{\tt II}$ be two almost-canonically seminormed rings. Let $U=U^{\tt I}\otimes U^{\tt II}$ be their tensor product. Then, $U$ is bigraded:
  \[
    \nota{U_{\bullet,\circ}}:=U^{\tt I}_{\bullet}\otimes U^{\tt II}_{\circ}.
  \]
  Taking the total grading, we obtain a graded ring and hence are able to consider almost-canonical seminorms on it. 
  In this section, we will show that the strong identity condition (\SIC) on $U$ is equivalent to the strong identity conditions on $U^{\tt I}$ and $U^{\tt II}$.
  \begin{remark}
    In practice, especially in the context of vertex algebras, $U^{\tt I}$ and $U^{\tt II}$ are often taken to be complete with respect to their almost-canonical seminorms. Then, it makes more sense to consider the \emph{completed tensor product} (cf. \cite[\S A.7]{DGK23}) $U^{\tt I}\ctensor U^{\tt II}$ rather than the algebraic tensor product $U^{\tt I}\otimes U^{\tt II}$. 
    However, as discussed in \zcref{eg:completions}, this causes no difference in discussion of mode transition algebras and, in particular of the strong identity condition.
  \end{remark}

\subsection{The seminorm on the tensor product}
  \begin{proposition}\label{prop:seminorm-on-tensor}
    With the total grading, $U$ is an almost-canonically seminormed ring whose seminorm can be described as follows:
    \[
      \NL[n]U_{p} = \sum_{\substack{i+j=n\\a+b=p}}\NL[i]U^{\tt I}_{a}\otimes \NL[j]U^{\tt II}_{b}\txand 
      \NR[n]U_{p} = \sum_{\substack{i+j=n\\a+b=p}}\NR[i]U^{\tt I}_{a}\otimes \NR[j]U^{\tt II}_{b}.
    \]
  \end{proposition}
  \begin{proof}
    The formulas clearly define a seminorm on $U$ that is compatible with the multiplication. It remains to show that the canonical seminorm is dense in it. 
    This boils down to proving the formulas for the canonical seminorm $\cNL[\bullet]U$ and $\cNR[\bullet]U$. 

    Here, we only prove the formula for the left neighborhoods $\cNL[\bullet]U$; the proof for the right neighborhoods $\cNR[\bullet]U$ is similar.
    Indeed, we have
    \begin{align*}
      \cNL[n]U_{p} &= \sum_{t\le -n} U_{p-t}U_{t} = 
      \sum_{t\le -n}\Big(\sum_{i+j=p-t}U^{\tt I}_{i}\otimes U^{\tt II}_{j}\Big)\Big(\sum_{k+l=t}U^{\tt I}_{k}\otimes U^{\tt II}_{l}\Big) \\
      &= 
      \sum_{\substack{i+j+k+l=p\\k+l\le-n}}
      U^{\tt I}_{i}U^{\tt I}_{k}\otimes U^{\tt II}_{j}U^{\tt II}_{l}\\
      &= 
      \sum_{\substack{a+b=p\\k+l\le-n}}U^{\tt I}_{a-k}U^{\tt I}_{k}\otimes U^{\tt II}_{b-l}U^{\tt II}_{l}\\
      &= 
      \sum_{\substack{i+j=n\\a+b=p}}\Big(\sum_{k\le -i}U^{\tt I}_{a-k}U^{\tt I}_{k}\Big)\otimes\Big(\sum_{l\le -j}U^{\tt II}_{b-l}U^{\tt II}_{l}\Big)\\
      &= 
      \sum_{\substack{i+j=n\\a+b=p}}\cNL[i]U^{\tt I}_{a}\otimes \cNL[j]U^{\tt II}_{b},
    \end{align*}
    as desired.
  \end{proof}

  \begin{corollary}\label{coro:Omega-of-tensor}
    For any $(U^{\tt I}|{}_R)$-bimodule $W^{\tt I}$ and any $(U^{\tt II}|{}_R)$-bimodule $W^{\tt II}$, we have
    \[
      \Omega^{\mathsf{L}}_{n}(W^{\tt I} \otimes W^{\tt II}) = 
      \sum_{i+j=n} \Omega^{\mathsf{L}}_{i}(W^{\tt I}) \otimes \Omega^{\mathsf{L}}_{j}(W^{\tt II}).
    \]
    Similar for $\Omega^{\mathsf{R}}_{n}$'s.
  \end{corollary}

  With \zcref{prop:seminorm-on-tensor}, we are able to determine the mode transition algebra $\mathfrak{A}(U)$ of $U$. 
  \begin{corollary}\label{coro:mode-transition-algebra-tensor}
    The mode transition algebra $\mathfrak{A}(U)$ of $U$ is given by 
    \[
      \mathfrak{A}_{p,-q}(U) = 
      \sum_{\substack{a+b=p\\ c+d=q}}
      \mathfrak{A}_{a,-c}(U^{\tt I})\otimes\mathfrak{A}_{b,-d}(U^{\tt II}).
    \]
  \end{corollary}
  \begin{proof}
    Indeed, by \zcref{prop:seminorm-on-tensor}, we have
    \[
      \mathfrak{L}^{0}_{p}(U) = 
      \sum_{a+b=p}\frac{U^{\tt I}_{a}\otimes U^{\tt II}_{b}}{\cNL[1]U^{\tt I}_{a}\otimes U^{\tt II}_{b} + U^{\tt I}_{a}\otimes \cNL[1]U^{\tt II}_{b}} =
      \sum_{a+b=p} \mathfrak{L}^{0}_{a}(U^{\tt I}) \otimes \mathfrak{L}^{0}_{b}(U^{\tt II}).
    \]
    The computation for $\mathfrak{R}^{0}_{-q}(U)$ is similar. Then, the statement follows.
  \end{proof}

  In particular, for $p=q=0$, we have 
  \begin{corollary}\label{coro:A-tensor}
    $\mathsf{A}_{0}(U) = \mathsf{A}_{0}(U^{\tt I})\otimes\mathsf{A}_{0}(U^{\tt II})$.
  \end{corollary}
  Throughout this section, we will omit the subscript $0$ from the notations $\mathsf{A}_{0}$, $\Phi^{\mathsf{L}}_{0}$, $\Omega^{\mathsf{L}}_{0}$, etc., for simplicity.

\subsection{Tensor products of categories and functors}
  Recall that we have adjoint pairs: 
  \[
    \begin{tikzcd}
      \Mod[Ex](U^{\tt I}|{}_R) 
        \ar[d,shift left=1ex,"\Omega^{\mathsf{L}}_{U^{\tt I}}"]&
      \Mod[Ex](U|{}_R) 
        \ar[d,shift left=1ex,"\Omega^{\mathsf{L}}_{U}"]&
      \Mod[Ex](U^{\tt II}|{}_R) 
        \ar[d,shift left=1ex,"\Omega^{\mathsf{L}}_{U^{\tt II}}"]\\
      \Mod(\mathsf{A}(U^{\tt I})|{}_R) 
        \ar[u,shift left=1ex,"\Phi^{\mathsf{L}}_{U^{\tt I}}"]&
      \Mod(\mathsf{A}(U)|{}_R)
        \ar[u,shift left=1ex,"\Phi^{\mathsf{L}}_{U}"]&
      \Mod(\mathsf{A}(U^{\tt II})|{}_R)
        \ar[u,shift left=1ex,"\Phi^{\mathsf{L}}_{U^{\tt II}}"]
    \end{tikzcd}
  \]

  We will delcare their relation now. We need some facts about Deligne-Kelly tensor products of Grothendieck categories, which we put in appendix.
  \begin{theorem}\label{thm:tensor-of-cats}
    The Deligne-Kelly tensor product of categories of exhaustive $U^{\tt I}$-modules and $U^{\tt II}$-modules is equivalent to that of exhaustive $U^{\tt I}\otimes U^{\tt II}$-modules:
    \begin{equation}\label{eq:DK-ExMod}
      \Mod[Ex](U^{\tt I}\otimes U^{\tt II}|{}_R) \simeq \Mod[Ex](U^{\tt I}|{}_R)\boxtimes \Mod[Ex](U^{\tt II}|{}_R).
    \end{equation}
    Furthermore, through such an identification, we have 
    \[
      \Phi^{\mathsf{L}}_{U} = \Phi^{\mathsf{L}}_{U^{\tt I}}\boxtimes \Phi^{\mathsf{L}}_{U^{\tt II}}.
    \]
  \end{theorem}
  \begin{proof}
    By, \zcref{eg:tensor-of-Mods}, we have $\Mod(U^{\tt I}\otimes U^{\tt II}|{}_R) \simeq \Mod(U^{\tt I}|{}_R)\boxtimes \Mod(U^{\tt II}|{}_R)$.
    It suffices to show that the image of $\Mod[Ex](U^{\tt I}|{}_R)\otimes \Mod[Ex](U^{\tt II}|{}_R)$ in $\Mod(U^{\tt I}\otimes U^{\tt II}|{}_R)$ via this equivalence is precisely the subcategory $\Mod[Ex](U^{\tt I}\otimes U^{\tt II}|{}_R)$. 
    Since the tensors $M\otimes_{R}N$, where $M \in \Mod(U^{\tt I}|{}_R)$ and $N \in \Mod(U^{\tt II}|{}_R)$, generate $\Mod(U^{\tt I}\otimes U^{\tt II}|{}_R)$, it suffices to show that $M\otimes_{R}N$ is exhaustive if and only if both $M$ and $N$ are exhaustive.
    This follows from \zcref{prop:seminorm-on-tensor}.

    For the second assertion, it boils down to showing that, $\mathfrak{L}^{0}(U) = \mathfrak{L}^{0}(U^{\tt I}) \otimes \mathfrak{L}^{0}(U^{\tt II})$, which has been shown in the proof of \zcref{coro:mode-transition-algebra-tensor}.
  \end{proof}

  Combining this with \zcref{coro:Omega-of-tensor} declare the relations between the adjoint pairs in previous diagram. Then, we have

  \begin{corollary}\label{coro:SIC-tensor}
    Suppose $\k$ is a field and suppose $U^{\tt I}$ and $U^{\tt II}$ satisfy the quasi-rigid assumption in \zcref{thm:Adjeq-SIC} or the Frobenius assumption in \zcref{thm:AdjeqR-SIC}.
    Then, the strong identity condition (\SIC) holds for $U$ if and only if it holds for both $U^{\tt I}$ and $U^{\tt II}$.
  \end{corollary}
  \begin{proof}
    First, by \zcref{prop:seminorm-on-tensor}, $U$ satisfy the assumption in \zcref{thm:Adjeq-SIC} or \zcref{thm:AdjeqR-SIC}.
    Then, by \zcref{thm:tensor-of-cats}, applying \zcref{thm:tensor-of-equivalence} to the $\Phi^{\mathsf{L}}$-functors implies that $\Phi^{\mathsf{L}}_{U}$ is an equivalence if and only if both $\Phi^{\mathsf{L}}_{U^{\tt I}}$ and $\Phi^{\mathsf{L}}_{U^{\tt II}}$ are equivalences.
    Then, the assertion follows by \zcref{thm:PhiOmega0,thm:Adjeq-SIC,thm:AdjeqR-SIC}.
  \end{proof}

\subsection{Strong identity elements on tensor product}
  In the following, we give a more explicit construction of the strong identity elements.
  \begin{lemma}
    If $\one^{\tt I}_{n}$ and $\one^{\tt II}_{n}$ ($n\in\N$) are strong identity elements of $\mathfrak{A}(U^{\tt I})$ and $\mathfrak{A}(U^{\tt II})$ respectively, then the elements 
    \[
      \one_{n} = \sum_{p+q=n}\one^{\tt I}_{p}\otimes\one^{\tt II}_{q}\qquad
      (n\in\N)
    \]
    verify the strong identity condition for $U$.
  \end{lemma}
  \begin{proof}
    This is a direct verification using \zcref{coro:mode-transition-algebra-tensor}.
  \end{proof}

  The another direction is more complicated.

\begin{theorem}\label{thm:SIC-tensor}
  Suppose $\k$ is semisimple and suppose the graded $U$-modules $\mathfrak{L}^{0}$ and $\mathfrak{R}^{0}$ are quasi-rigid. 
  Then, the strong identity condition (\SIC) holds for $U$ if and only if it holds for both $U^{\tt I}$ and $U^{\tt II}$.
\end{theorem}
\begin{proof}
  We view each $\mathfrak{A}_{n,-n}$ as a matrix algebra $\mathfrak{M}^{n}_{\bullet,\circ}$, where
  \[
    \mathfrak{M}^{n}_{p,q} := \mathfrak{A}_{p,-q}(U^{\tt I})\otimes\mathfrak{A}_{n-p,q-n}(U^{\tt II}).
  \]
  For the strong identity element $\one_{n}$ of $\mathfrak{A}_{n,-n}(U)$, we can write it as
  \[
    \one_{n} = \sum_{p,q}\one^{n}_{p,q},\qquad\text{where }
    \one^{n}_{p,q} \in \mathfrak{M}^{n}_{p,q}.
  \]
  Note that, by \zcref{rem:multiplication-of-fA}, we have  
  \[
    \mathfrak{M}^{n}_{a,b} \star \mathfrak{M}^{n}_{c,d} = 
    \left(\mathfrak{A}_{a,-b} \star \mathfrak{A}_{c,-d}\right) \otimes \left(\mathfrak{A}_{n-a,b-n} \star \mathfrak{A}_{n-c,d-n}\right) \subset
    \begin{dcases*}
      \mathfrak{M}^{n}_{a,d}, & if $b=c$;\\
      0, & otherwise.
    \end{dcases*}
  \]
  Therefore, for any $\mathfrak{x} \in \mathfrak{M}^{n}_{p,q}$, we have
  \[
    \mathfrak{x} = \mathfrak{x} \star \one_{n} = \sum_{k}\mathfrak{x}\star \one^{n}_{q,k} = \mathfrak{x}\star \one^{n}_{q,q},
  \]
  where the last two equality follows from the above multiplication rule and by comparing the indexes of the components. 
  Likewise, we have $\mathfrak{x} = \one^{n}_{p,p} \star \mathfrak{x}$. 
  In particular, we see that $\one^{n}_{p,q} = 0$ if $p \neq q$ and that each $\mathfrak{M}^{n}_{p,p}$ is a unital subalgebra of $\mathfrak{A}_{n,-n}(U)$ with unity $\one^{n}_{p,p}$. 

  To move on, we need a lemma: 
  \begin{lemma}
    Suppose $A$ and $B$ are two $\k$-algebras and are rigid as $\k$-modules. 
    If $A\otimes B$ is a nonzero unital algebra, then both $A$ and $B$ are unital algebras.
  \end{lemma}
  \begin{proof}
    First, both $A$ and $B$ are nonzero since $A\otimes B\neq 0$. 
    Write the identity element of $A\otimes B$ as a finite sum 
    \[
      1_{A\otimes B} = \sum_{i} a_{i}\otimes b_{i},\qquad a_{i}\in A,b_{i}\in B.
    \]
    On the other hand, since $B$ is a rigid $\k$-module, there are finite collections of elements $y_j \in B$ and $f_j \in B^{\k|\vee}$ such that \[ \sum_j f_j(y_j) = 1_\k. \]
    Then, we can define 
    \[
      1_A = \sum_{i,j} a_i \cdot f_j(b_iy_j).
    \]
    For any $x \in A$, we have
    \begin{align*}
      1_A \cdot x
      &= \sum_{j} (a_ix) \cdot f_j(b_iy_j) = \sum_{i,j} (\id_A\otimes f_j)(a_ix\otimes b_iy_j) \\
      &= \sum_{j} (\id_A\otimes f_j)\left((\sum_i a_i\otimes b_i) \cdot (x\otimes y_j)\right) \\
      &= \sum_{j} (\id_A\otimes f_j)(x\otimes y_j) = x\cdot \sum_j f_j(y_j) = x.
    \end{align*}
    This shows that $1_A$ is a left unity of $A$. Similarly, we can construct a right unity of $A$, which must be equal to $1_A$. Hence, $A$ is unital. 
    Similarly, $B$ is also unital.
  \end{proof}

  Note that, from our assumption, each block $\mathfrak{A}_{p,-q}$ of $U^{\tt I}$ or of $U^{\tt II}$ is finite over $\k$. 
  By this lemma, since $\one^{n}_{p,p}$ is the unity of 
  \[
    \mathfrak{M}^{n}_{p,p} = \mathfrak{A}_{p,-p}(U^{\tt I})\otimes\mathfrak{A}_{n-p,p-n}(U^{\tt II}),
  \]
  both $\mathfrak{A}_{p,-p}(U^{\tt I})$ and $\mathfrak{A}_{n-p,p-n}(U^{\tt II})$ are unital algebras. 
  Let $\one^{\tt I}_{p}$ and $\one^{\tt II}_{n-p}$ be their respective unities. Note that, since $\mathfrak{A}_{0,0}(-) = \mathsf{A}(-)$ is already unital with unity $1_{\mathsf{A}(-)}$, we must have: 
  \[
   \one^{\tt I}_{0} = 1_{\mathsf{A}(U^{\tt I})}
   \txand
   \one^{\tt II}_{0} = 1_{\mathsf{A}(U^{\tt II})}.
  \]
  Let $n$ and $p$ varies over all possibilities, we obtain two collections of elements $\one^{\tt I}_{m}$ and $\one^{\tt II}_{m}$ ($m\in\N$) in the mode transition algebras of $U^{\tt I}$ and $U^{\tt II}$ respectively and we have 
  \[
    \one_{n} = \sum_{p+q=n} \one^{\tt I}_{p} \otimes \one^{\tt II}_{q}.
  \]
  It remains to show that they verify the strong identity condition.

  For any $\mathfrak{a}\in \mathfrak{A}_{p,-q}(U^{\tt I})$, we have 
  \begin{align*}
    \mathfrak{a}\otimes 1_{\mathsf{A}(U^{\tt II})} 
    &= 
    (\mathfrak{a}\otimes 1_{\mathsf{A}(U^{\tt II})}) \star \one_{q} \\
    &= 
    (\mathfrak{a}\otimes 1_{\mathsf{A}(U^{\tt II})}) \star \sum_{k} (\one^{\tt I}_{k}\otimes \one^{\tt II}_{q-k}) \\
    &=
    (\mathfrak{a} \star \one^{\tt I}_{q}) \otimes 
    (1_{\mathsf{A}(U^{\tt II})} \star \one^{\tt II}_{0})
    = (\mathfrak{a} \star \one^{\tt I}_{q}) \otimes  1_{\mathsf{A}(U^{\tt II})}.
  \end{align*}
  Since $\k$ is semisimple, and $1_{\mathsf{A}(U^{\tt II})} \neq 0$, the map $\mathfrak{A}_{p,-q}(U^{\tt I}) \to \mathfrak{A}_{p,-q}(U^{\tt I})\otimes \mathsf{A}(U^{\tt II})$ provided by $\Box \mapsto \Box\otimes 1_{\mathsf{A}(U^{\tt II})}$ is injective. Hence, we have $\mathfrak{a} = \mathfrak{a} \star \one^{\tt I}_{q}$.
  Similarly, we have $\mathfrak{a} = \one^{\tt I}_{p} \star \mathfrak{a}$. 
  This shows that $\one^{\tt I}_{n}$ ($n\in\N$) are strong identity elements of $\mathfrak{A}(U^{\tt I})$. 
  The same argument shows that $\one^{\tt II}_{n}$ ($n\in\N$) are strong identity elements of $\mathfrak{A}(U^{\tt II})$.
\end{proof}

\section{The end formula for the mode transition algebra}\label{sec:End}
  In this section, we give a description of the mode transition algebra via ends.

\subsection{A Peter-Weyl expression for (co)regular bimodules}
  In \cite[\S 2.4]{FSS}, Fuchs, Schaumann, and Schweigert gave a Peter-Weyl type expression for the (co)regular bimodule of a finite-dimensional algebra via (co)ends. For the purpose of our application, we need to generalize it to a slightly more general setting.
  \begin{proposition}\label{prop:PW}
    Let $A,B,R$ be rings such that $A \in \Mod*(A|{}_R)$, and let $G$ be an $R$-linear functor from $\Mod*(A|{}_R)$ to $\Mod(B|{}_R)$. 
    The end of the functor 
    \begin{align*}
      \widetilde{G} \colon 
      \Mod*(A|{}_R) \times \opp{\Mod*(A|{}_R)} &\longrightarrow \Mod(B|A) \\
      (M,N) & \longmapsto G(M) \otimes_{R} N^{\vee|R}
    \end{align*}
    is given as a $(B|A)$-bimodule by
    \[
      \int_{M \in \Mod*(A|{}_R)} G(M) \otimes_{R} M^{\vee|R} \cong G(A)
    \]
    with the natural $(B|A)$-bimodule structure on $G(A)$ (where the right $A$-action is induced by the right multiplication on $A$). 
    Its dinatural family is defined by 
    \begin{align*}
      i_M^A \colon G(A) &\longrightarrow G(M) \otimes_{R} M^{\vee|R} \\
      G(a) &\longmapsto \wick{G(a\cdot \c {-}) \otimes \c {-}}.
    \end{align*}
    Here, $G(a)$ means a general element in $G(A)$; $G(a\cdot m)$ means the image of $G(a)$ under the map $G(\cdot m)$, where $\cdot m \colon A\to M$ is the map $a\mapsto a \cdot m$; and $\wick{ \c{-} \cdots \c{-}}$ is the coevaluation for $M$.
    
    Dually, the coend of $\widetilde{G}$ is given by
    \[
      \int^{M \in \Mod*(A|{}_R)} G(M) \otimes_{R} M^{\vee|R} \cong G(A^{\vee|R}) 
    \]
    with the aforementioned natural $(B|A)$-bimodule structure on $G(A)$ and the natural $(A|A)$-bimodule structure on $A^{\vee|R}$ (where the left $A$-action is induced by the left multiplication on $A$). 
    Its dinatural family is defined by 
    \begin{align*}
      j_M^A \colon G(M) \otimes_{R} M^{\vee|R} &\longrightarrow G(A^{\vee|R})\\ 
      G(m) \otimes m' &\longmapsto 
      G(\wick{  \c {-} \cdot \langle \c {-} \cdot m \vert m' \rangle}),
    \end{align*}
    where $G(m)$ means a general element in $G(M)$; and $\wick{ \c{-} \cdots \c{-} }$ is the coevaluation for $A^{\vee|R}$.
  \end{proposition}
  \begin{proof}
    The first statement is due to the fact that the forgetting functor $u \colon \Mod*(A|_{R})\to\Mod*(|R)$ is corepresented by $A$. Namely, $u=\yo^A=\Hom_{A|{}_R}(A,-)$, the Yoneda embedding of $A$. 
    Indeed, we have 
    \begin{align*}
      \int_{M \in \Mod*(A|{}_R)} G(M) \otimes_{R} M^{\vee|R} 
      &=
      \int_{M \in \Mod*(A|{}_R)} \Hom_{R}(M, G(M)) \\
      &=
      \fun{Nat}(\yo^A,G) = G(A)
    \end{align*}
    by the \emph{Yoneda lemma} and the fact that the forgetting functor $\Mod(B|A) \to \Mod(\k)$ creates limits.
    The second follows from that the dualizing functor $(-)^{\vee|_{R}} \colon \Mod*(A|_{R})\to\Mod*(R|)$ is represented by $A^{\vee|R}$. Namely, we have $u=\yo_{A^{\vee|R}}=\Hom_{A|{}_R}(-,A^{\vee|R})$.
    Indeed, we have
    \[
      \int^{M \in \Mod*(A|{}_R)} G(M) \otimes_{R} M^{\vee|R} 
      =
      \int^{M \in \Mod*(A|{}_R)} G(M) \otimes_{R} \yo_{A^{\vee|R}}(M) 
      =
      G(A^{\vee|R})
    \]
    by the \emph{co-Yoneda lemma} and the fact that the forgetting functor $\Mod(B|A) \to \Mod(\k)$ creates colimits.

    We also refer to \cite[Proposition 2.8]{FSS} for an more explicit proof, which also works in our setting with minor modifications.
  \end{proof}

\subsection{The end formula}
  We have the following expressions of the mode transition algebra $\mathfrak{A}$ via end-construction:
  \begin{theorem}\label{thm:EndMode}
    Suppose $\mathsf{A}_{0} \in \Mod*(\mathsf{A}_{0}|{}_R)$. Then, we have
    \[
      \mathfrak{A} = \int_{M \in \Mod*(\mathsf{A}_{0}|{}_{R})} \Phi^{\mathsf{L}}_{0}(M) \otimes_{R} \Phi^{\mathsf{R}}_{0}(M^{\vee|R}).
    \]
  \end{theorem}
  Note that the assumption is satisfied in particular when $R=\mathsf{A}_{0}$. In the case $R=\k$, the assumption means the $\mathsf{A}_{0}$ is rigid as a $\k$-module.
  \begin{proof}
    Since $\Phi^{\mathsf{R}}_{0}$ is a right adjoint, it preserves limits. Hence, we have
    \[
      \int_{M \in \Mod*(\mathsf{A}_{0}|{}_{R})} \Phi^{\mathsf{L}}_{0}(M) \otimes_{R} \Phi^{\mathsf{R}}_{0}(M^{\vee|R}) 
      =
      \Phi^{\mathsf{R}}_{0}(\int_{M \in \Mod*(\mathsf{A}_{0}|{}_{R})} \Phi^{\mathsf{L}}_{0}(M) \otimes_{R} M^{\vee|R}).
    \]
    More precisely, the functor $\Phi^{\mathsf{R}}_{0}$ in the above should be the tensor product of the identity endofunctor of $\Mod[Ex](U|{}_{R})$ with the functor $\Phi^{\mathsf{R}}_{0}\colon\Mod*({}_{R}|\mathsf{A}_{0})\to\Mod[Ex]({}_{R}|U)$.

    On the other hand, by \zcref{prop:PW}, we have 
    \[
      \int_{M \in \Mod*(\mathsf{A}_{0}|{}_{R})} \Phi^{\mathsf{L}}_{0}(M) \otimes_{R} M^{\vee|R} = \Phi^{\mathsf{L}}_{0}(\mathsf{A}_{0}).
    \]
    Hence, we conclude that:
    \[
      \int_{M \in \Mod*(\mathsf{A}_{0}|{}_{R})} \Phi^{\mathsf{L}}_{0}(M) \otimes_{R} \Phi^{\mathsf{R}}_{0}(M^{\vee|R}) 
      =
      \Phi^{\mathsf{R}}_{0}(\Phi^{\mathsf{L}}_{0}(\mathsf{A}_{0})) = \mathfrak{A}.
    \]
    This finishes the proof.
  \end{proof}

  \begin{corollary}\label{coro:EndMode}
    If $U$ satisfies the strong identity condition, then the mode transition algebra $\mathfrak{A}$ can be expressed as 
    \[
      \mathfrak{A} = \int_{M \in \Mod*(\mathsf{A}_{0}|{}_{R})} \Phi^{\mathsf{L}}_{0}(M) \otimes_{R} \Phi^{\mathsf{L}}_{0}(M)^{\dagger|R},
    \]
    whenever $\mathsf{A}_{0} \in \Mod*(\mathsf{A}_{0}|{}_R)$.
    Furthermore, suppose either 
      \begin{enumerate}
        \item $U$ is \emph{quasi-rigid}: $\mathfrak{L}^{n} \in \Mod*[Gr](U|{}_{\mathsf{A}_{n}})$ and $\mathfrak{R}^{n} \in \Mod*[Gr]({}_{\mathsf{A}_{n}}|U)$; or 
        \item $R$ is semisimple and $U$ is \emph{weakly quasi-finite}: components of $\mathfrak{L}^{n},\mathfrak{R}^{n}$ are finite over $\mathsf{A}_{n}$.
      \end{enumerate}
    Then, $\mathfrak{A}$ can be expressed as 
    \[
      \int_{W \in \Mod*[Ex](U|{}_{R})} W \otimes_{R} W^{\dagger|R}.
    \]
    Here, the grading on $W$ is the $\gr\Omega$-grading.
  \end{corollary}
  \begin{proof}
    Follows from the previous theorem and \zcref{lem:induced-dual,coro:PhiOmega0-ord}.
  \end{proof}

\part{Applications and examples}\label{part4}
This part specializes the general theory to concrete algebraic and geometric settings, including explicit computations for Weyl algebras, applications to VOA theory, and counterexamples, thereby testing the scope and limitations of the strong identity condition.

\section{Weyl algebras -- a non-VOA example}\label{sec:Weyl}
  To illustrate the theory independent of VOA-specific technicalities, we study the \emph{Weyl algebra}\footnote{We omit the subscript to avoid confusion with the degree index.}
  \[
    \nota{\mathcal{D}} = \nota{\mathcal{D}_{\mathbb{A}^{d}}} := \k\brk[a]{\vect{x},\vect{\partial}}/(\partial_{i}x_j-x_j\partial_{i}-\delta_{i,j},x_ix_j-x_jx_i,\partial_{i}\partial_{j}-\partial_{j}\partial_{i}).
  \]
  It is the ring of differential operators on the affine space $\mathbb{A}^{d}_{\k}$. Modules over it are prototypes of the main objects of study in the area of algebraic analysis. We refer to \cite{DmodLectures,DModBook1995,HTTDmodules} for further introduction to such a topic. Here, we only consider it as an almost-canonically seminormed ring.

  There are various useful filtrations on $\mathcal{D}$, such as the \emph{order filtration} (filtered by the orders of differential operators) and the \emph{Berstein filtration} (filtered by total degree of differential operators and polynomials). 
  Here, we consider the following \emph{Fuchsian grading}: 
  \[
    \deg x_i=1\txand
    \deg\partial_{x_i}=-1.
  \]
  The canonical seminorm with respect to it can be characterized as follows: for each $n\in\N$,
  \begin{align*}
    \cNL[n]\mathcal{D} &= \Set{f\in\mathcal{D}\given\text{the total exponents of $\partial_{i}$'s in each monomial of $f$ is at least $n$}},\\
    \cNR[n]\mathcal{D} &= \Set{f\in\mathcal{D}\given\text{the total exponents of $x_i$'s in each monomial of $f$ is at least $n$}}.
  \end{align*}
  We see that the left and right neighborhoods are precisely the filtration by \emph{orders of differential operators} and the filtration by \emph{degrees of polynomials} respectively.

  \begin{proposition}\label{prop:modeAofWeyl}
    The components of the mode transition algebra $\mathfrak{A}$ of the Weyl algebra $\mathcal{D}$ is given by 
    \[
      \mathfrak{A}_{p,-q} \cong \bigoplus_{\abs{\alpha}=p,\abs{\beta}=q}\k\vect{x}^{\alpha}\otimes\vect{\partial}^{\beta}.
    \]
    Here, $\alpha = (\alpha_i),\beta = (\beta_i)$ are multi-indices, and $\abs{\alpha} := \sum_i\alpha_i,\abs{\beta} := \sum_i\beta_i$. 
  \end{proposition}
  \begin{proof}
    From the above discussion, we have 
    \begin{itemize}
      \item $\mathfrak{L}^{0} \cong \k[\vect{x}]$, the \emph{coordinate ring} $\mathcal{O}$ of the affine space $\mathbb{A}^{d}_{\k}$; and 
      \item $\mathfrak{R}^{0} \cong \k[\vect{\partial}]$, the \emph{Fourier transformation} of $\mathcal{O}$.
    \end{itemize}
    Then, their gradings clearly coincide with the degree grading of polynomials in variables $\vect{x}$ and $\vect{\partial}$ respectively. 
    The structure of the mode transition algebra then follows.
  \end{proof}

  In order to discuss the strong identity condition, we need to understand the multiplication on the mode transition algebra $\mathfrak{A}$. 

\subsection{Rank one Weyl algebra}
  We first consider the rank one situation as it will be used in the consideration of arbitrary rank. By \zcref{lem:def:ostar}, we have 
  \[
    [{\partial}^{a}]^{\mathsf{R}}_{0}\ostar[{x}^{b}]^{\mathsf{L}}_{0} = 
    \begin{dcases*}
      [{\partial}^{a}{x}^{b}]_{0} & if $a=b$, \\
      0 & otherwise.
    \end{dcases*}
  \]
  To determine the multiplication on $\mathfrak{A}$, it suffices to compute $[{\partial}^{a}{x}^{a}]_{0}$.
  \begin{lemma}\label{lem:partialx=1}
    For any $a\in\N$, we have $[{\partial}^{a}{x}^{a}]_{0}=a!$.
  \end{lemma}
  \begin{proof}
    By the relation $\Lie*{\partial}{x}=1$, we have
    \[
      {\partial}^{a}{x}^{a} = {\partial}^{a-1}{x}^{a}\partial + {\partial}^{a-1}\Lie*{\partial}{{x}^{a}} \equiv {\partial}^{a-1}\Lie*{\partial}{{x}^{a}} = a{\partial}^{a-1}x^{a-1} \mod \cNLR[1]\mathcal{D}.
    \]
    Repeating this, we obtain the conclusion.
  \end{proof}
  \begin{remark}\label{rem:partialx=0}
    Note that, by the proof of \zcref{lem:def:ostar}, we must have ${\partial}^{a}x^{b}\in(\cNL[1]\mathcal{D}+\cNR[1]\mathcal{D})$ whenever $a\neq b$. This avoids verifying a lot of congruences.
  \end{remark}

  With the help of this lemma, we have:
  \begin{proposition}\label{prop:SICegWeyl}
    If the rational number field $\mathbb{Q}$ is contained in $\k$. Then, the Weyl algebra $\mathcal{D}_{\mathbb{A}^{1}}$ verifies the strong identity condition (\SIC).
  \end{proposition}
  \begin{proof}
    Consider the following elements:
    \[
      \one_{n} = \frac{1}{n!}[x^{n}]^{\mathsf{L}}_{0}\otimes[\partial^{n}]^{\mathsf{R}}_{0}\in\mathfrak{A}_{n}.
    \]
    Then, for any individual $[x^{a}]^{\mathsf{L}}_{0}\otimes[\partial^{b}]^{\mathsf{R}}_{0}$, we have 
    \[
      ([x^{a}]^{\mathsf{L}}_{0}\otimes[\partial^{b}]^{\mathsf{R}}_{0}) \star \one_{b} 
      = [x^{a}]^{\mathsf{L}}_{0}\otimes\frac{1}{b!}[\partial^{b}x^{b}\partial^{b}]^{\mathsf{R}}_{0} 
      = [x^{a}]^{\mathsf{L}}_{0}\otimes[\partial^{b}]^{\mathsf{R}}_{0},
    \]
    and similarly $\one_{a}\star ([x^{a}]^{\mathsf{L}}_{0}\otimes[\partial^{b}]^{\mathsf{R}}_{0}) = [x^{a}]^{\mathsf{L}}_{0}\otimes[\partial^{b}]^{\mathsf{R}}_{0}$.
  \end{proof}
  \begin{remark}\label{rem:SICnonegWeyl}
    On the other hand, by \zcref{lem:partialx=1}, if there is any integer $n\in\N$ vanishing in $\k$, then the multiplication $\star$ degenerates on higher levels and thus there is no strong identity elements. 
  \end{remark}

  Applying \zcref{prop:decomposition_of_fA,thm:decomposition_of_fL}, we obtain the following:
  \begin{corollary}
    We have the following formulas for the Weyl algebra $\mathcal{D}_{\mathbb{A}^{1}}$:
    \begin{align*}
      \mathsf{A}_{n} &= \prod_{k=0}^{n}\mu(\mathfrak{A}_{k}) = \fun{span}_{\k}\Set*{[1]_{n},[x\partial]_{n},\cdots,[x^{n}\partial^{n}]_{n}},\\
      \mathfrak{L}^{n}_{p} &= \bigoplus_{k=0}^{n} \mu(\mathfrak{A}_{p+k,-k}) = \fun{span}_{\k}\Set*{[x^{p}]_{n},[x^{p+1}\partial]_{n},\cdots,[x^{p+n}\partial^{n}]_{n}},\\
      \mathfrak{R}^{n}_{-q} &= \bigoplus_{k=0}^{n} \mu(\mathfrak{A}_{k,-k-q}) = \fun{span}_{\k}\Set*{[\partial^{q}]_{n},[x\partial^{1+q}]_{n},\cdots,[x^{n}\partial^{n+q}]_{n}}.
    \end{align*}
  \end{corollary}
  These fine structures would be helpful in the calculation of $\mathcal{D}$-modules. For instance, one can see applications in later of this section.

\subsection{Arbitrary rank}
  Now, we move to the arbitrary rank situation.
  The multiplication on $\mathfrak{A}$ is determined by the following lemma:
  \begin{lemma}\label{lem:partialx-multi}
    For any multi-indices $\alpha=(\alpha_i)$ and $\beta=(\beta_i)$, we have
    \[
      [\vect{\partial}^{\alpha}]^{\mathsf{R}}_{0}\ostar[\vect{x}^{\beta}]^{\mathsf{L}}_{0} = 
      \begin{dcases*}
        \alpha! & if $\alpha=\beta$, \\
        0 & otherwise,
      \end{dcases*}
    \]
    where $\alpha!$ is defined as the product of all $\alpha_i!$'s.
  \end{lemma}
  \begin{proof}
    First, by the commutativity between different variables, we may write $\vect{\partial}^{\alpha}\vect{x}^{\beta}$ into the form 
    \[
      \vect{\partial}^{\alpha}\vect{x}^{\beta} = \cdots \partial_i^{\alpha_i} \cdots \partial_j^{\alpha_j} \cdots x_i^{\beta_i} \cdots x_j^{\beta_j} \cdots = 
      \cdots \partial_i^{\alpha_i} x_i^{\beta_i} \cdots \partial_j^{\alpha_j} x_j^{\beta_j} \cdots.
    \]
    
    If $\alpha\neq\beta$, then at least we have $\alpha_{i_0}\neq\beta_{i_0}$ for some $i_0$. Then, since $\partial_{i_0}^{\alpha_{i_0}}x_{i_0}^{\beta_{i_0}} \in \cNL[1]\mathcal{D}+\cNR[1]\mathcal{D}$ (cf. \zcref{rem:partialx=0}), the above must vanish modulo $\cNL[1]\mathcal{D}+\cNR[1]\mathcal{D}$. 
    On the other hand, when $\alpha = \beta$, we can repeatedly apply \zcref{lem:partialx=1} to each variable to obtain the statement.
  \end{proof}

  \begin{theorem}\label{thm:SICegWeyl}
    If the rational number field $\mathbb{Q}$ is contained in $\k$. 
    Then, the Weyl algebra $\mathcal{D}$ verifies the strong identity condition (\SIC).
  \end{theorem}
  \begin{proof}
    Consider the following elements:
    \[
      \one_{n} = 
      \sum_{\abs{\alpha}=n}\frac{1}{\alpha!}[\vect{x}^{\alpha}]^{\mathsf{L}}_{0}\otimes[\vect{\partial}^{\alpha}]^{\mathsf{R}}_{0}\in\mathfrak{A}_{n},
    \]
    where $\alpha!$ is defined as the product of all $\alpha_i!$'s.
    Then, for any individual $[\vect{x}^{\beta}]^{\mathsf{L}}_{0}\otimes[\vect{\partial}^{\gamma}]^{\mathsf{R}}_{0}$ with $\abs{\beta}=n$ and $\abs{\gamma}=m$, applying \zcref{lem:partialx-multi}, we have
    \begin{align*}
      \one_{n} \star ([\vect{x}^{\beta}]^{\mathsf{L}}_{0}\otimes[\vect{\partial}^{\gamma}]^{\mathsf{R}}_{0}) 
      &= \sum_{\abs{\alpha}=n}\frac{1}{\alpha!}([\vect{x}^{\alpha}]^{\mathsf{L}}_{0}\otimes[\vect{\partial}^{\alpha}]^{\mathsf{R}}_{0}) \star ([\vect{x}^{\beta}]^{\mathsf{L}}_{0}\otimes[\vect{\partial}^{\gamma}]^{\mathsf{R}}_{0}) \\
      &= \sum_{\abs{\alpha}=n}\frac{1}{\alpha!}[\vect{x}^{\alpha}\vect{\partial}^{\alpha}\vect{x}^{\beta}]^{\mathsf{L}}_{0}\otimes[\vect{\partial}^{\gamma}]^{\mathsf{R}}_{0} \\
      &= \frac{1}{\beta!}[\vect{x}^{\beta}\vect{\partial}^{\beta}\vect{x}^{\beta}]^{\mathsf{L}}_{0}\otimes[\vect{\partial}^{\gamma}]^{\mathsf{R}}_{0} 
      = [\vect{x}^{\beta}]^{\mathsf{L}}_{0}\otimes[\vect{\partial}^{\gamma}]^{\mathsf{R}}_{0}
    \end{align*}
    similarly, $([\vect{x}^{\beta}]^{\mathsf{L}}_{0}\otimes[\vect{\partial}^{\gamma}]^{\mathsf{R}}_{0})\star \one_{m} = [\vect{x}^{\beta}]^{\mathsf{L}}_{0}\otimes[\vect{\partial}^{\gamma}]^{\mathsf{R}}_{0}$.
  \end{proof}
  \begin{remark}
    The strong identity element $\one_{1}$ is clearly related to the \emph{Euler vector field} $\sum_{i}x_{i}\partial_{i}$. 
    For an application of this notion, we refer to \cite[Theorem 1.6.5]{HTTDmodules} for the $\mathcal{D}$-affineness of projective spaces. 
    A preparation for the proof is the following fact: differential operators on $\mathbb{P}^{d}_{\k}$ can be obtained from that on $\mathcal{D}_{d+1}$. Indeed, $\mathcal{D}_{\mathbb{P}^{d}_{\k}}$ is the quotient of the degree-zero (under the Fuchsian grading) part of $\mathcal{D}_{d+1}$ by the ideal generated by the Euler vector field.
  \end{remark}

  By \zcref{prop:decomposition_of_fA,thm:decomposition_of_fL}, we obtain the following:
  \begin{corollary}\label{coro:ALRofWeyl}
    We have the following formulas for the Weyl algebra $\mathcal{D}$:
    \begin{align*}
      \mathsf{A}_{n} &= \prod_{k=0}^{n}\mu(\mathfrak{A}_{k}) = \bigoplus_{\abs{\alpha}=\abs{\beta}\le n}\k[\vect{x}^{\alpha}\vect{\partial}^{\beta}]_{n},\\
      \mathfrak{L}^{n}_{p} &= \bigoplus_{k=0}^{n} \mu(\mathfrak{A}_{p+k,-k}) = \bigoplus_{\abs{\alpha}-\abs{\beta}=p, \abs{\beta}\le n}\k[\vect{x}^{\alpha}\vect{\partial}^{\beta}]_{n},\\
      \mathfrak{R}^{n}_{-q} &= \bigoplus_{k=0}^{n} \mu(\mathfrak{A}_{k,-k-q}) = \bigoplus_{\abs{\alpha}\le n, \abs{\beta}-\abs{\alpha}=q}\k[\vect{x}^{\alpha}\vect{\partial}^{\beta}]_{n}.
    \end{align*}
    In particular, all of them are free $\k$-modules of finite rank. 
  \end{corollary}
  These fine structures would be helpful in the calculation of $\mathcal{D}$-modules. For instance, one can see applications in later of this section.

\subsection{Consequences of the Morita-type equivalence}
  We keep assuming that $\k$ is a $\mathbb{Q}$-algebra. 
  By \zcref{thm:PhiOmega0,coro:PhiOmega0-ord}, we have the following Morita-type equivalences:
  \[
    \Phi^{\mathsf{L}}_{0}\colon
    \Mod(\mathsf{A}_{0})\rightleftharpoons
    \Mod[Ex](\mathcal{D})\colon
    \Omega^{\mathsf{L}}_{0}
    \txand
    \Phi^{\mathsf{L}}_{0}\colon
    \Mod*(\mathsf{A})\rightleftharpoons
    \Mod*[Ex](\mathcal{D})\colon
    \Omega^{\mathsf{L}}_{0}
  \]
  To explain what do they mean in the context of $\mathcal{D}$-modules, we first recall that any $\mathcal{D}$-module $\mathcal{M}$ can be viewed as an $\mathcal{O}$-module equipped with a \emph{de Rham complex}:
  \[
    \mathcal{M}\overset{\nabla}{\longrightarrow}\mathcal{M}\otimes_{\mathcal{O}}\Omega^{1}\overset{\nabla}{\longrightarrow}\mathcal{M}\otimes_{\mathcal{O}}\Omega^{2}\overset{\nabla}{\longrightarrow}\cdots
  \]
  Here, the connection $\nabla$ is determined by the action of vector fields $\partial_i$ on $M$. 
  Hence, we see that: 
  \[
    \Omega^{\mathsf{L}}_{0}(\mathcal{M}) = \mathsf{H}^0_{\rm dR}(\mathcal{M}).
  \]
  Namely, $\Omega^{\mathsf{L}}_{0}(\mathcal{M})$ is precisely the space of \emph{horizontal sections} of $\mathcal{M}$. 

  On the other hand, for any $\k$-module $M$, the induced module $\Phi^{\mathsf{L}}_{0}(M)$ is precisely the free $\mathcal{O}$-module $\mathcal{O}\otimes_{\k}M$ equipped with the canonical de Rham complex: 
  \[ 
    M \otimes_{\k} \left(\mathcal{O}\overset{\d{}}{\longrightarrow}\Omega^{1}\overset{\d{}}{\longrightarrow}\Omega^{2}\overset{\d{}}{\longrightarrow}\cdots\right)
  \]

  A priori, being exhaustive only means the underlying space is discrete and the action of $\mathcal{D}$ is continuous (see \zcref{prop:ExModDiscCont}).
  However, the above Morita-type equivalences provid us the following characterizations, which is non-obvious at first sight:
  \begin{corollary}\label{coro:ExModWeyl}
    A $\mathcal{D}$-module $\mathcal{M}$ is exhaustive if and only if there is a $\k$-module $M$ such that 
    \[
      \mathcal{M} \cong \Phi^{\mathsf{L}}_{0}(M).
    \]
    In other words, exhaustive $\mathcal{D}$-modules are the same as (possibly infinite rank) vector bundles over $\mathbb{A}^{d}_{\k}$ with trivial connections and quasi-rigid ones are those of finite rank.
  \end{corollary}

\section{Application to vertex operator algebras}\label{sec:VOA}
  In this section, we apply the theory developed in the previous sections to the setting of vertex operator algebras.
  We will consider the enveloping algebras of vertex operator algebras and their twisted variants. 
  For further background, see \cite{FZ92,FBZ04,Fre07,NT,MNT10,DGK23}. 
  For simplicity, we assume $\k$ is a field of characteristic $0$. But using the techniques of \cite{VertexRings}, one can generalize the discussion to positive characteristic and or more generally, to commutative rings.

\subsection{Vertex operator algebra and their automorphisms and modules}
  \begin{definition}\label{def:VOA}
    A \concept{vertex operator algebras} (\emph{VOA} for short) is a quadruple $(\mathbb{V}, Y(\cdot,z), \vac, \cfv)$, throughout simply denoted by $\mathbb{V}$, where 
    \begin{itemize}
      \item $\nota{\mathbb{V}} = \bigoplus_{k\in\Z} \mathbb{V}_k$ is a graded vector space;
      \item $\nota{Y( \nocolor{\placeholder} , z )}\colon \mathbb{V} \to (\End{\mathbb{V}})\dbrack{z^{\pm1}}$ is a linear map assigning to each $a \in \mathbb{V}$ the \concept{vertex operator} 
      \[
        Y(a,z) = \sum_{n\in\Z} \nota{\vo{a}{n}} z^{-n-1},
      \]
      which is a formal power series with coefficients in $\End{\mathbb{V}}$;
      \item $\nota{\vac} \in V_0$ and $\nota{\cfv} \in V_2$ are two distinguished vectors, called the \concept{vacuum vector} and the \concept{conformal vector} respectively.
    \end{itemize}
    These data must satisfy the following axioms:
    \begin{enumerate}[label={\textbf{V}\arabic*}]
      \item \concept{\small(Grading-restriction)}
          For each $n \in \Z$: $\dim_{\k} \mathbb{V}_n < \infty$, and $\mathbb{V}_n = 0$ for $n \ll 0$.
      \item \concept{\small(Truncation)} 
          For all $a,b \in \mathbb{V}$: $\vo{a}{n}b = 0$ if $n$ is sufficiently large.
      \item \concept{\small(Vacuum)} 
          $Y( \vac, z ) = \id_{\mathbb{V}}$.
      \item \concept{\small(Creation)} 
          For all $a \in \mathbb{V}$: 
          $Y( a, z ){\vac} \in \mathbb{V}\dbrack{z}$ and $\lim\limits_{z\to 0}Y( a, z ){\vac} = a$.
      \item \concept{\small(Jacobi identity)} 
          For all $a, b \in \mathbb{V}$ and $m,n,k \in \Z$:
          \[
            \sum_{i\ge 0}\binom{m}{i}
                \vo{(\vo{a}{k+i}b)}{m+n-i} =
            \sum_{i\ge 0}(-1)^{i}\binom{k}{i}\left(
                    \vo{a}{m+k-i}\vo{b}{n+i} - 
                    (-1)^{k}\vo{b}{n+k-i}\vo{a}{m+i}
                \right).
          \]
      \item \concept{\small(Virasoro relations)} 
          Defining $\nota{\VL{n}}$ ($n \in \Z$) as the coefficients of $Y( \cfv, z )$: 
          \[
            Y( \cfv, z ) = 
            \sum_{n \in \mathbb{Z}}
                \vo{\cfv}{n}z^{-n-1} = 
            \sum_{n \in \mathbb{Z}}
                \VL{n}z^{-n-2},
          \]
          then there is a number $\cch\in\k$, called the \concept{central charge}, such that
          \[
            \Lie*{\VL{m}}{\VL{n}} = 
              (m-n)\VL{m+n} + 
              \delta_{m+n,0}\frac{(m-1)m(m+1)}{12}\cch,
          \]
          where $m,n \in \Z$.
      \item \concept{\small($\VL{0}$-spectrum)} 
          For each $n\in\N$: $\VL{0}|_{\mathbb{V}_{n}} = n\id_{\mathbb{V}_{n}}$.
      \item \concept{\small($\VL{-1}$-derivative)} 
          For all $a \in \mathbb{V}$: $Y( \VL{-1}a, z ) = \pdv{z} Y( a, z )$.
    \end{enumerate}
  \end{definition}
  \begin{definition}
    An \concept{automorphism} of a vertex operator algebra $\mathbb{V}$ is a linear automorphism $g\colon \mathbb{V} \to \mathbb{V}$ preserving $\vac$ and $\cfv$ and such that
    \[
      gY(a,z)g^{-1} = Y(g(a),z)\txforall a \in \mathbb{V}.
    \]
  \end{definition}

  \begin{notation}\label{not:g-eigenspace}
    Let $g$ be an automorphism of $\mathbb{V}$ satisfying $g^T=\id_{\mathbb{V}}$ for some $T \in \N$. 
    Then, we can decompose $\mathbb{V}$ into eigenspaces of $g$:
    \begin{equation}
      \nota{\mathbb{V}^{[r]}}:=
            \Set*{  a \in \mathbb{V}  \mid  g(a) = \zeta_{T}^ra  },
    \end{equation}
    where $\nota{\zeta_{T}}$ is a fixed primitive $T$-th root of unity, and $[r]$ stands for the congruence class of $r$ modulo $T$.  
  \end{notation}
  \begin{remark}
    In the literature, $T$ is often taken to be the order of $g$. However, this does not affect the following discussions. Indeed, $\mathbb{V}^{[r]}$ has to be zero if $[r]$ is not annihilated by the order of $g$.
  \end{remark}

  \begin{definition}
    The $C_2$-subspace $C_2(\mathbb{V})$ of $\mathbb{V}$ is the subspace spanned by all elements of the form $a_{-2}b$, where $a,b \in \mathbb{V}$.
    $\mathbb{V}$ is said to be \concept{$C_2$-cofinite} if the quotient space $\mathbb{V}/C_2(\mathbb{V})$ is finite-dimensional. 
  \end{definition}

  \begin{definition}
    A \concept{weak $g$-twisted $\mathbb{V}$-module} is a vector space $W$ equipped with a linear map
    $
      Y_{W}(\cdot,z)\colon \mathbb{V} \to (\End{W})\dbrack{z^{\pm\sfrac{1}{T}}}
    $
    assigning to each $a \in \mathbb{V}_{k}$ the \concept{twisted vertex operator}
    \[
      Y_{W}(a,z) = \sum_{n\in\frac{1}{T}\Z} \nota{J_{n}^{W}(a)} z^{-n-k},
    \]
    such that the following axioms are satisfied:
    \begin{enumerate}[label={\textbf{M}\arabic*}]
      \item \concept{\small(Truncation)}
          For all $a \in \mathbb{V}$ and $w \in W$: $J_{n}^{W}(a)w = 0$ if $n$ is sufficiently large.
      \item \concept{\small(Vacuum)}
          $Y_{W}(\vac,z) = \id_{W}$.
      \item \concept{\small(Equivariance)}
          For all $a \in \mathbb{V}$:
          \[
            Y_{W}(g(a),z) = \lim_{z^{\sfrac{1}{T}} \to \zeta_{T}^{-1} z^{\sfrac{1}{T}}} Y_{W}(a,z).
          \]
          Note that the right-hand side is a coordinate change of formal power series in $z^{\sfrac{1}{T}}$.
      \item \concept{\small(Twisted Jacobi identity)}
          For all $a\in V^{[r]}, b\in V$, $m\in \frac{r}{T}+\Z, n\in\frac{1}{T}\Z$, and $k\in\Z$:
          \begin{align*}
            \MoveEqLeft
            \sum_{i\ge 0}\binom{\wt a+m-1}{i}
                J_{m+n+k}\left( \vo{a}{k+i}b \right)
            \\
            &=
            \sum_{i\ge 0}(-1)^{i}\binom{k}{i}\left(
                    J_{m+k-i}(a)J_{n+i}(b) - 
                    (-1)^{k}J_{n+k-i}(b)J_{m+i}(a)
                \right).
          \end{align*}
    \end{enumerate}
    A weak $g$-twisted $\mathbb{V}$-module $W$ is called an \concept{admissible $g$-twisted $\mathbb{V}$-module} if it admits a positive grading $W = \bigoplus_{n \in \frac{1}{T}\N} W_n$ such that for all $a \in \mathbb{V}$, $m \in \frac{1}{T}\Z$, and $n \in \frac{1}{T}\N$:
    \[
      J_{m}^{W}(a) W_n \subseteq W_{n-m}.
    \]
  \end{definition}
  \begin{definition}
    A \concept{grading-restricted generalized $g$-twisted $\mathbb{V}$-module} is a weak $g$-twisted $\mathbb{V}$-module $W$ together with a $\k$-grading $W = \bigoplus_{\lambda \in \k} W_{(\lambda)}$ such that the following axioms are satisfied:
    \begin{enumerate}[label={\textbf{M}\arabic*},start=5]
      \item \concept{\small(Grading-restriction)}
          For each $\lambda \in \k$: $\dim_{\k} W_{(\lambda)} < \infty$, and $W_{(\lambda + n)} = 0$ for $n\ll 0$.
      \item \concept{\small(Generalized $\VL{0}$-spectrum)}
          For each $\lambda \in \k$ and $w \in W_{(\lambda)}$: 
          \[
            (\VL{0} - \lambda \id_{W_{(\lambda)}})^{N} w = 0
          \]
          for some $N \in \N$.
    \end{enumerate}
    If the $\VL{0}$-action is semisimple, i.e. the exponent $N$ in \textbf{M}6 can always be taken to be $1$, then $W$ is called an \concept{ordinary $g$-twisted $\mathbb{V}$-module}.
  \end{definition}
  \begin{definition}
    $\mathbb{V}$ is said to be \concept{rational} if any admissible $\mathbb{V}$-module is completely reducible.
  \end{definition}

\subsection{(Twisted) enveloping algebras}
  Let $\mathbb{V}$ be a vertex operator algebra and let $g$ be an automorphism of $\mathbb{V}$ satisfying $g^T=\id_{\mathbb{V}}$ for some $T \in \N$. 
  We are going to construct the \emph{$g$-twisted} enveloping algebra of $\mathbb{V}$. When $g=\id$, it recovers the untwisted enveloping algebra of $\mathbb{V}$.

  \begin{construction}
    Let $\nota{\mathbf{U}_{\sfrac{1}{T}}(\mathbb{V})}$ be the linear space spanned by the (possibly empty) words
    \[  
      J_{n_1}(a^1) \cdots J_{n_k}(a^k),
    \]
    where $a^1,\cdots,a^k\in \mathbb{V}$ and $n_1,\cdots,n_k\in\frac{1}{T}\Z$. We further assume that $J_{n_1}(a^1) \cdots J_{n_k}(a^k)$ is linear in each $a^i$ for $1 \le i \le k$. 
    It is a unital associative algebra with the multiplication given by \emph{concatenation} and the identity $1_{\mathbf{U}_{\sfrac{1}{T}}(\mathbb{V})}$ given by the empty word. 
    This algebra is $\frac{1}{T}\Z$-graded with the grading given by 
    \[
      \deg J_{n_1}(a^1)\cdots J_{n_k}(a^k) = -n_1-\cdots-n_k.
    \]
    Let $\widehat{\mathbf{U}}_{\sfrac{1}{T}}(\mathbb{V})_{\bullet}$ be the degreewise completions of $\mathbf{U}_{\sfrac{1}{T}}(\mathbb{V})_{\bullet}$ with respect to its \emph{canonical seminorm}:
    \[
      \nota{\widehat{\mathbf{U}}_{\sfrac{1}{T}}(\mathbb{V})_{\bullet}} 
      := 
      \varprojlim_{n} 
          \frac{\mathbf{U}_{\sfrac{1}{T}}(\mathbb{V})_{\bullet}}{\NL[n]\mathbf{U}_{\sfrac{1}{T}}(\mathbb{V})_{\bullet}} 
      = 
      \varprojlim_{n'}
          \frac{\mathbf{U}_{\sfrac{1}{T}}(\mathbb{V})_{\bullet}}{\NR[n']\mathbf{U}_{\sfrac{1}{T}}(\mathbb{V})_{\bullet}}.
    \]
    Elements in $\widehat{\mathbf{U}}_{\sfrac{1}{T}}(\mathbb{V})$ are called ($\sfrac{1}{T}$-fractional) \concept{modes}.
  \end{construction}

  \begin{definition}
    \label{def:enveloping-algebra}
    The \concept{$g$-twisted enveloping algebra} $\nota{\mathscr{U}_{g}(\mathbb{V})}$ is the quotient graded algebra of $\widehat{\mathbf{U}}_{\sfrac{1}{T}}(\mathbb{V})$ by the following homogeneous relations:
    \begin{enumerate}[label={\textbf{U}\arabic*}]
      \item \concept{\small(Vacuum)}
          \[
            J_{n}(\vac) \sim \delta_{n,0}1_{\widehat{\mathbf{U}}_{\sfrac{1}{T}}(\mathbb{V})}
            \qquad (n \in \tfrac{1}{T}\Z);
          \] 
      \item \concept{\small($g$-equivariance)}
          \[
            \zeta_{T}^{-nT}J_{n}(g(a)) \sim J_{n}(a)
            \qquad (a \in \mathbb{V}, n \in \tfrac{1}{T}\Z);
          \]
      \item \concept{\small(Twisted Jacobi identity)} 
          \begin{align*}
            \MoveEqLeft
            \sum_{i\ge 0}\binom{\wt a+m-1}{i}
                J_{m+n+k}\left( \vo{a}{k+i}b \right)
            \\
            &\sim
            \sum_{i\ge 0}(-1)^{i}\binom{k}{i}\left(
                    J_{m+k-i}(a)J_{n+i}(b) - 
                    (-1)^{k}J_{n+k-i}(b)J_{m+i}(a)
                \right)
          \end{align*}
          for all $a\in \mathbb{V}^{[r]}, b\in \mathbb{V}$, $m\in \frac{r}{T}+\Z, n\in\frac{1}{T}\Z$, and $k\in\Z$.
    \end{enumerate}
    For simplicity, when $g$ is the identity automorphism, we omit the subscripts $\sfrac{1}{1}$ and $\id$.
  \end{definition}
  \begin{remark}
    It seems that the above definition relies on $T$. But in fact, if we replace $T$ by any multiple of the order of $g$, then the resulting $g$-twisted enveloping algebra is isomorphic to the original one.
    Indeed, by the $g$-equivariance property, we must have $J_{n}(a) = 0$ whenever $a \in \mathbb{V}^{[r]}$ but $n\not\in \frac{r}{T}+\Z$.
  \end{remark}

  Modules of $\mathscr{U}_{g}(\mathbb{V})$ provide the weakest notion of modules of $\mathbb{V}$ that is even weaker the so called \emph{weak modules}. 
  By spelling out the definitions, we have the following dictionary:
  \begin{proposition}
    Let $W$ be a weak $g$-twisted $\mathbb{V}$-module. Then, we have:
    \begin{enumerate}
      \item $W$ is an exhaustive $\mathscr{U}_{g}(\mathbb{V})$-module if and only if it satisfies the following \emph{weak truncation}:
      \begin{quote}
        for any $w \in W$, we require $J^{W}_{n_1}(a^1) \cdots J^{W}_{n_k}(a^k)w = 0$ whenever $n_1+\cdots+n_k$ is sufficiently large.
      \end{quote}
      \item $W$ is a positively-graded (or positively-gradable) $\mathscr{U}_{g}(\mathbb{V})$-module if and only if it is admissible\footnote{The usage of the terminology ``admissible module'' in literature is ambiguous: it could be either an object in the category of graded modules, or its image in the category of weak modules.}.
      \item \label{prop:whatISmodV}
      If $W$ is a grading-restricted generalized $g$-twisted $\mathbb{V}$-module, then it is a direct sum of quasi-rigid (cf. \zcref{def:ordinary}) $\mathscr{U}_{g}(\mathbb{V})$-modules.
      Over an algebraically closed field $\k$, the converse also holds.
      \item Whenever $W$ is graded, the contragredient module $W'$ of $W$ is the degreewise dual $\mathscr{U}_{g}(\mathbb{V})$-module of $W$.
    \end{enumerate}
  \end{proposition}

  From the above interpretations, we have
  \begin{corollary}
    $\mathbb{V}$ is $g$-rational if and only if $\mathscr{U}_{g}(\mathbb{V})$ is $\k$-rational.
  \end{corollary}

  In the rest of this section, we assume $\k$ is an algebraically closed field.

  Recall that:
  \begin{definition}
    A VOA $\mathbb{V}$ is called of \concept{CFT type} if $\mathbb{V}_0 = \k\vac$ and $\mathbb{V}_n = 0$ for $n < 0$. 
  \end{definition}
  If this is the case, then for any $\mathscr{U}_{g}(\mathbb{V})$, its Zhu algebra $\mathsf{A}_{0}$ is augmented: the augmentation is given by the module action map $\mathsf{A}_{0} \to \End_{\k}(\mathbb{V}_{0}) = \k$. Hence, by the discussion of Part 2, we have:
  \begin{theorem}\label{thm:SIC-CFT}
    For a VOA $\mathbb{V}$ of CFT type with a finite automorphism $g$, the following are equivalent:
    \begin{enumerate}
      \item $\mathscr{U}_{g}(\mathbb{V})$ satisfies the strong identity condition (\SIC).
      \item $\mathscr{U}_{g}(\mathbb{V})$ admits a strong identity expansion (\SIE).
      \item The \zcref[noname]{Omegasplit} condition holds for $\mathscr{U}_{g}(\mathbb{V})$.
      \item $\mathfrak{L}^{n}$ (resp. $\mathfrak{R}^{n}$) are projective objects in $\Mod[Ex](\mathscr{U}_{g}(\mathbb{V}))$.
      \item The adjunction 
        \[
          \begin{tikzcd}
            \Phi^{\mathsf{L}}_{0}\colon \Mod(\mathsf{A}_{0}|) & 
            \Mod[Ex](U|)
            \colon\Omega^{\mathsf{L}}_{0}.
            \ar[from=1-1, to=1-2, phantom, "\scriptstyle\bot" description]
            \ar[from=1-1, to=1-2, shift left=1ex] 
            \ar[from=1-2, to=1-1, shift left=1ex] 
          \end{tikzcd}
        \]
      is a Morita-type equivalence.
      \item For all $n$, we have Morita-type equivalence
        \[
          \begin{tikzcd}
            \mathbf{\Phi}^{\mathsf{L}}_{n}\colon \Mod(\mathscr{A}^{\mathsf{L}}_{n}|) & 
            \Mod[Ex](U|) \colon\Omega^{\mathsf{L}}_{n}.
            \ar[from=1-1, to=1-2, phantom, "\scriptstyle\bot" description]
            \ar[from=1-1, to=1-2, shift left=1ex] 
            \ar[from=1-2, to=1-1, shift left=1ex] 
          \end{tikzcd}
        \] 
    \end{enumerate}
    And if this is the case, any exhaustive $\mathscr{U}_{g}(\mathbb{V})$-module is admissible. 
  \end{theorem}
  \begin{remark}
    Even if $\mathbb{V}$ is not of CFT type or $\k$ is not an algebraically closed field, the implications except for (v) $\Rightarrow$ (i) still hold. See \zcref{sec:Morita_to_SIC} for more conditions that guarantee (v) $\Rightarrow$ (i).
  \end{remark}

  It is shown in \cite[Theorem 9.2.1]{MNT10} that if $\mathbb{V}$ is $C_2$-cofinite, then $\mathscr{U}(\mathbb{V})$ is quasi-finite, namely all components of $\mathfrak{L}^{n}, \mathfrak{R}^{n}$ are finite-dimensional over $\k$. Hence, we have
  \begin{corollary}
    If $\mathbb{V}$ is furthermore $C_2$-cofinite, then the above conditions are equivalent to the following:
    \begin{enumerate}[start=6]
      \item We have the adjunction 
        \[
          \begin{tikzcd}
            \Phi^{\mathsf{L}}_{0}\colon \Mod*(\mathsf{A}_{0}|) & 
            \Mod*[Ex](U|)
            \colon\Omega^{\mathsf{L}}_{0}.
            \ar[from=1-1, to=1-2, phantom, "\scriptstyle\bot" description]
            \ar[from=1-1, to=1-2, shift left=1ex] 
            \ar[from=1-2, to=1-1, shift left=1ex] 
          \end{tikzcd}
        \]
      and it is a Morita-type equivalence.
      \item For all $n$, we have Morita-type equivalence
        \[
          \begin{tikzcd}
            \mathbf{\Phi}^{\mathsf{L}}_{n}\colon \Mod*(\mathscr{A}^{\mathsf{L}}_{n}|) & 
            \Mod*[Ex](U|) \colon\Omega^{\mathsf{L}}_{n}.
            \ar[from=1-1, to=1-2, phantom, "\scriptstyle\bot" description]
            \ar[from=1-1, to=1-2, shift left=1ex] 
            \ar[from=1-2, to=1-1, shift left=1ex] 
          \end{tikzcd}
        \] 
    \end{enumerate}
  \end{corollary}

  In particular, in the context of \cite{DGK23}, we have
  \begin{corollary}\label{coro:CFT-VOA}
    For a VOA $\mathbb{V}$ of CFT type, the following conditions are equivalent:
    \begin{enumerate}
      \item $\mathbb{V}$ satisfies smoothing (cf. \cite[Definition 5.0.2]{DGK23}).
      \item The equivalent conditions in \zcref{thm:SIC-CFT} hold. In particular, all admissible $\mathbb{V}$-modules are induced from $\mathsf{A}_{0}$-modules.
      \item Any sheaf of coinvariants of $\mathbb{V}$-modules, assuming coherent, is locally free on any of the moduli spaces $\widehat{\overline{\Moduli}}_{g,n}$.
      \item The mode transition algebra $\mathfrak{A}$ equals to 
      \begin{equation}\label{eq:Mode=End}
        \mathfrak{A}(\mathscr{U}(\mathbb{V})) = \int_{W} W \otimes_{\k} W^{\prime}
      \end{equation}
      where the end is taken over all grading-restricted generalized $\mathbb{V}$-modules.
    \end{enumerate}
  \end{corollary}
  \begin{proof}
    The equivalence of \SIC and smoothing is due to \cite[Theorem 5.0.3]{DGK23}. 
    Under these conditions, (iii) follows from \cite[Corollary 5.2.6]{DGK23}. 
    Suppose (i) fails, by \zcref{thm:SIC-CFT}, there is a grading-restricted generalized $\mathbb{V}$-module $W$ that is not induced from $\mathsf{A}_{0}$-modules. 
    Then, by the main theorem of \cite{Zhang-NonEquiv}, (iii) fails.
    By \zcref{prop:whatISmodV}, the category of grading-restricted generalized $\mathbb{V}$-modules and the category of quasi-rigid $\mathscr{U}(\mathbb{V})$-modules have the same ind-completion (in $\Mod[Ex](\mathscr{U}(\mathbb{V}))$). In particular, the formula in \zcref{coro:EndMode} can be rewritten as stated.
    On the other hand, the main theorem of \cite{Zhang-NonEquiv} shows that (iv) fails when (i) fails.
  \end{proof}

\section{Examples of VOAs that fail the strong identity condition}\label{sec:Examples}

In this section, we analysis two examples of VOAs that fail the strong identity condition, in addition to those studied in \cite{DGK23}.

\subsection{The cyclic orbifold VOA $M(1)^+$}
	In this subsection, we show that the cyclic orbifold $M(1)^+$ \cite{FLM,DN99} of the rank-one Heisenberg VOA $M(1) = M_{\widehat{\mathbb{C}\alpha}}(1,0)$ does not satisfy the strong identity property. Hence \emph{SIC is not closed under taking orbifold}.  
	\subsubsection{Basics of $M(1)^+$}
	
	We first recall some basic facts about $M(1)^+$, see \cite{DN99} for more details. 
  Assume $(\alpha|\alpha)=1$. 
	By definition, $M(1)^+$ is the fixed-point subVOA of $M(1)$ with respect to the order-two automorphism: 
	\[
	\theta: M(1)\longrightarrow M(1),\quad 
  \alpha_1(n_1) \cdots \alpha_r(n_r)\vac \mapsto (-1)^r \alpha_1(n_1) \cdots \alpha_r(n_r)\vac. 
	\]
	$\mathbb{V} = M(1)^+$ is a CFT-type simple VOA, with two generators in degree $2$ and $4$: 
	\[
    \omega = \frac{1}{2}\alpha(-1)^2\vac\in \mathbb{V}_2\quad \mathrm{and}\quad J=\alpha(-1)^4\vac-2\alpha(-3)\alpha(-1)\vac+\frac{3}{2}\alpha(-2)^2\vac\in \mathbb{V}_4,
	\]
	where $\omega$ is the Virasoro element of $M(1)^+$. 
	The irreducible modules are given by 
	\[
    M(1)^+,\quad M(1)^-,\quad M(1,\lambda)\ (\lambda\neq 0),\quad M(1)(\theta)^+,\quad M(1)(\theta)^-,
  \]
	where $M(1)(\theta)$ is a $\theta$-twisted $M(1)$-module. Let $e^\lambda$ be the highest-weight vector of $M(1,\lambda)$, the following formulas are useful in our later discussion:
	\begin{equation}\label{eq:4.1}
		L(-1)e^\lambda=\lambda\alpha(-1)e^\lambda,\quad J_{(2)}e^\lambda=(4\lambda^2-2) L(-1)e^\lambda. 
	\end{equation} 
	
	Let $\mathscr{U}=\mathscr{U}(M(1)^+)$. 
	The Zhu algebra $\mathsf{A}$ of $M(1)^+$ admits the following identification:
	\begin{align*}
		\mathsf{A} \cong \mathscr{U}_0/\NL[1]\mathscr{U}_0 &\overset{\simeq}{\longrightarrow} \mathbb{C}\brk[s]{x,y}/\brk[a]{P,Q},\\
		[\omega]=\omega_{[1]} &\longmapsto x,\\
    [J]=J_{[3]} &\longmapsto y,
	\end{align*}
	where $[\omega] = \omega + O(M(1)^+) \in \mathsf{A}$ and $\omega_{[1]} \in \mathscr{U}_0$ are different notation for the same element in the two descriptions of $\mathsf{A}$, and 
	\[
    P=(y+x-4x^2)(70y+908x^2-515x+27),\quad Q=(x-1)(x-\tfrac{1}{16})(x-\tfrac{9}{16})(y+x-4x^2).
	\]
	Moreover, the action of $[\omega] $ and $[J]$ on $e^\lambda\in M(1,\lambda)$ are given by 
	\begin{equation}\label{eq:4.2}
		[\omega].e^\lambda=\frac{\lambda^2}{2} \cdot e^\lambda,\quad [J].e^\lambda=\left(\lambda^4-\frac{\lambda^2}{2}\right)\cdot e^\lambda.
	\end{equation}

\subsubsection{Generalized Verma modules are not irreducible}
	\begin{lemma}\label{lm:M(1)+}
		Let $\mathscr{U}=\mathscr{U}(M(1)^+)$. 
    Then $L(-1) = \omega_{[0]}$ and $J_{[2]}$ define nonzero elements in $\mathfrak{L}^{0}_{1} = \mathscr{U}_1/\NL[1] \mathscr{U}_1$.
	\end{lemma}
	\begin{proof}
    Consider the exhaustive $\mathscr{U}$-module $M(1,\lambda)$ with $\lambda \neq 0$, and let $e^\lambda$ be its highest-weight vector.
    Then, for any $j \ge 1$, the subspace $\mathscr{U}_{-j}$ annihilates $e^\lambda$. 
    This immediately implies that the entire $\NL[1]\mathscr{U}_1$ annihilates $e^\lambda$.
    On the other hand, by \eqref{eq:4.1}, the action of $\omega_{[0]} \in \mathscr{U}_1$ on $e^\lambda$ gives:
    \[
      \omega_{[0]}.e^\lambda = L(-1)e^\lambda = \lambda \alpha(-1)e^\lambda \neq 0.
    \]
    Therefore, $\omega_{[0]} \notin \NL[1]\mathscr{U}_1$. 
    Similarly, $J_{[2]} \notin \NL[1]\mathscr{U}_1$.
    The statement follows immediately.
	\end{proof}
	
	Since $M(1)^+$ has PBW-type spanning elements \cite[Proposition 3.4]{DN99}:
	\[
    L(-m_1)\dots L(-m_s)J_{(-n_1)}\dots J_{(-n_t)}\vac,
	\]
	where $m_1\ge \dots \ge m_s\ge 2$ and $n_1\ge \dots \ge n_t\ge 1$, and the quotient space $\mathfrak{L}^{0}_{1} = \mathscr{U}_1/\NL[1] \mathscr{U}_1$ has spanning elements: 
	\[
    a^1_{[n_1]}\dots a^r_{[n_r]}+\NL[1] \mathscr{U}_1,\quad  \deg(a^i_{[n_i]})\ge 0,\ \forall i,\ \sum_{i=1}^r \deg(a^i_{[n_i]})=1,
	\]
	we can move the degree one terms $L(-1)=\omega_{[0]}$ and $J_{[2]}$ to the left, and all the degree zero terms, which are polynomials of $\omega_{[1]}$ and $J_{[3]}$, to the right modulo $\NL[1] \mathscr{U}_1$, so then the spanning elements of $\mathfrak{L}^{0}_{1}$ can be written as 
	\begin{equation}\label{eq:spn}
		L(-1)\cdot F(\omega_{[1]},J_{[3]})+J_{[2]}\cdot G(\omega_{[1]},J_{[3]})+ \NL[1] \mathscr{U}_1, 
	\end{equation}
	where $F(x,y),G(x,y)\in \mathbb{C}[x,y]$. 
	
	\begin{lemma}\label{lm:L(-1)nonzero}
    View the $\mathsf{A}$-module $\mathbb{C}\vac$ as the bottom component of $M(1)^+$. Then 
		\[
      \mathfrak{L}^{0}_{1} \otimes_\mathsf{A} \mathbb{C}\vac
      \cong 
      \mathfrak{L}^{0}_{1}/\mathfrak{L}^{0}_{1}\brk[a]{x,y} 
      \neq 0,
		\]
		where $\brk[a]{x,y}$ is the ideal of $\mathbb{C}[x,y]$ generated by $x,y$.  
	\end{lemma}
	\begin{proof}
    The left action of $\mathsf{A}$ on $\mathbb{C}\vac$ induces the left action of $\mathbb{C}[x,y]$ as $x.\vac = L(0)\vac = 0$ and $y.\vac = J_{(3)}\vac = 0$. 
    Hence, we have $\mathfrak{L}^{0}_{1}\otimes_\mathsf{A} \mathbb{C}\vac \cong \mathfrak{L}^{0}_{1}/\mathfrak{L}^{0}_{1}\brk[a]{x,y}$. 
    By Lemma~\ref{lm:M(1)+}, $L(-1)$ defines a nonzero elementin $\mathfrak{L}^{0}_{1}$. 
		We claim that its image $\overline{L(-1)}$ in $\mathfrak{L}^{0}_{1}/\mathfrak{L}^{0}_{1}\brk[a]{x,y}$ is nonzero. 
		
		Indeed, suppose $\overline{L(-1)}=0$. Then by \eqref{eq:spn}, there exists polynomials $F(x,y),G(x,y) \in \brk[a]{x,y}$ such that the following relation holds in $\mathfrak{L}^{0}_{1}$: 
		\begin{equation}\label{eq:L(-1)eq}
			L(-1) \equiv L(-1)F(\omega_{[1]},J_{[3]})+J_{[2]} G(\omega_{[1]},J_{[3]}) \mod{\NL[1] \mathscr{U}_1}.
		\end{equation}
    Consider the exhaustive $\mathscr{U}$-module $M(1,\lambda)$ with $\lambda \neq 0$. 
    We may apply both sides of \eqref{eq:L(-1)eq} to $e^\lambda$, and notice that the action of $\NL[1] \mathscr{U}_1$ vanishes on the highest-weight vector $e^\lambda$.
    By \eqref{eq:4.1} and \eqref{eq:4.2}, we have the following equality in $M(1,\lambda)$:
		\[
      L(-1)e^\lambda=\left(F(\lambda^2/2,\lambda^4-\lambda^2/2)+ (4\lambda^2-2) G(\lambda^2/2,\lambda^4-\lambda^2/2)\right)\cdot L(-1)e^\lambda.
		\]
		Since $L(-1)e^\lambda=\lambda \alpha(-1)e^\lambda\neq 0$ and $\lambda$ is arbitrary, we have a polynomial identity:
		\[
      F\left(\frac{t}{2},t^2-\frac{t}{2}\right)+(4t-2)G\left(\frac{t}{2},t^2-\frac{t}{2}\right)=1
		\]
		Since $F(x,y),G(x,y) \in \brk[a]{x,y}$, the variable $t$ divides the left hand side of the equation above, which is a contradiction. 
	\end{proof}

	\begin{proposition}
		The enveloping algebra $\mathscr{U}(M(1)^+)$ fails the strong identity condition.
	\end{proposition}
	\begin{proof}
		Suppose not, then by \zcref{thm:PhiOmega0}, the irreducible adjoint module $M(1)^+$ is isomorphic to the generalized Verma module associated to $\Omega_0(M(1)^+)$. 
    Since the $\mathscr{U}_{0}$-module $\mathbb{C}\vac$ is a direct summand of $\Omega_0(M(1)^+)$, we see that the homogeneous map (which is the counit for $\Phi^{\mathsf{L}}\dashv(-)_{0}$)
		\[
      \Phi^{\mathsf{L}}(\mathbb{C}\vac) = 
      \mathfrak{L}^{0}\otimes_{\mathsf{A}} \mathbb{C}\vac
      \longrightarrow 
      M(1)^+,\quad a_{[n]}\otimes \vac \longmapsto a_{(n)}\vac
		\]
		is injective. 
    By \zcref{lm:L(-1)nonzero}, the degree-one subspace 
    \[
      \Phi^{\mathsf{L}}(\mathbb{C}\vac)_1 = \mathfrak{L}^{0}_{1} \otimes_\mathsf{A} \mathbb{C}\vac
    \]
		is nonzero. 
    On the other hand, the degree-one subspace of $M(1)^+$ is zero. This is a contradiction.  
	\end{proof}

\subsection{The (universal) affine VOAs \texorpdfstring{$V^k(\mathfrak{g})$}{Vk(g)} and \texorpdfstring{$L_k(\mathfrak{g})$}{Lk(g)}}

  In this subsection, we show that the positive integral level $k\in \mathbb{Z}_{>0}$ universal affine VOAs $V^k(\mathfrak{g})$ and the generic level $k+h^\vee\notin \mathbb{Q}_{\ge 0}$ affine VOAs $L_k(\mathfrak{g})=V^k(\mathfrak{g})$ do not satisfy the strong identity condition. This is in contrast to the affine VOAs $L_k(\mathfrak{g})$ at positive integral levels, which are known to satisfy this condition (cf. \cite{DGK23}).

\subsubsection{Basics of affine VOAs}
  We first recall some basic facts about affine VOAs, see \cite{FZ92} for more details.

  Let $\mathfrak{g}$ be a finite-dimensional semisimple Lie algebra with a Cartan subalgebra $\mathfrak{h}$, and let $\Delta$ be the root system associated to $\mathfrak{h}$ with root lattice $Q\subset \mathfrak{h}^\ast$. Let $P\subset \mathfrak{h}^\ast$ be the weight lattice. Normalize the invariant bilinear form on $\mathfrak{g}$ so that $(\theta|\theta)=2$, where $\theta$ is the longest root of $\Delta$. Let $\hat{\mathfrak{g}}=\mathfrak{g}\otimes \mathbb{C}[t,t^{-1}]\oplus \mathbb{C} K$ be its affinization, with Lie bracket given by
  \[
    \Lie*{K}{\hat{\mathfrak{g}}}=0,
    \quad 
    \Lie*{a(m)}{b(n)} = \Lie*{a}{b}(m+n)+m\delta_{m+n,0}(a|b) K,\quad 
    a,b\in \mathfrak{g},\ m,n\in \mathbb{Z}.
  \]
  Let $\hat{\mathfrak{g}}_{\pm}=\mathfrak{g}\otimes t^{\pm}\mathbb{C}[t^{\pm}]$, and $\hat{\mathfrak{g}}_0=\mathfrak{g}\oplus \mathbb{C} K$. This leads to a triangular decomposition:
  \begin{equation}\label{eq:trihatg}
    \hat{\mathfrak{g}} = \hat{\mathfrak{g}}_{<0}\oplus \hat{\mathfrak{g}}_0\oplus \hat{\mathfrak{g}}_{>0},
  \end{equation}
  Let $\mathbb{C} \mathbf{1}$ be a $\hat{\mathfrak{g}}_{\ge 0}$-module with $K.\mathbf{1}=k\mathbf{1}$ and $\mathfrak{g}\otimes \mathbb{C}[t]$ annihilates $\mathbf{1}$. The Weyl vacuum module
  \[
    V^{k}(\mathfrak{g})=U(\hat{\mathfrak{g}})\otimes _{U(\hat{\mathfrak{g}}_{\ge 0})} \mathbb{C} \mathbf{1}
  \]
  is a vertex algebra, with $\mathbf{1} =1\otimes \mathbf{1}$. 
  If the level $k$ is non-critical i.e., $k\neq -h^\vee$, then $V^k(\mathfrak{g})$ is a VOA, with Virasoro element $\omega_{\mathrm{aff}}=\frac{1}{2(h^\vee+k)}\sum_{i=1}^{\dim \mathfrak{g}}u^i(-1)u_i(-1)\mathbf{1}$, called the {\em universal affine VOA (or vacuum module VOA) at level $k$}, where $h^\vee$ is the dual Coxeter number of $\Delta$, and $\{u^i\}$ and $\{u_i\}$ are dual orthonormal basis of $\mathfrak{g}$.

  For non-critical level $k$, $V^{k}(\mathfrak{g})$ has a unique maximal $\hat{\mathfrak{g}}$-submodule $W^k(\mathfrak{g})$. The irreducible quotient $L_{k}(\mathfrak{g})=V^{k}(\mathfrak{g})/W^k(\mathfrak{g})$ is also a VOA called the {\em affine VOA at level $k$}.

  If $k\in \mathbb{Z}_{>0}$, then $W^k(\mathfrak{g})=U(\hat{\mathfrak{g}})e_{\theta}^{k+1}(-1)\mathbf{1}$, where $e_\theta\in \mathfrak{g}_\theta$, and $L_k(\mathfrak{g})$ is a strongly rational VOA (cf. \cite{DLM97}).

  If $k\notin \mathbb{Z}_{>0}$, there are several cases. Let $r^\vee$ be the lacing number of $\Delta$. By \cite{GK07},
  \[
    W^k(\mathfrak{g})\neq 0\iff r^\vee(k+h^\vee)\in \mathbb{Q}_{\ge 0}\setminus\left\{\frac{1}{m}:m\in \mathbb{Z}_{>0} \right\}.
  \]
  In particular, for generic level $k+h^\vee \in \mathbb{C}\setminus\mathbb{Q}_{\ge 0}$, $V^k(\mathfrak{g})$ must be simple. i.e., $V^k(\mathfrak{g})=L_k(\mathfrak{g})$. One important class of $k\in \mathbb{Q}_{\ge 0}$ is the Kac-Wakimoto admissible level \cite{KW88}:
  \[
    k+h^\vee = \frac{p}{q},\quad p,q\in \mathbb{N},\ \gcd(p,q)=1,\ p\ge 
    \begin{cases}
      h^\vee &\mathrm{if}\ (r^\vee,q)=1,\\
      h &\mathrm{if}\ (r^\vee,q)\neq 1.
    \end{cases}
  \]
  In this case, $V^k(\mathfrak{g})\neq L_k(\mathfrak{g})$, and the admissible-level affine VOA $L_k(\mathfrak{g})$ is not rational \cite{A16}.

  The Zhu algebras of $V^k(\mathfrak{g})$ and $L_k(\mathfrak{g})$ have the following description:
  \[
    A(V^k(\mathfrak{g}))\cong U(\mathfrak{g}),\quad A(L_k(\mathfrak{g}))\cong U(\mathfrak{g})/I_k,
  \]
  where $I_k=(W^k(\mathfrak{g})+O(V^k(\mathfrak{g})))/O(V^k(\mathfrak{g}))$ is a two-sided ideal of $U(\mathfrak{g})$.

\subsubsection{Generalized Verma module over $V^k(\mathfrak{g})$}
  From a direct calculation, Cai proved $V^k(\mathfrak{sl}_2)$ does not satisfy the strong identity condition when $k\neq -2$ \cite{arXiv:2601.04187}. 
  We generalize it to arbitrary simple Lie algebras $\mathfrak{g}$, but with the level $k\in \mathbb{Z}_{>0}$ or $k+h^\vee\notin \mathbb{Q}_{\ge 0}$.

  \begin{lemma}\cite{FZ92,FBZ04}\label{lm:UVkg}
    The enveloping algebra $\mathscr{U}=\mathscr{U}(V^k(\mathfrak{g}))$ is isomorphic to
    \[
      \widetilde{U}(\mathfrak{g},k)= \widetilde{U}(\hat{\mathfrak{g}})/\brk[a]{K-k\mathrm{Id}},
    \]
    where $\widetilde{U}(\hat{\mathfrak{g}})$ is the completion of $U(\hat{\mathfrak{g}})$ with respect to its canonical seminorm (recall the $\Z$-grading on $\hat{\mathfrak{g}}$).
  \end{lemma}

  In particular, $a_{[m]}\in L(V^k(\mathfrak{g}))$ is identified with $a(m)\in \widetilde{U}(\mathfrak{g},k)$. Then $\mathscr{U}_n$ is the closure of the subspace spanned by the following elements:
  \begin{equation}\label{eq:scrUaff}
    a^1(m_1)\dots a^r(m_r),\quad a^i\in \mathfrak{g},\ m_i\in \mathbb{Z},\ \sum_{i=1}^r m_i=-n.
  \end{equation}
  
  The following Lemma is an immediate consequence:

  \begin{lemma}
    Given any exhaustive $V^k(\mathfrak{g})$-module $W$, the bottom degree $\Omega_0(W)$ is the space of ``highest-weight vectors'' with respect to the triangular decomposition of $\hat{\mathfrak{g}}$ \eqref{eq:trihatg}:
    \begin{equation}\label{eq:Omaff}
      \Omega_0(W) = \operatorname{span}\SetCond{ w\in W }{ a(n)w=0,\ a\in\mathfrak{g},\ n>0 }.
    \end{equation}

    Given any $\mathsf{A}=A(V^k(\mathfrak{g}))=U(\mathfrak{g})$-module $S$, the generalized Verma module $\Phi^\mathsf{L}(S)$ is isomorphic to the Weyl vacuum module associated to the $\mathfrak{g}$-module $S$:
    \begin{equation}\label{eq:PhiLaff}
      \Phi^\mathsf{L}(S)\cong U(\hat{\mathfrak{g}})\otimes _{U(\hat{\mathfrak{g}}_{\ge 0})}S=V^k(\mathfrak{g},S),
    \end{equation}
    where the $\hat{\mathfrak{g}}_{\ge 0}$ action on $S$ is given by $a(n)v=0$ for all $n>0$, and $K.v=k\cdot v$, for all $v\in S$.
  \end{lemma}
  \begin{proof}
    Let $\mathscr{U}=\mathscr{U}(V^k(\mathfrak{g}))\cong \widetilde{U}(\mathfrak{g},k)$.
    It follows from \eqref{eq:scrUaff} that the $\mathsf{N}_{\mathsf{L}}^{1}\mathscr{U}$ is the closure of the left ideal generated by the elements:the closure of the left ideal generated by the elements: 
		\[
      \alpha = a^1(m_1)\dots a^r(m_r),\quad a^i\in \mathfrak{g},\ m_i\in \mathbb{Z},\ \sum_{i=1}^r m_i\ge 1. 
		\]
    Then, for any $i$, we have
    \[
    \begin{aligned}
      \alpha.w &= a^1(m_1)\dots \widehat{a^i(m_i)}\dots a^r(m_r) (a^i(m_i).w) \\
      &+ \sum_{l=i+1}^r a^1(m_1)\dots \widehat{a^i(m_i)}\dots [a^i,a^l](m_i+m_l)\dots a^r(m_r).w.
    \end{aligned}
    \]
    Repeating this, we see that $\mathsf{N}_{\mathsf{L}}^{1}\mathscr{U}$ is the closure of the left ideal generated by the elements $a(n)$ with $a\in \mathfrak{g}$ and $n>0$.
    This shows that first statement. 

    Hence exhaustive $\mathscr{U}$-modules are the same as restricted $\hat{\mathfrak{g}}$-modules. 
    Then, by the universal property of $\Phi^\mathsf{L}(S)$ and the fact that $\mathsf{A}\cong U(\mathfrak{g})$, we have
		\[
      \Hom_{\mathscr{U}}(\Phi^\mathsf{L}(S),W)
      \cong 
      \Hom_\mathsf{A}(S,\Omega_0(W))
      \cong 
      \Hom_{\hat{\mathfrak{g}}}(V^k(\mathfrak{g},S),W),
		\]
    for all exhaustive $\mathscr{U}$-module $W$. This shows \eqref{eq:PhiLaff}.
  \end{proof}

  \begin{proposition}
    Let $\mathfrak{g}$ be a simple Lie algebra. The universal affine VOA $V^k(\mathfrak{g})$ fails the strong identity condition for generic level $k+h^\vee\notin \mathbb{Q}_{\ge 0}$ or positive integral level $k\in \mathbb{Z}_{>0}$.
  \end{proposition}
  \begin{proof}
    We first assume $k\in \mathbb{Z}_{>0}$. 
    Given an irreducible $U(\mathfrak{g})$-module $S = L(\lambda)$, with $\lambda\in P_+$, by \eqref{eq:PhiLaff} and \cite{FZ92}, the generalized Verma module $\Phi^\mathsf{L}(S)=V^k(\mathfrak{g},\lambda)$ has a nonzero singular vector $e_\theta(-1)^{k-(\lambda|\theta)+1}v_\lambda$, where $v_\lambda\in L(\lambda)$ is the highest-weight vector. Consequently, this singular vector belongs to the bottom space $\Omega_0(\Phi^\mathsf{L}(S))$ but is not in $S=L(\lambda)$. This implies that the adjunction unit $S \to \Omega_0(\Phi^\mathsf{L}(S))$ is not an isomorphism, contradicting \zcref[noname]{PhiOmegaL}.

    Now, we assume $k+h^\vee\notin \mathbb{Q}_{\ge 0}$. Then, $V^k(\mathfrak{g}) = L_k(\mathfrak{g})$ is the affine VOA \cite{GK07}. 
    Note that there exists $\lambda\in \mathfrak{h}^\ast$ such that $\pairing{\lambda}{\theta^\vee}=k$. 
    Indeed, since $\pairing{\rho}{\theta^\vee} = h^\vee-1\neq 0$, where $\rho=\frac{1}{2}\sum_{\alpha\in \Delta_+}\alpha$, we may choose
    \[
      \lambda=\frac{k}{h^\vee-1} \rho\in \mathfrak{h}^\ast.
    \]
    Consider the Verma module $M(\lambda)$ over the finite Lie algebra $\mathfrak{g}$. For any $\alpha\in \Delta_+$, we have
    \[
      \pairing{\lambda+\rho}{\alpha^\vee} = \left(\frac{k}{h^\vee-1}+1\right)\cdot \pairing{\rho}{\alpha^\vee}\notin \mathbb{Z}_{>0}.
    \]
    Indeed, if the right hand side is equal to some positive integer $n$, then $k+h^\vee=\frac{n}{\pairing{\rho}{\alpha^\vee}}(h^\vee-1)+1$, a non-negative rational number, which contradicts the assumption for $k$. Therefore, the Verma module $S=M(\lambda)$ is an irreducible $\mathsf{A}=U(\mathfrak{g})$-module.

    Consider the generalized Verma module $\Phi^\mathsf{L}(M(\lambda))$, which equals $V^k(\mathfrak{g},\lambda)$ by \eqref{eq:PhiLaff}, for the VOA $V^k(\mathfrak{g})$. 
    Let $v_\lambda\in M(\lambda)$ be the highest-weight vector. Choose roots vectors in $\mathfrak{g}_{\pm \theta}$ so that
    \[
      \Lie*{e_\theta}{f_\theta} = \theta^\vee,\quad (e_\theta|f_\theta)=1.
    \]
    Let $v=e_\theta(-1)v_\lambda\in V^k(\mathfrak{g},\lambda)\setminus\{0\}$. We claim that $v$ is a $\hat{\mathfrak{g}}$-singular vector. i.e., $\hat{\mathfrak{g}}_{>0}.v=0$. Indeed, if $n\ge 2$, then for any $a\in \mathfrak{g}$, we have
    \[
      a(n).e_\theta(-1)v_\lambda = [a,e_\theta](n-1)v_\lambda=0,
    \]
    since $\hat{\mathfrak{g}}_{>0}.M(\lambda)=0$. For $n=1$, we have
    \[
      a(1).e_{\theta}(-1)v_\lambda=[a,e_\theta](0)v_\lambda+k(a|e_\theta)\cdot v_\lambda.
    \]
    Since $v_\lambda$ is a highest-weight vector, $[a,e_\theta](0)v_\lambda=[a,e_\theta].v_\lambda=0$ unless $a=f_\theta$. By the definition of Killing form, we also have $(a|e_\theta)=0$ unless $a=f_\theta$. Now for $a=f_\theta$, we have
    \[
      \Lie*{f_\theta}{e_\theta}(0)v_\lambda + k(f_\theta|e_\theta)\cdot v_\lambda = 
      -\theta^\vee.v_\lambda + k\cdot v_\lambda = 
      (-\pairing{\lambda}{\theta^\vee}+k)\cdot v_\lambda=0,
    \]
    by our choice of $\lambda$. This shows $\hat{\mathfrak{g}}_{>0}.v=0$. By \eqref{eq:Omaff}, we have $v\in \Omega_0(V^k(\mathfrak{g},\lambda))\setminus M(\lambda)$. Hence $\Omega_0(V^k(\mathfrak{g},\lambda))\neq M(\lambda)$. This strictly violates \zcref[noname]{PhiOmegaL}. Thus, $V^k(\mathfrak{g})$ does not satisfy the strong identity condition.
  \end{proof}

\begin{remark}
The reason why $V^k(\mathfrak{g})$ fails the strong identity condition is because the Zhu algebra $\mathsf{A}=U(\mathfrak{g})$ does not carry any information about the level $k$, which is essential in the representation theory of $\hat{\mathfrak{g}}$. 
Hence, it is natural to expect that Zhu algebra $\mathsf{A}$ itself is not enough to describe the representation theory of $V^k(\mathfrak{g})$.

For the admissible level $k+h^\vee=p/q$. The situation is more complicated. Although a subcategory of the exhaustive module category of the admissible level affine VOA $L_k(\mathfrak{g})$ is semisimple \cite{A16}, the VOA $L_k(\mathfrak{g})$ itself is not rational. Furthermore, there is no easy description of the generalized Verma modules as in \eqref{eq:PhiLaff}. We conjecture that $L_k(\mathfrak{g})$ also fails the strong identity condition.
\end{remark}

\section{Perspectives}

  In this paper, we have identified the strong identity condition (SIC) as the precise representation-theoretic mechanism governing algebraic smoothing through mode transition algebras. This reorganizes the relation between Zhu-type induction, exhaustive modules, and sewing constructions into a unified framework. At the same time, the failure of SIC in non-semisimple settings indicates that the present framework is not yet final. We conclude by briefly indicating several directions suggested by the results of this paper.

\subsection{Generalized graded structures}

  Throughout this paper, we have worked primarily with $\Z$-graded algebras and modules. 
  A natural generalization is to replace the $\Z$-grading with a grading over a partially ordered abelian group $\Lambda$, together with suitable categories of modules that are graded by $\Lambda$-sets.

  Such a framework is motivated by several phenomena already present in VOA theory. For example, grading-restricted generalized modules naturally carry non-integral weight gradings, even though they can be repackaged as admissible modules. 
  More generally, tensor products of graded algebras, higher-dimensional vertex algebras, quantum groups and related structures frequently produce associative algebras equipped with more complicated grading data. Extending canonical seminorms to this setting would provide a broader algebraic framework for studying these representation theories.

\subsection{Beyond the strong identity condition}

  One of the main conclusions of this paper is that SIC characterizes precisely when algebraic smoothing admits a dual-pair formalism of the kind appearing in analytic sewing. Consequently, the failure of SIC in general suggests that the mode transition algebra alone does not fully account for the sewing phenomena appearing in the non-semisimple setting.

  This raises the possibility that additional homological or derived structures may be necessary in order to extend algebraic smoothing. In particular, the end formula for the mode transition algebra suggests that a refined categorical or derived version of the present framework may exist, capable of measuring the obstruction to algebraic smoothing in genuinely logarithmic settings.

\subsection{Finer topologies on enveloping algebras}

  The almost-canonical seminorm provides a topology on the enveloping algebra $\mathscr{U}(\mathbb{V})$ sufficient for constructing mode transition algebras and studying algebraic smoothing. However, the end formula obtained in this paper suggests that the present topology may not yet fully capture the dual-pair formalism underlying analytic sewing.

  This raises the possibility that $\mathscr{U}(\mathbb{V})$ admits finer topological structures more directly related to the categorical end
  \[
    \int_W W \otimes W'.
  \]
  Such structures could lead to a more intrinsic algebraic realization of dual-pair sewing and may also shed light on finiteness conditions such as $C_1$-cofiniteness and the coherence of conformal blocks.

\section*{Acknowledgements}

  The authors would like to thank Angela Gibney, Daniel Krashen, and Colton Griffin for insightful discussions.

\appendix
\section{Tensor product of categories and functors}
  The goal of this appendix is to explain the following decompositions of abelian categories:
  \begin{align}
    \Mod[Ex](U^{\tt I}\otimes U^{\tt II}|{}_R) 
    &\simeq 
    \Mod[Ex](U^{\tt I}|{}_R)\boxtimes \Mod[Ex](U^{\tt II}|{}_R),
    \label{eq:TensorOfU}\\
    \Mod(\mathsf{A}(U^{\tt I}\otimes U^{\tt II})|{}_R) 
    &\simeq
    \Mod(\mathsf{A}(U^{\tt I})|{}_R)\boxtimes\Mod(\mathsf{A}(U^{\tt II})|{}_R).
    \label{eq:TensorOfA}
  \end{align}
  And then study the functors $F \boxtimes G$ between the left-hand sides induced from functors $F$ and $G$ between the right-hand sides.

\subsection{Deligne-Kelly tensor products}
  We first recall the theory of Deligne-Kelly tensor products of locally presentable abelian categories. Our main references are \cite{TensorGrothCats,TensorDK}. See also \cite{TannakaCats,Kelly82} for the original definitions of Deligne products and Kelly products, and \cite{LPCats} for the theory of locally presentable categories.

  \begin{definition}\label{def:naive-tensor}
    Let $\mathcal{A}$ and $\mathcal{B}$ be two $\k$-linear categories. The \concept{(na\"ive) tensor product} $\nota{\mathcal{A}\otimes\mathcal{B}}$ of them is the $\k$-linear category whose objects are pairs $(A,B)$ of objects $A\in\mathcal{A}$ and $B\in\mathcal{B}$, and whose Hom-objects are given by
    \[
      \Hom_{\mathcal{A}\otimes\mathcal{B}}((a_1,b_1),(a_2,b_2)):=\Hom_{\mathcal{A}}(a_1,a_2)\otimes\Hom_{\mathcal{B}}(b_1,b_2).
    \]
  \end{definition}
  However, even when both $\mathcal{A}$ and $\mathcal{B}$ are abelian categories, the na\"ive tensor product is not an abelian category in general. 
  To fix this, Deligne has introduced tensor products of abelian categories. Unfortunately, Deligne's product does not always exist. 
  On the other hand, Kelly has also introduced another tensor product, existing for locally presentable $\k$-linear categories. It is shown in \cite{TensorDK} that the two products coincide if the Deligne product exists. 
  For our purpose, we will use a mixed notion of them, which takes two locally presentable $\k$-linear \emph{abelian} categories and produces a new one. According to \cite{TensorGrothCats}, this is indeed given by Kelly's theory, with the collaboration of a generalized variant of Deligne products.

  Before moving on, let us recall some notions from category theory. 
  \begin{definition}
    A $\k$-linear category $\mathcal{C}$ is called \emph{locally presentable} if it is cocomplete and there is a regular cardinal $\uplambda$ such that $\mathcal{C}$ has a set of strong generators consisting of \emph{$\uplambda$-presented objects}, i.e., objects $c\in\mathcal{C}$ such that the $\k$-linear functor $\Hom_{\mathcal{C}}(c,-)\colon\mathcal{C}\to\mathbb{K}$ preserves $\uplambda$-filtered colimits. We will say that $\mathcal{C}$ is \emph{locally $\uplambda$-presentable} if we want to emphasize the regular cardinal $\uplambda$.
    If this is the case, then the full subcategory $\mathcal{C}_{\uplambda'}$ of $\uplambda'$-presented objects, where $\uplambda'$ is any regular cardinal above $\uplambda$, is essentially small and \emph{$\uplambda'$-cocomplete} (i.e. closed under taking $\uplambda'$-small colimits). Then, the entire category $\mathcal{C}$ can be recovered from $\mathcal{C}_{\uplambda'}$ by taking $\uplambda'$-filtered colimits of objects in $\mathcal{C}_{\uplambda'}$.  
  \end{definition}

  \begin{recollection}[{cf. \cite[\S 5]{TensorGrothCats}}]
    Let $\mathcal{A}$ and $\mathcal{B}$ be two locally presentable $\k$-linear categories. Then, there is another locally presentable $\k$-linear category $\nota{\mathcal{A}\boxtimes\mathcal{B}}$ and a $\k$-bilinear functor $\boxtimes\colon\mathcal{A}\otimes\mathcal{B}\to\mathcal{A}\boxtimes\mathcal{B}$ that is cocontinuous in each variable and induces equivalences (for any $\k$-linear category $\mathcal{C}$)
    \[
      \cat{CoCont}(\mathcal{A}\boxtimes\mathcal{B},\mathcal{C}) \simeq 
      \cat{CoCont}(\mathcal{A},\mathcal{B};\mathcal{C}),
    \]
    where the left-hand side is the category of cocontinuous $\k$-linear functors from $\mathcal{A}\boxtimes\mathcal{B}$ to $\mathcal{C}$, and the right-hand side is the category of $\k$-bilinear functors from $\mathcal{A}\otimes\mathcal{B}$ to $\mathcal{C}$ that are cocontinuous in each variable.
    Furthermore, if both $\mathcal{A}$ and $\mathcal{B}$ are abelian, then $\mathcal{A}\boxtimes\mathcal{B}$ is an abelian category; and if both $\mathcal{A}$ and $\mathcal{B}$ are Grothendieck categories, then $\mathcal{A}\boxtimes\mathcal{B}$ is a Grothendieck category.
    
    The category $\mathcal{A}\boxtimes\mathcal{B}$ is called the \concept{Deligne-Kelly tensor product} of $\mathcal{A}$ and $\mathcal{B}$.
  \end{recollection}

  \begin{example}\label{eg:tensor-of-Mods}
    Let $A$ and $B$ be two $\k$-rings, not necessarily commutative. 
    Then, the Deligne-Kelly tensor product $\Mod(A)\boxtimes\Mod(B)$ of their categories of left modules is equivalent to the category $\Mod(A\otimes B)$ of all left $A\otimes B$-modules. This is indeed due to the definition of the \emph{tensor product of Grothendieck categories} in the paper \cite{TensorGrothCats}.
    In particular, by \zcref{coro:A-tensor}, this explains the decomposition $\Mod(\mathsf{A}(U^{\tt I}\otimes U^{\tt II})|{}_R) 
    \simeq
    \Mod(\mathsf{A}(U^{\tt I})|{}_R)\boxtimes\Mod(\mathsf{A}(U^{\tt II})|{}_R)$.
  \end{example}

  The following lemmas will be used later:
  \begin{lemma}\label{lem:TensorOfHom}
    Let $\mathcal{A}$ and $\mathcal{B}$ be two locally $\uplambda$-presentable $\k$-linear abelian categories. Then, for any $\uplambda$-presented object $a_1,a_2\in\mathcal{A}$ and arbitrary $b_1,b_2\in\mathcal{B}$, we have
    \[
      \Hom_{\mathcal{A}\boxtimes\mathcal{B}}(a_1\boxtimes b_1,a_2\boxtimes b_2):=\Hom_{\mathcal{A}}(a_1,a_2)\otimes\Hom_{\mathcal{B}}(b_1,b_2).
    \]
  \end{lemma}
  \begin{proof}
    The category $\mathcal{A}\boxtimes\mathcal{B}$ can be constructed as the category of all $\k$-bilinear functors from $\mathcal{A}_{\uplambda}\otimes\mathcal{B}_{\uplambda}$ to $\mathbb{K}$ that are $\uplambda$-continuous, namely preserving all $\uplambda$-small colimits, in each variable. See, for instance, \cite[Proposition 5.2]{TensorGrothCats}.
    From this construction, we see that the restriction of $\boxtimes$ to $\mathcal{A}_{\uplambda}\otimes\mathcal{B}_{\uplambda}$ is an embedding (via the Yoneda embedding). Hence, the statement follows from \zcref{def:naive-tensor}.
  \end{proof}

\subsection{Tensor product of functors}
  \begin{definition}
    Let $F\colon \mathcal{A}\to\mathcal{A}'$ and $G\colon \mathcal{B}\to\mathcal{B}'$ be two cocontinuous $\k$-linear functors between locally presentable $\k$-linear abelian categories. Then, by the definition of the Deligne-Kelly tensor product, we have a cocontinuous $\k$-linear functor $F\boxtimes G$ via the composition:
    \[
      \cat{CoCont}(\mathcal{A},\mathcal{A}')\otimes\cat{CoCont}(\mathcal{B},\mathcal{B}')\overset{}{\longrightarrow}
      \cat{CoCont}(\mathcal{A},\mathcal{B};\mathcal{A}'\boxtimes\mathcal{B}') \simeq \cat{CoCont}(\mathcal{A}\boxtimes\mathcal{B},\mathcal{A}'\boxtimes\mathcal{B}'),
    \]
    where the first one sends each pair $(F,G)$ to the $\k$-bilinear functor $(a,b)\mapsto (F(a),G(b))$. 
  \end{definition}
  \begin{lemma}
      If $\k$ is semisimple, then the universal $\k$-bilinear functor $\boxtimes\colon\mathcal{A}\otimes\mathcal{B}\to\mathcal{A}\boxtimes\mathcal{B}$ is exact in each variable.
    \end{lemma}
    \begin{proof}
      Suppose both $\mathcal{A}$ and $\mathcal{B}$ are locally $\uplambda$-presentable. Then, by the construction of $\mathcal{A}\boxtimes\mathcal{B}$, to show that $\boxtimes\colon\mathcal{A}\otimes\mathcal{B}\to\mathcal{A}\boxtimes\mathcal{B}$ is exact in each variable, it suffices to do so for its restriction to $\mathcal{A}_{\uplambda}\otimes\mathcal{B}_{\uplambda}$. This clearly follows from \zcref{lem:TensorOfHom} since when $\k$ is semisimple, any $\k$-module is projective.
    \end{proof}

\begin{theorem}\label{thm:tensor-of-equivalence}
  Let $F\colon \mathcal{A}\to\mathcal{A}'$ and $G\colon \mathcal{B}\to\mathcal{B}'$ be two cocontinuous $\k$-linear functors between locally presentable $\k$-linear abelian categories and assuming $\mathcal{A}$ and $\mathcal{B}$ are nonzero Grothendieck abelian categories. If $\k$ is a field, then, $F\boxtimes G$ is an equivalence if and only if both $F$ and $G$ are equivalences.
\end{theorem}
\begin{proof}
  Suppose both $F$ and $G$ are equivalences. Then, $F\boxtimes G$ is an equivalence by the universal property of the Deligne-Kelly tensor product. This direction do not need $\k$ being a field.

  Conversely, assume that $F\boxtimes G$ is an equivalence. 
  Since $F$ and $G$ are cocontinuous additive functors from Grothendieck abelian categories, by the \emph{special adjoint functor theorem}, they must have right adjoints $F^{\dagger}$ and $G^{\dagger}$ respectively. Let $\eta^F \colon \id_{\mathcal{A}} \to F^{\dagger}F, \eta^G \colon \id_{\mathcal{B}} \to G^{\dagger}G$ denote the units 
  of the adjunctions. 
  Similarly, $F\boxtimes G$ has a right adjoint $(F\boxtimes G)^{\dagger}$ with unit $\eta$.
  On \emph{presented} objects, the unite $\eta$ can be explicitly described by the following commutative diagram:
  \[
  \begin{tikzcd}
    \id_{\mathcal{A}}\boxtimes\id_{\mathcal{B}} 
        \ar[d,"{\eta^F\boxtimes \id_{\mathcal{B}}}"']
        \ar[r,"{\id_{\mathcal{A}}\boxtimes \eta^G}"] &
    \id_{\mathcal{A}}\boxtimes G^{\dagger}G 
        \ar[d,"{\eta^F\boxtimes \id_{G^{\dagger}G}}"] \\
    F^{\dagger}F\boxtimes\id_{\mathcal{B}} 
        \ar[r,"{\id_{F^{\dagger}F}\boxtimes \eta^G}"'] &
    F^{\dagger}F\boxtimes G^{\dagger}G 
  \end{tikzcd}
  \]
  Applying it to any \emph{presented} objects $a \in \mathcal{A}$ and $b \in \mathcal{B}$, we obtain the following commutative diagrams of $\cat{Vect}_{\k}$-valued functors on \emph{presented} objects:
  \[
  \begin{tikzcd}
    \Hom_{\mathcal{A}}(-,a)\otimes_{\k}\Hom_{\mathcal{B}}(-,b)
        \ar[d,"(\eta^F_{a})_{\ast} \otimes \id"']
        \ar[r,"\id \otimes (\eta^G_{b})_{\ast}"] &
    \Hom_{\mathcal{A}}(-,a)\otimes_{\k}\Hom_{\mathcal{B}}(-,G^{\dagger}G(b)) 
        \ar[d,"(\eta^F_{a})_{\ast} \otimes \id"] \\
    \Hom_{\mathcal{A}}(-,F^{\dagger}F(a))\otimes_{\k}\Hom_{\mathcal{B}}(-,b)
        \ar[r," \id \otimes (\eta^G_{b})_{\ast}"] &
    \Hom_{\mathcal{A}}(-,F^{\dagger}F(a))\otimes_{\k}\Hom_{\mathcal{B}}(-,G^{\dagger}G(b)) 
  \end{tikzcd}
  \]
  and the compositions are isomorphisms.
  Since we have assumed that neither of $\mathcal{A},\mathcal{A}',\mathcal{B},\mathcal{B}'$ are zero, 
  the above morphisms are nonzero and we conclude that $\eta^F_{a}$ and $\eta^G_{b}$ are both isomorphisms.
  Since represented objects generate the categories, this finishes the proof.
\end{proof}

%%-----------------------------------------

\bibliographystyle{amsalpha}
\bibliography{bib}

%%-----------------------------------------
\end{document}